\documentclass[a4paper, 12pt]{article}
\usepackage{graphicx}
\usepackage{setspace}
\usepackage{psfrag}
\usepackage[utf8]{luainputenc}
\usepackage{amsmath}
\usepackage{amssymb}
\usepackage{amsthm}
\usepackage{color}
\usepackage{mathtools, nccmath}
\usepackage{afterpage}
\usepackage{subfig}
\usepackage{esint}
\usepackage{diagbox}
\usepackage{enumitem}
 
\usepackage{authblk}
\usepackage{geometry}

\newtheorem{theorem}{Theorem}[section]
\newtheorem{lemma}[theorem]{Lemma}
\newtheorem{remark}[theorem]{Remark}

\DeclareMathOperator{\acosh}{acosh}

\usepackage{hyperref}
\hypersetup{
    colorlinks,
    bookmarksopen,
    bookmarksnumbered,
    citecolor=cyan,
    filecolor=cyan,
    linkcolor=blue,
    urlcolor=green
}

\usepackage{xcolor}
\colorlet{linkequation}{red}
\newcommand*{\SavedEqref}{}
\let\SavedEqref\eqref
\renewcommand*{\eqref}[1]{%
  \begingroup
    \hypersetup{
      linkcolor=linkequation,
      linkbordercolor=linkequation,
    }%
    \SavedEqref{#1}%
  \endgroup
}

\title{Convergence of substructuring Methods for the Cahn-Hilliard Equation}

\author{Gobinda Garai\thanks{School of Basic Sciences, IIT Bhubaneswar, India, gg14@iitbbs.ac.in} \and Bankim C. Mandal\thanks{School of Basic Sciences, IIT Bhubaneswar, India, bmandal@iitbbs.ac.in}}

\date{}
\begin{document}

\maketitle
\setcounter{page}{1}
\setstretch{.63}
\begin{abstract}
In this paper, we formulate and study substructuring type algorithm for the Cahn-Hilliard (CH) equation, which was originally proposed to describe the phase separation phenomenon for binary melted alloy below the critical temperature and since then it has appeared in many fields ranging from tumour growth simulation, image processing, thin liquid films, population dynamics etc. Being a non-linear equation, it is important to develop robust numerical techniques to solve the CH equation. Here we present the formulation of Dirichlet-Neumann (DN) and Neumann-Neumann (NN) methods applied to CH equation and study their convergence behaviour. We consider the domain-decomposition based DN and NN methods in one and two space dimension for two subdomains and extend the study for multi-subdomain setting for NN method. We verify our findings with numerical results.\\[.12cm]
\end{abstract}

\section{Introduction} \label{intro}

The Cahn-Hilliard equation has been suggested as a prototype to represent the evolution of a binary melted alloy below the critical temperature in \cite{Cahn, Hilliard}. Initially both components of binary alloy uniformly present in the system, then they go through rapid cooling below the critical temperature. As a result the homogeneous state becomes unstable and phase separation occurs. The phase separation is a process when a homogeneous mixture of two components $A$ and $B$ in one thermodynamic phase suddenly separates into regions consisting of two different phases. We can describe the phase separation by defining two components A and B with concentrations and making the following assumptions:
\begin{itemize}
\item the domain is filled with binary melted alloy: $A$ and $B$ particles with concentrations $m_1$ and $m_2$ respectively.
\item diffusion is the only form of transportation.
\item AA and BB interactions are favorable.
\item AB and BA interactions are unfavorable.
\end{itemize}
Then the phase function or the concentration of species can be defined in either of two ways
$$u =\frac{m_1 - m_2}{m_1 + m_2},\mbox{\hspace{1cm}}
 c=\frac{m_1}{m_1 + m_2}.$$
It follows that $-1 \leq u \leq 1$ and $0 \leq c \leq 1.$ Here we denote concentration as $u$. With these assumptions we define the Ginzburg-Landau free energy functional: 
$$ E(u):=\int_{\Omega}\left(f(u)+\frac{\epsilon^2}{2}\vert\nabla u\vert^2\right)d\boldsymbol{x}, $$
where $f(u)$ has the primitive $F(u)$, which is generally taken as $F(u) = \frac{1}{4}(u^2 - 1)^2$ and $\frac{\epsilon^2}{2}\vert\nabla u\vert^2$ is gradient energy and $\epsilon$ is thickness of the interface. The first variation $\frac{\delta E(u)}{\delta u}$ would quantify how the energy changes when the particle concentrations change. This variation is known as the chemical potential of the system and given by
\begin{equation}\label{chemical potential}
\frac{\delta E(u)}{\delta u} = f(u) - \epsilon^2\Delta u = \mu. 
\end{equation}
Using Fick’s 1st diffusion law, $J = -M(u)\nabla\mu,$
where $M(u)$ being the mobility. We have a mass conservation constraint, so by continuity equation we have the following:
\begin{equation}\label{conservation}
\frac{\partial u}{\partial t} + \nabla . J = 0.
\end{equation}
For constant mobility $M(u)$ and using \eqref{chemical potential} in \eqref{conservation}, CH equation takes the form
\begin{equation}\label{CH}
\frac{\partial u}{\partial t} = \Delta(f(u) - \epsilon^2\Delta u), 
\end{equation}
with the natural boundary conditions for all $t\in(0,T]$
\begin{equation}\label{BC}
\frac{\partial u}{\partial n} = 0\,\ \mbox{ on} \ \ \partial\Omega,\mbox{\hspace{1cm}}
 \frac{\partial \mu}{\partial n} = 0\,\ \mbox{ on}  \ \ \partial\Omega. \end{equation}
The initial-boundary value problem for a closed system is then to solve \eqref{CH} subject to the boundary condition \eqref{BC} and the initial condition $u(x, 0) = u_0(x),\; x\in \Omega.$ We introduce the chemical potential $v:$
\begin{equation*}
v := f(u) - \epsilon^2\Delta u,
\end{equation*}
to reformulate the CH equation as the following mixed form,
\begin{equation}\label{mixedCH}
\begin{aligned}
\frac{\partial u}{\partial t} & = \Delta v, \,\ \mbox{ for} \,\ (x,t)\in\Omega\times(0, T],\\
v & = f(u) - \epsilon^2\Delta u, \,\ \mbox{ for} \,\ (x,t)\in\Omega\times(0, T],\\
u(x, 0) & = u_0(x), \,\ x\in\Omega,
\end{aligned}
\end{equation}
with the Neumann boundary condition for all $t\in(0, T]$
\[
\frac{\partial u}{\partial n} = \frac{\partial v}{\partial n} = 0 \mbox{ on} \,\ \partial\Omega. 
\]
By differentiating the energy functional $E(u)$ and total mass $\int_{\Omega}u$ with respect to time $t$, we get 
\begin{equation}\label{energy minimization}
\frac{d}{dt}E(u)\leq 0, \mbox{\hspace{1cm}} \frac{d}{dt}\int_{\Omega}u = 0.
\end{equation}
Meaning, CH equation describes energy minimization and the total mass conservation while the system evolve.

The CH equation is a nonlinear equation and it is impossible to find its analytical solution. However the existence of solution is guaranteed in \cite{ElliottZheng, liuzhao}. Extensive studies have been carried out to find numerical schemes for the CH equation to approximate the solution with either Dirichlet \cite{DuNicolaides, David} or Neumann boundary conditions\cite{elliott1987numerical, furihata2001stable, shin2011conservative, ElliottZheng, Elliott, StuartHumphries, Christlieb, ChengFeng}. A review on numerical treatment to the CH equation can be found in \cite{lee2014physical}. The possible applications of CH equation as a model are: image inpainting \cite{EsedoAndrea}, tumour growth simulation \cite{Wise}, population dynamics \cite{cohen1981generalized}, dendritic growth \cite{kim1999universal}, planet formation \cite{tremaine2003origin} etc.

Since the non-increasing property \eqref{energy minimization} of the total energy is an essential feature of the CH equation, it is a key issue for long time simulation that is expected to be preserved by numerical techniques. The solution of CH equation involves two different dynamics, one is phase separation  which is quick in time, and another is phase coarsening which is slow in time. The fine-scale phase regions are formed during the early stage and they are separated by the interface which is of width $\epsilon.$ Whereas during phase coarsening, the solution tends to an equilibrium state which minimizes
the system energy.  In order to deal with the property \eqref{energy minimization}, Eyre \cite{David, Eyre} proposed an unconditionally gradient stable scheme. The idea is to split the homogeneous free energy $F(u)$ into a sum of a convex term and a concave term, and then treating the convex term implicitly and the concave term explicitly to obtain, for example, a first order in time and 2nd order in space approximation for the one dimensional CH equation, as described below:
\begin{equation}\label{1stchdis}
\begin{aligned}
u_j^{n+1} & = \delta_t\Delta _hv_j^{n+1} + u_j^{n},\\
v_j^{n+1} & = (u_j^{n+1})^3 -u_j^{n} - \epsilon^2\Delta _hu_j^{n+1}, 
\end{aligned}
\end{equation}
where $\delta_t$ is time step and $\Delta_h$ is discrete Laplacian.
The equation \eqref{1stchdis} represents a large set of nonlinear coupled equation due to the cubic term. To linearise the problem, the term $(u_j^{n+1})^3$ is disintegrated as $(u_j^{n})^2u_j^{n+1}$; we rewrite the resulting equation as:
\begin{equation}\label{linearCH}
\begin{aligned}
u_j^{n+1} - \delta_t\Delta _hv_j^{n+1} & = u_j^{n},\\
v_j^{n+1} + \epsilon^2\Delta _hu_j^{n+1} - (u_j^{n})^2u_j^{n+1} & = -u_j^{n}, 
\end{aligned}
\end{equation}
which is also an unconditionally gradient stable scheme and has the same accuracy as the nonlinear scheme \eqref{1stchdis}\cite{David, vollmayr2003fast}. So, at each time level one has to solve the following system of elliptic equations to get the solution of \eqref{mixedCH}
\begin{equation}\label{EE}
\begin{aligned}
\bar{u} - \delta_t\Delta \bar{v} & = f_{\bar{u}},\\
\bar{v} + \epsilon^2\Delta\bar{u} - c^2\bar{u} & = f_{\bar{v}}, 
\end{aligned}
\end{equation}
where $c = u_j^{n}, \bar{u} = u_j^{n+1}, \bar{v} = v_j^{n+1}, f_{\bar{u}} = c \,\ \text{and}\,\ f_{\bar{v}} = -c.$ The above system \eqref{EE} can be reformulated as the following,
\begin{equation}\label{MF1}
\begin{aligned}
\begin{bmatrix} 
I & -\delta_t\Delta \\
\epsilon^2\Delta-c^2 & I 
\end{bmatrix}
\quad
\begin{bmatrix} 
{\bar{u}} \\
{\bar{v}} 
\end{bmatrix}
\quad 
 = \quad \begin{bmatrix} 
f_{\bar{u}} \\
f_{\bar{v}} 
\end{bmatrix},
\quad \mbox{in}\,\ \Omega.
\end{aligned}
\end{equation}
In 2D, one also gets the above system \eqref{MF1} at each time level for suitably chosen $f_{\bar{u}}, f_{\bar{v}}, c$. It is worth mentioning that many other basic algorithms approximating the solution of CH equation can be reformulated as \eqref{MF1}, for example the semi-implicit Euler's scheme\cite{KimLee}, the LSS scheme \cite{EsedoAndrea}. In addition, the linearisation of a non-linearly stabilized splitting scheme \cite{Christlieb} would also lead to the form \eqref{MF1}.

Since the spatial mesh size $h$ is $O(\epsilon)$ or even finer, the linear equation will result in a very large scale algebraic system that should be solved sequentially for simulating the long term behaviour of CH equation.
Consequently, it is of great importance to accelerate the simulation using parallel computation, which can be achieved by domain decomposition techniques \cite{Picard, LionsI, LionsII}. In this work, we lay our efforts on the Dirichlet-Neumann and Neumann-Neumann methods. These algorithms were first considered by Bjørstad \& Widlund \cite{BjWid} and Bourgat et al. \cite{bourgat1988variational}; also see \cite{QuarVal, TosWid} and to see some recent work on Neumann-Neumann method we refer to \cite{ scalability, nilpotent, local, continuous}.
The main objective of our work is to solve the problem \eqref{MF1} with the imposed transmission condition and analyse the convergence  behaviour for two as well as multisubdomain setting in 1D and 2D.

We introduce the DN algorithm in one and two spatial dimension for two subdomains, and study the convergence result in Section \ref{Section2} . In section \ref{Section3} we present the NN algorithm for  multiple subdomain, and analyze the convergence behaviour. To illustrate our analysis, the accuracy and robustness of the proposed techniques, we show numerical results in Section \ref{Section4}.

\section{The Dirichlet-Neumann Method}\label{Section2}
In this section, we introduce the Dirichlet-Neumann method for the second order elliptic system \eqref{MF1}. For convenience we use the notation $u, v$ instead of $\bar{u}, \bar{v}$ and rewrite the system \eqref{MF1} as
\begin{equation}\label{modelproblem}
\begin{aligned}
\begin{bmatrix} 
I & -\delta_t\Delta \\
\epsilon^2\Delta-c^2 & I 
\end{bmatrix}
\begin{bmatrix} 
{u} \\
{v} 
\end{bmatrix}
 = \begin{bmatrix} 
f_{{u}} \\
f_{{v}} 
\end{bmatrix},
\quad \mbox{in}\,\ \Omega,
\end{aligned}
\end{equation}
together with the Neumann boundary condition $\mathcal{B}\begin{bmatrix} 
u \\
v 
\end{bmatrix}
=0$ along $\partial\Omega$.

Suppose the spatial domain $\Omega$ is partitioned into two non-overlapping subdomains $\Omega_{1}, \Omega_{2}$. We denote 
$u_{i}$ the restriction of the solution $u$ of \eqref{modelproblem} to $\Omega_{i}$ for $i=1, 2$ and set $\Gamma:=\partial\Omega_{1}\cap\partial\Omega_{2}$. The DN algorithm starts with initial guesses $g^{[0]}, h^{[0]}$ along the interface $\Gamma$ and solve for $k = 1, 2, \cdots$
{\footnotesize
\begin{equation}
\begin{aligned}\label{DNCH}
&\left\{
\begin{aligned}
\begin{bmatrix} 
I & -\delta_t\Delta \\
\epsilon^2\Delta-c^2 & I 
\end{bmatrix}
\begin{bmatrix} 
u_1^{[k]} \\
v_1^{[k]} 
\end{bmatrix}
  = \begin{bmatrix} 
f_u \\
f_v 
\end{bmatrix},\,\ \text{in} \,\ \Omega_1, \\
\mathcal{B}\begin{bmatrix} 
u_1^{[k]} \\
v_1^{[k]} 
\end{bmatrix}
 = 0,\,\ \mbox{on}\,\ \partial\Omega_1\cap\partial\Omega,\\
\begin{bmatrix} 
u_1^{[k]} \\
v_1^{[k]} 
\end{bmatrix}
 = 
\begin{bmatrix} 
g^{[k-1]} \\
h^{[k-1]} 
\end{bmatrix},
\,\ \mbox{on}\,\ \Gamma,\\
\end{aligned}\right.
&
\left\{\begin{aligned}
\begin{bmatrix} 
I & -\delta_t\Delta \\
\epsilon^2\Delta-c^2 & I 
\end{bmatrix}
\begin{bmatrix} 
u_2^{[k]} \\
v_2^{[k]} 
\end{bmatrix}
  = \begin{bmatrix} 
f_u \\
f_v 
\end{bmatrix}, \,\ \text{in} \,\ \Omega_2, \\
\mathcal{B}\begin{bmatrix} 
u_2^{[k]} \\
v_2^{[k]} 
\end{bmatrix}
 = 0,\,\ \mbox{on}\,\ \partial\Omega_2\cap\partial\Omega,\\
\frac{\partial}{\partial n_2}\begin{bmatrix} 
u_2^{[k]} \\
v_2^{[k]} 
\end{bmatrix}
 = 
-\frac{\partial}{\partial n_1}\begin{bmatrix} 
u_1^{[k]} \\
v_1^{[k]} 
\end{bmatrix},
\,\ \mbox{on}\,\ \Gamma.\\
\end{aligned}\right. \\
\end{aligned}
\end{equation}
\par}
Then we update the interface trace by
\[
\begin{bmatrix} 
g^{[k]} \\
h^{[k]} 
\end{bmatrix}
 = \theta \begin{bmatrix} 
u_2^{[k]} \\
v_2^{[k]} 
\end{bmatrix}_{\big|_{\Gamma}} + 
(1-\theta)\begin{bmatrix} 
g^{[k-1]} \\
h^{[k-1]} 
\end{bmatrix},
\]
where $\theta\in(0, 1)$ is a relaxation parameter. We now consider the error equation, corresponding to the DN algorithm \eqref{DNCH} for further analysis, so that $f_u=0=f_v$.
The ultimate goal of our analysis is to study how the error $g^{[k]}, h^{[k]}$ converges to zero as $k\rightarrow\infty$.
\subsection{Convergence analysis in 1D}
To determine the convergence behaviour of the algorithm \eqref{DNCH} in one spatial dimension, let $\Omega = (-a, b)$ is decomposed into $\Omega_1 = (-a, 0)$ and $\Omega_2 = (0, b)$ with interface $\Gamma=\{0\}$. For $k = 1, 2, \cdots$, we solve
\begin{flalign*}
  & \begin{aligned} & \begin{cases}
  AE_1^{[k]} & = 0, \quad \quad \text{in} \,\ \Omega_1 \\
E_1^{[k]} & = 
\begin{bmatrix} 
g^{[k-1]} \\
h^{[k-1]} 
\end{bmatrix}, \quad \text{on} \,\ \Gamma \\
\frac{\partial}{\partial x}E_1^{[k]} & = 0, \quad \quad \text{on} \,\ \partial\Omega_1 \backslash \Gamma\\
  \end{cases}\\
  \MoveEqLeft[3]\text{}
  \end{aligned}
  & &
  \begin{aligned}
      & \begin{cases}
    AE_2^{[k]} & = 0,  \quad \quad \text{in} \,\ \Omega_2\vspace{.2cm}\\
\frac{\partial}{\partial x}E_2^{[k]} & = \frac{\partial}{\partial x}E_1^{[k]},\quad \text{on} \,\ \Gamma\vspace{.2cm} \\
\frac{\partial}{\partial x}E_2^{[k]} & = 0, \quad \quad \text{on} \,\ \partial\Omega_2\backslash \Gamma \\
    \end{cases} \\
    \MoveEqLeft[-3]\text{}
  \end{aligned}
\end{flalign*}
 and then update the interface trace by
\begin{equation}\label{interfacetrace1d}
\begin{bmatrix} 
g^{[k]} \\
h^{[k]} 
\end{bmatrix}
 = \theta {E_2^{[k]}}_{\big|_{\Gamma}} + 
(1-\theta)\begin{bmatrix} 
g^{[k-1]} \\
h^{[k-1]} 
\end{bmatrix},
\end{equation}
where
\begin{equation*}
A = 
\begin{bmatrix} 
1 & -\delta_t\frac{d^2}{dx^2} \\
\epsilon^2\frac{d^2}{dx^2}-c^2 & 1 
\end{bmatrix},\quad
E_j^{[k]} = \begin{bmatrix} 
u_j^{[k]}(x) \\
v_j^{[k]}(x)
\end{bmatrix}.
\end{equation*}
We solve the subdomain problems by solving the following algebraic equations
\begin{equation}\label{Algebraiceq}
AE_j^{[k]} =0, j=1,2
\end{equation}
where $E_j^{[k]}$ is the subdomain solution in $\Omega_j$ for $j=1,2$ at $k-$th iteration. We assume the solution of the equation \eqref{Algebraiceq} for every iteration $k$ is of the following form,
\begin{equation}
E_j=\Psi_j e^{\xi x}; \xi\; \text{being a parameter to be determined.}
\end{equation}
Inserting the above form of $E_j$ into the equation \eqref{Algebraiceq} gives the following,
\begin{equation}\label{coefficientmatrix}
\begin{bmatrix} 
1 & -\delta_t\xi^2\\
\epsilon^2\xi^2-c^2 & 1 
\end{bmatrix}
\Psi_j = 0, \;\text{since the exponential term never vanish.}
\end{equation}
The equation \eqref{coefficientmatrix} has non-trivial solutions only if the coefficient matrix is singular, i.e the determinant of coefficient matrix is zero, 
\begin{equation*}
\det
\begin{bmatrix} 
1 & -\delta_t\xi^2\\
\epsilon^2\xi^2-c^2 & 1 
\end{bmatrix}
 = 0.
\end{equation*}
Solving this equation yields
\begin{equation}\label{xi13}
\xi_{1,2} = \pm\sqrt{\lambda_1},\quad \xi_{3,4} = \pm\sqrt{\lambda_2},
\end{equation}
where $\lambda_{1,2}$ are given by 
\begin{equation*}
\lambda_{1,2} = \frac{c^2\delta_t\pm \sqrt{c^4\delta_t^2-4\epsilon^2\delta_t}}{2\epsilon^2\delta_t}
\end{equation*}
respectively. Thus the solution of \eqref{Algebraiceq} has the following form 
\begin{equation*}\label{generalsol}
E_j^{[k]} = \sum_{l=1} ^4 \zeta_{j,l}^{[k]}\mu_l e^{\xi_l x}, 
\end{equation*}
where for each $l$, $\zeta_{j,l}$ are constant for $j=1,2$ and $\mu_l$ is an eigenvector of the coefficient matrix in  \eqref{coefficientmatrix} associated to the eigenvalue zero for $\xi = \xi_l$, and is explicitly given by 
\begin{equation*}
\mu_{1} = \mu_{2} = \begin{bmatrix} 
\delta_t\lambda_1 \\
1 
\end{bmatrix},\quad \mu_{3} = \mu_{4} = \begin{bmatrix} 
\delta_t\lambda_2\\
1 
\end{bmatrix}.
\end{equation*}
Using the transmission conditions on the interface $\Gamma$ and the physical boundary conditions on $\partial\Omega_j\backslash\Gamma$, we find the constants $\zeta_{j,l}$ for each subdomain. We have the subdomain solution at $k-$th iteration for Dirichlet and Neumann step respectively,
\[
\begin{array}{cc}
\begin{aligned}
E_1^{[k]} & = 
\begin{bmatrix} 
\mu_1 & \mu_3
\end{bmatrix} 
\begin{bmatrix} 
\gamma_{1,x} + \frac{\sigma_{1,a} \sigma_{1,x}}{\gamma_{1,a}} & 0 \\
0 &  \gamma_{3,x} + \frac{\sigma_{3,a} \sigma_{3,x}}{\gamma_{3,a}}
\end{bmatrix}
\begin{bmatrix} 
\eta_1 \\
\eta_2 
\end{bmatrix},\\
E_2^{[k]} & = 
\begin{bmatrix} 
\mu_1 & \mu_3
\end{bmatrix} 
\begin{bmatrix} 
\frac{\sigma_{1,x}\sigma_{1,a}}{\gamma_{1,a}}  - \frac{\gamma_{1,x}\gamma_{1,b} \sigma_{1,a}}{\gamma_{1,a}\sigma_{1,b}} & 0 \\
0 & \frac{\sigma_{3,x}\sigma_{3,a}}{\gamma_{3,a}}  - \frac{\gamma_{3,x}\gamma_{3,b} \sigma_{3,a}}{\gamma_{3,a}\sigma_{3,b}}
\end{bmatrix}
\begin{bmatrix} 
\eta_1 \\
\eta_2 
\end{bmatrix},
\end{aligned}
\end{array}
\]
where $\sigma_{i,x}:=\sinh(\xi_i x), \gamma_{i,x}:=\cosh(\xi_i x)$,  for $i=1,3$ and 
$\eta_1 := \frac{g^{[k-1]}-\delta_t\lambda_2 h^{[k-1]}}{\delta_t\lambda}, \eta_2 := \frac{g^{[k-1]}-\delta_t\lambda_1 h^{[k-1]}}{-\delta_t\lambda}$ with
$\lambda = \lambda_1 - \lambda_2$.  
Now inserting the subdomain solutions into the updating condition \eqref{interfacetrace1d} we get 
\begin{equation}\label{DNupdate1d}
\begin{bmatrix} 
g^{[k]} \\
h^{[k]} 
\end{bmatrix} 
=  \mathbb{H}
\begin{bmatrix} 
g^{[k-1]} \\
h^{[k-1]} 
\end{bmatrix},
\end{equation}
where $\mathbb{H}$ is the iteration matrix given by   
\begin{equation*}
\mathbb{H} = 
\begin{bmatrix} 
1-\theta+\theta\frac{-\lambda_1\rho_1 + \lambda_2\rho_2}{\lambda} & \theta\frac{\delta_t\lambda_1\lambda_2(\rho_1 - \rho_2)}{\lambda} \\
\\
\theta\frac{\rho_2 - \rho_1}{\delta_t\lambda} & 1-\theta+\theta\frac{\lambda_2\rho_1 - \lambda_1\rho_2}{\lambda}
\end{bmatrix} 
\end{equation*}
with \begin{equation*}
\rho_1 = \frac{\sigma_{1,a}\gamma_{1,b}}{\gamma_{1,a}\sigma_{1,b}},\; \rho_2 = \frac{\sigma_{3,a}\gamma_{3,b}}{\gamma_{3,a}\sigma_{3,b}}.
\end{equation*}

\begin{theorem}[Convergence of DN for $a=b$]
When the subdomains are of the same size, $a=b$, the DN algorithm converges linearly for $0<\theta<1, \theta\neq 1/2.$ For $\theta = 1/2$, it converges in two iterations.
\end{theorem}
\begin{proof}
If $a=b$, then we get from \eqref{DNupdate1d} 
\begin{equation*}
\begin{bmatrix} 
g^{[k]} \\
h^{[k]} 
\end{bmatrix} 
=  \begin{bmatrix} 
1-2\theta & 0 \\
0 & 1-2\theta 
\end{bmatrix}
\begin{bmatrix} 
g^{[k-1]} \\
h^{[k-1]} 
\end{bmatrix}.
\end{equation*}
Thus the convergence is linear for $0<\theta<1, \theta\neq 1/2$. For $\theta=1/2$, the method converges to the exact solution in two iterations.
\end{proof}
 
From the numerical results of Table \ref{dnequaltable_1d}, it is clear that $\theta =1/2$ is the optimal relaxation parameter for equal subdomain setting even irrespective of time step size. This motivates to study the convergence results for unequal sudomain with the parameter value of $\theta$ being $1/2$. We consider two cases: $a<b$, which means that the Neumann subdomain is bigger than the Dirichlet subdomain, and $a>b$, when the Dirchlet subdomain is bigger than Neumann subdomain. Before going to the main result, we prove the following Lemma, which is needed to study the convergence results.
\begin{lemma}\label{DNlemma}
The function $f(t)=\frac{\sinh (at)}{\sinh (bt)}$ with $0 < a < b$ has the following properties
\begin{enumerate}[label=(\roman*)]
\item $\forall t > 0, 0 < f(t) < \frac{a}{b}$.\label{item:1}
\item $\forall t > 0, f$ is a monotonically decreasing function.\label{item:2}
\end{enumerate}
\end{lemma}
\begin{proof}
Clearly, $f(t)>0, \forall t > 0$. We prove part \ref{item:2} first, that naturally leads to part \ref{item:1} as we take the limit $t\rightarrow 0^+$ . Taking the derivative of logarithm of $f(t)$ with respect to $t$ we get
$$(\log(f(t)))'= \frac{f'(t)}{f(t)}=a\coth(at)-b\coth(bt).$$
Now $f'(t)<0$ iff $a\coth(at)<b\coth(bt)$, which holds true since $\coth(t)$ is positive and decreasing for $t > 0$ and $0<a<b$. This completes the result.
\end{proof}

\begin{theorem}[Convergence of DN for $b>a$]\label{NLD1d}
If  $\theta = 1/2$ and the Dirichlet subdomain is smaller than the Neumann subdomain, then the error of the DN algorithm for two subdomains satisfies the linear convergence estimate,
\[
\parallel g^{[k]}\parallel_{L^{\infty}(\Gamma_i)} \leq
\begin{cases}
\left(\frac{b - a}{2b}\right)^k \max\big\{\parallel g^{[0]}\parallel_{L^{\infty}(\Gamma_i)}, \parallel h^{[0]}\parallel_{L^{\infty}(\Gamma_i)}\big\}, & \text{if}\; \delta_t > \frac{4\epsilon^2}{c^4},\\
\left(\frac{b - a}{\sqrt{2}b}\right)^k \max\big\{\parallel g^{[0]}\parallel_{L^{\infty}(\Gamma_i)}, \parallel h^{[0]}\parallel_{L^{\infty}(\Gamma_i)}\big\}, & \text{if}\; \delta_t < \frac{4\epsilon^2}{c^4}.
\end{cases}
\]
\end{theorem}
\begin{proof}
When $\delta_t > \frac{4\epsilon^2}{c^4}$, it is clear that $\lambda_{1,2}$ are real and positive, so are $\xi_{1,3}$. The iteration matrix $\mathbb{H}$ has two different eigenvalues $\frac{1}{2}(1 - \rho_1), \frac{1}{2}(1 - \rho_2)$. Hence, the convergence is achieved iff the spectral radius $\rho(\mathbb{H}) = \max\{\vert \frac{1}{2}(1 - \rho_1)\vert, \vert \frac{1}{2}(1 - \rho_2)\vert\}$ of the iteration matrix $\mathbb{H}$ is less than one, i.e $(\rho(\mathbb{H}))^k$ approaches to zero as $k\rightarrow\infty$. Upon simplification of the eigenvalue term we have $\frac{\sinh((b-a)\xi_i)}{2\sinh(b\xi_i)\cosh(a\xi_i)}$, for $i = 1, 3$. Since $\cosh(x)>1$ for $x>0$, we have using Lemma \ref{DNlemma},
\[  
\left|\frac{\sinh((b-a)\xi_i)}{2\sinh(b\xi_i)\cosh(a\xi_i)}\right| \leq \left| \frac{\sinh((b-a)\xi_i)}{2\sinh(b\xi_i)}\right| \leq \left(\frac{b - a}{2b}\right).
\]
For $\delta_t < \frac{4\epsilon^2}{c^4}$,  $\lambda_1$ and  $\lambda_2$ becomes complex conjugates, that we denote as $\lambda_{1,2}=\lambda_{\Re} \pm i\lambda_{\Im}, i = \sqrt{-1}$, where
\[
\lambda_{\Re}=\frac{ c^2 }{2\epsilon^2},\; \lambda_{\Im}=\frac{\sqrt{4\delta_t\epsilon^2 - \delta_t^2 c^4} }{2\delta_t\epsilon^2}.
\]
 Then $\xi_{1,3}$ as in \eqref{xi13} takes the form $\xi_1=\xi_{\Re} + i\xi_{\Im}, \xi_3=\xi_{\Re} - i\xi_{\Im}$, where 
 \[
\xi_{\Re}=\sqrt{\frac{\lambda_{\Re}+\sqrt{\lambda_{\Re}^2+\lambda_{\Im}^2}}{2}}, \; \xi_{\Im}=\sqrt{\frac{-\lambda_{\Re}+\sqrt{\lambda_{\Re}^2+\lambda_{\Im}^2}}{2}}.
\]
Clearly $\xi_{\Re}, \xi_{\Im}$ are positive numbers. Now if we take the modulus of the eigenvalues of the iteration matrix $\mathbb{H}$ and use the fact $\vert \cosh(a\xi_i)\vert>1$ for $a\xi_{\Re}>0$, we get
\begin{equation*}
\begin{aligned}
\left|\frac{\sinh((b-a)\xi_i)}{2\sinh(b\xi_i)\cosh(a\xi_i)}\right| & \leq &
\frac{\sqrt{\sinh^2((b-a)\xi_{\Re}) + \sin^2((b-a)\xi_{\Im})}}{2\sqrt{\sinh^2(b\xi_{\Re}) + \sin^2(b\xi_{\Im})}} \\
 & \leq & \frac{\sqrt{2}\sinh((b-a)\xi_{\Re})}{2\sinh(b\xi_{\Re})} \leq \left(\frac{b - a}{\sqrt{2}b}\right),
\end{aligned}
\end{equation*} 
the second inequality follows from $\sin^2((b-a)\xi_{\Im}) < \sinh^2((b-a)\xi_{\Im}) < \sinh^2((b-a)\xi_{\Re})$, as  $\xi_{\Im} < \xi_{\Re}$, and the last inequality follows from the Lemma \ref{DNlemma}. Hence the estimate.
\end{proof}
\begin{theorem}[Convergence of DN for $a>b$]\label{DLN}
If  $\theta = 1/2$ and the Dirichlet subdomain is larger than the Neumann subdomain, then the error of the DN algorithm for two subdomains satisfies the linear convergence estimate,
\[
\parallel g^{[k]}\parallel_{L^{\infty}(\Gamma_i)} \leq
\begin{cases}
\left(\frac{a-b}{2b}\right)^k \max\big\{\parallel g^{[0]}\parallel_{L^{\infty}(\Gamma_i)}, \parallel h^{[0]}\parallel_{L^{\infty}(\Gamma_i)}\big\}, & \text{if}\; \delta_t > \frac{4\epsilon^2}{c^4},\\
\left(\frac{a-b}{\sqrt{2}b}\right)^k \max\big\{\parallel g^{[0]}\parallel_{L^{\infty}(\Gamma_i)}, \parallel h^{[0]}\parallel_{L^{\infty}(\Gamma_i)}\big\}, & \text{if}\; \delta_t < \frac{4\epsilon^2}{c^4}.
\end{cases}
\]
\end{theorem}
\begin{proof}
The proof is similar to the case $a<b$ as in theorem \eqref{NLD1d}. To get the estimate, one has to adjust the negative sign inside modulus.
\end{proof}
\begin{remark}
The linear estimate in Theorem \eqref{DLN} does not always guarantee convergence. For example when $a > 3b$ (or $a > (1+\sqrt{2})b$ in 2nd case) i.e. when Dirichlet subdomain is much larger than Neumann subdomain, one should switch the role of the subdomains and solve the Neumann problem in larger subdomain.
\end{remark}
\subsection{Convergence analysis in 2D}
We now study the convergence of the DN method for a decomposition into two subdomains in two spatial dimension. Let the domain $\Omega = (-a, b)\times(0,L)$ be decomposed into two subdomains, given by $\Omega_1 = (-a, 0)\times(0,L)$ and $\Omega_2 = (0, b)\times(0,L)$.
We analyse this 2D case by converting it into a collection of 1D problems using the Fourier sine transform in the $y$-direction.
Writing the solution $u_i^{[k]}, v_i^{[k]}$ for $i=1, 2$ in a Fourier sine series along the $y$-direction yields
\[
u_{i}^{[k]}(x,y)=\sum_{m\geq1}\hat u_{i}^{[k]}(x,m)\sin\left(\frac{m\pi y}{L}\right),\;
v_{i}^{[k]}(x,y)=\sum_{m\geq1}\hat v_{i}^{[k]}(x,m)\sin\left(\frac{m\pi y}{L}\right)
\]
After a Fourier sine transform, the DN algorithm \eqref{DNCH} for the error equations in two dimensional setting for CH equation becomes
\begin{flalign*}
  & \begin{aligned} & \begin{cases}
\hat A\hat E_1^{[k]} & = 0, \quad \quad \text{in} \,\ (-a, 0) \\
\frac{\partial}{\partial x}\hat E_1^{[k]} & = 0, \quad \quad \text{at} \,\ x = -a\\
\hat E_1^{[k]} & = \begin{bmatrix} 
\hat g^{[k-1]} \\
\hat h^{[k-1]}
\end{bmatrix}, \quad \text{at} \,\ x = 0 \\
  \end{cases}\\
  \MoveEqLeft[2]\text{}
  \end{aligned}
  & &
  \begin{aligned}
      & \begin{cases}
\hat A\hat E_2^{[k]} & = 0,  \quad \quad \text{in} \,\ (0, b)\vspace{.2cm}\\
\frac{\partial}{\partial x}\hat E_2^{[k]} & = \frac{\partial}{\partial x}\hat E_1^{[k]},\quad \text{at} \,\ x = 0\vspace{.2cm} \\
\frac{\partial}{\partial x}\hat E_2^{[k]} & = 0, \quad \quad \text{at} \,\ x = b\\ 
    \end{cases} \\
    \MoveEqLeft[-5]\text{}
  \end{aligned}
\end{flalign*}
and the update condition becomes
\[ 
 \begin{bmatrix} 
\hat g^{[k]} \\
\hat h^{[k]}
\end{bmatrix} 
= 
\theta
\hat E_2^{[k]}(0,m)
+ 
(1 - \theta)
\begin{bmatrix} 
\hat g^{[k-1]} \\
\hat h^{[k-1]} 
\end{bmatrix}
\]
where
\begin{equation*}
\hat A = 
\begin{bmatrix} 
1 & -\delta_t(\frac{d^2}{dx^2} - p_m^2) \\
\epsilon^2(\frac{d^2}{dx^2} - p_m^2) - c^2 & 1 
\end{bmatrix},\quad
\hat E_j^{[k]} = \begin{bmatrix} 
\hat u_j^{[k]}(x) \\
\hat v_j^{[k]}(x)
\end{bmatrix}.
\end{equation*}
$\hat E_j^{[k]}$ for $j=1,2$ denotes the error in Fourier space and $p_m^2 = \frac{\pi ^2 m^2}{L^2}$.
Now we can do the same treatment similar to one dimensional analysis and get the recurrence relation as:
\begin{equation}\label{reccurence2d}
\begin{bmatrix} 
\hat g^{[k]} \\
\hat h^{[k]} 
\end{bmatrix} 
= \hat {\mathbb{H}}
\begin{bmatrix} 
\hat g^{[k-1]} \\
\hat h^{[k-1]} 
\end{bmatrix},
\end{equation}
where $ \hat {\mathbb{H}} $ is the iteration matrix given by   
\begin{equation*}
\hat {\mathbb{H}} = 
\begin{bmatrix} 
1-\theta+\theta\frac{-\lambda_1\rho_1 + \lambda_2\rho_2}{\lambda} & \theta\frac{\delta_t\lambda_1\lambda_2(\rho_1 - \rho_2)}{\lambda} \\
\\
\theta\frac{\rho_2 - \rho_1}{\delta_t\lambda} & 1-\theta+\theta\frac{\lambda_2\rho_1 - \lambda_1\rho_2}{\lambda}
\end{bmatrix} 
\end{equation*}
with the expressions of $\rho_1$ and $\rho_2$ exactly as defined earlier, but having the modified Fourier symbols
\begin{equation*}
\xi_{1,2} = \pm\sqrt{\lambda_1 + p_m^2},\quad \xi_{3,4} = \pm\sqrt{\lambda_2 + p_m^2},
\end{equation*}
with the same $\lambda_{1,2}$. 
\begin{theorem}[Convergence of DN in 2D]\label{NLD2d}
\begin{enumerate}[label=(\roman*)]
\item 
When the subdomains are of the same size, $a=b$, the DN algorithm converges linearly for $0<\theta<1, \theta\neq 1/2.$ For $\theta = 1/2$, it converges in two iterations.\label{item:1:2d}
\item
If $\theta = 1/2$ and the Dirichlet subdomain is smaller than the Neumann subdomain, i.e., $b>a$, then the error of the DN algorithm for two subdomains satisfies the linear convergence estimate,
\[
\parallel g^{[k]} \parallel_{L^{2}(\Gamma)} \leq
\begin{cases}
\left(\frac{b - a}{2b}\right)^k \max\left\{\parallel g^{[0]} \parallel_{L^{2}(\Gamma)}, \parallel h^{[0]} \parallel_{L^{2}(\Gamma)}\right\}, & \text{if}\; \delta_t > \frac{4\epsilon^2}{c^4},\\
\left(\frac{b - a}{\sqrt{2}b}\right)^k \max\left\{\parallel g^{[0]} \parallel_{L^{2}(\Gamma)}, \parallel h^{[0]} \parallel_{L^{2}(\Gamma)}\right\}, & \text{if}\; \delta_t < \frac{4\epsilon^2}{c^4}.
\end{cases}
\]\label{item:2:2d}
\item 
And for $\theta = 1/2$ and when the Dirichlet subdomain is larger than the Neumann subdomain, i.e., $a>b$, then the error of the DN algorithm for two subdomains satisfies the linear convergence estimate,
\[
\parallel g^{[k]}\parallel_{L^{2}(\Gamma)} \leq
\begin{cases}
\left(\frac{a - b}{2b}\right)^k \max\left\{\parallel g^{[0]} \parallel_{L^{2}(\Gamma)}, \parallel h^{[0]} \parallel_{L^{2}(\Gamma)}\right\}, & \text{if}\; \delta_t > \frac{4\epsilon^2}{c^4},\\
\left(\frac{a - b}{\sqrt{2}b}\right)^k \max\left\{\parallel g^{[0]} \parallel_{L^{2}(\Gamma)}, \parallel h^{[0]} \parallel_{L^{2}(\Gamma)}\right\}, & \text{if}\; \delta_t < \frac{4\epsilon^2}{c^4}.
\end{cases}
\]\label{item:3:2d}
\end{enumerate}
\end{theorem}
\begin{proof}
\ref{item:1:2d} If $a=b$, then \eqref{reccurence2d} gives 
\begin{equation*}
\begin{bmatrix} 
\hat g^{[k]} \\
\hat h^{[k]}
\end{bmatrix} 
=  \begin{bmatrix} 
1-2\theta & 0 \\
0 & 1-2\theta 
\end{bmatrix}
\begin{bmatrix} 
\hat g^{[k-1]} \\
\hat h^{[k-1]} 
\end{bmatrix}.
\end{equation*}
Now back-transformation will lead to the conclusion \ref{item:1:2d}. \\[.1cm]
\ref{item:2:2d} For $\delta_t > \frac{4\epsilon^2}{c^4}$, it is clear that $\lambda_{1,2}$ are real and positive, so are $\xi_{1,3}$. The spectral radius of the iteration matrix $\hat {\mathbb{H}}$ is given by $\rho(\hat {\mathbb{H}}) = \max\{\vert \frac{1}{2}(1 - \rho_1)\vert, \vert \frac{1}{2}(1 - \rho_2)\vert\}$. By the recurrence relation \eqref{reccurence2d} and using the Parseval-Plancherel identity we get
\[
\begin{array}{cc}
\parallel g^{[k]}\parallel_{L^{2}(\Gamma)} \leq \rho(\hat {\mathbb{H}}) \max\left\{\parallel g^{[k-1]} \parallel_{L^{2}(\Gamma)}, \parallel h^{[k-1]} \parallel_{L^{2}(\Gamma)}\right\}, \\ 
\parallel h^{[k]}\parallel_{L^{2}(\Gamma)} \leq \rho(\hat {\mathbb{H}}) \max\left\{\parallel g^{[k-1]} \parallel_{L^{2}(\Gamma)}, \parallel h^{[k-1]} \parallel_{L^{2}(\Gamma)}\right\}.
\end{array}
\]
Hence we get the result by estimating $\rho(\hat {\mathbb{H}})$ as in the case of Theorem \ref{NLD1d}. \\
For $\delta_t < \frac{4\epsilon^2}{c^4}$,  $\lambda_1$ and  $\lambda_2$ becomes complex conjugates and have the form given in Theorem \ref{NLD1d}, whereas $\xi_{1,3}$ take the form $\xi_1 =  \sqrt{\lambda_1 + p_m^2} = \xi_{\Re}(m) + i\xi_{\Im}(m), \xi_3 = \sqrt{\lambda_2 + p_m^2} = \xi_{\Re}(m) - i\xi_{\Im}(m)$, where 
 \[
\xi_{\Re}(m)=\sqrt{\frac{p_m^2 + \lambda_{\Re}+\sqrt{(p_m^2 + \lambda_{\Re})^2+\lambda_{\Im}^2}}{2}}, \; \xi_{\Im}(m)=\frac{\lambda_{\Im}}{2\xi_{\Re}(m)}.
\]
Clearly $\xi_{\Re}(m), \xi_{\Im}(m)$ are positive numbers for all $m \geq 1$. Finally using Parseval-Plancherel identity  and as similar to the proof of Theorem \ref{NLD1d}, we get the estimate.\\[.1cm]
\ref{item:3:2d}
For the last case, $a>b$, as well, a similar argument as above and in Theorem \ref{DLN} leads to the required estimate.
\end{proof}

\section{The Neumann-Neumann method for multiple subdomains}\label{Section3}

We now introduce the 2nd method of our interest, namely the NN algorithm for the CH equation for multiple subdomains. For a detail study on two subdomain decomposition, see \cite{garai2021convergence}. Suppose $\Omega$ is decomposed into non-overlapping subdomains $\{\Omega_i, 1\leq i \leq N\}$, as illustrated in Fig. \ref{multipledomain}.
The NN algorithm starts with initial guesses $g_i^{[0]}, h_i^{[0]}$ along the interfaces $\Gamma_i$ for $i = 1,\ldots,N-1$, and then performs the following two steps: at each iteration $k$, one first solves Dirichlet sub-problems on each $\Omega_i$ in parallel,
\begin{equation}\label{MNCH}
\begin{aligned}
\begin{bmatrix} 
I & -\delta_t\Delta \\
\epsilon^2\Delta-c^2 & I 
\end{bmatrix}
\quad
\begin{bmatrix} 
u_i^{[k]} \\
v_i^{[k]} 
\end{bmatrix}
& =  \begin{bmatrix} 
f_u \\
f_v 
\end{bmatrix},
\quad \mbox{in}\,\ \Omega_i,\\
\mathcal{B}\begin{bmatrix} 
u_i^{[k]} \\
v_i^{[k]} 
\end{bmatrix}
& =  0,\quad \mbox{on}\,\ \partial\Omega_i\cap\partial\Omega,\\
\begin{bmatrix} 
u_i^{[k]} \\
v_i^{[k]} 
\end{bmatrix}
 & = 
\begin{bmatrix} 
g_i^{[k-1]} \\
h_i^{[k-1]} 
\end{bmatrix}
\quad \mbox{on}\,\ \Gamma_i,
 \end{aligned}
\end{equation}
then the jump in Neumann traces on the interfaces are calculated and one solves the following Neumann sub-problems on each $\Omega_i$ in parallel,
\begin{equation}\label{MNNS}
\begin{aligned}
\begin{bmatrix} 
I & -\delta_t\Delta \\
\epsilon^2\Delta-c^2 & I 
\end{bmatrix}
\begin{bmatrix} 
\phi_i^{[k]} \\
\psi_i^{[k]} 
\end{bmatrix} 
& =  0,
\quad \mbox{in}\,\ \Omega_i,\\
\mathcal{B}\begin{bmatrix} 
\phi_i^{[k]} \\
\psi_i^{[k]} 
\end{bmatrix}
 & =  0,\quad \mbox{on}\,\ \partial\Omega_i\cap\partial\Omega,\\
\frac{\partial}{\partial n}\begin{bmatrix} 
\phi_i^{[k]} \\
\psi_i^{[k]} 
\end{bmatrix} & = 
\frac{\partial}{\partial n}\begin{bmatrix} 
u_i^{[k]} - u_{i+1}^{[k]} \\
v_i^{[k]} - v_{i+1}^{[k]} 
\end{bmatrix},
\quad \mbox{on}\,\ \Gamma_i.
 \end{aligned}
\end{equation}
\begin{figure}
    \centering
    \subfloat{{\includegraphics[width=5cm]{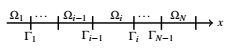} }}
    \qquad
    \subfloat{{\includegraphics[width=5cm]{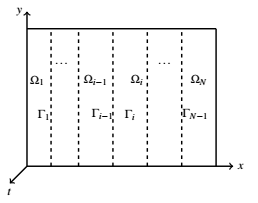} }}
    \caption{Multiple decomposition of domain in 1D (left) and 2D (right) at a particular time level.}
    \label{multipledomain}
\end{figure}
Lastly the interface traces are updated by
\[
\begin{bmatrix} 
g_i^{[k]} \\
h_i^{[k]} 
\end{bmatrix}
 = \begin{bmatrix} 
g_i^{[k-1]} \\
h_i^{[k-1]} 
\end{bmatrix} -
  \theta \begin{bmatrix} 
\phi_i^{[k]} - \phi_{i+1}^{[k]} \\
\psi_i^{[k]} - \psi_{i+1}^{[k]} 
\end{bmatrix}_{\big|_{\Gamma_i}}, 
\]
where $\theta\in(0, 1)$ is a relaxation parameter.
\subsection{Convergence Analysis in 1D}
We present our convergence estimates for the NN algorithm in 1D case. The domain $\Omega:=(0,L)$
is decomposed into $N$ subdomains $\Omega_{i}:=(x_{i-1},x_{i})$, $i=1,\ldots,N$, with subdomain length $d_{i}:=x_{i}-x_{i-1}$. We solve the error equations given by
\begin{flalign*}\label{NNerr1}
  & \begin{aligned} & \begin{cases}
A E_i^{[k]} & = 0, \quad \quad \text{in} \,\ \Omega_i, \\
 E_i^{[k]} & = \begin{bmatrix} 
 g_{i-1}^{[k-1]} \\
 h_{i-1}^{[k-1]} 
\end{bmatrix}, \quad \text{on} \,\ \Gamma_{i-1}, \\
 E_i^{[k]} & = \begin{bmatrix} 
 g_i^{[k-1]} \\
 h_i^{[k-1]} 
\end{bmatrix}, \quad \text{on} \,\ \Gamma_{i}, 
\end{cases}\\
  \MoveEqLeft[2]\text{}
  \end{aligned}
  & &
  \begin{aligned}
      & \begin{cases}
A F_i^{[k]} & = 0, \quad \quad \text{in} \,\ \Omega_i, \\
\frac{\partial}{\partial x}\ F_i^{[k]} & = \frac{\partial}{\partial x}\begin{bmatrix} 
u_{i-1}^{[k]} - u_{i}^{[k]} \\
v_{i-1}^{[k]} - v_{i}^{[k]} 
\end{bmatrix}, \quad \text{on} \,\ \Gamma_{i-1}, \\
\frac{\partial}{\partial x}\ F_i^{[k]} & = \frac{\partial}{\partial x}\begin{bmatrix} 
u_i^{[k]} - u_{i+1}^{[k]} \\
v_i^{[k]} - v_{i+1}^{[k]} 
\end{bmatrix}, \quad \text{on} \,\ \Gamma_i, 
 \end{cases} \\
    \MoveEqLeft[-5]\text{}
  \end{aligned}
\end{flalign*}
except for the first and last subdomains, where at the physical boundaries the Dirichlet condition in the Dirichlet step and Neumann condition in the Neumann step are replaced by homogeneous Neumann condition. The interface values for the next iteration are then updated as
\begin{equation}\label{trace1d}
\begin{bmatrix} 
g_i^{[k]} \\
h_i^{[k]} 
\end{bmatrix} 
= 
\begin{bmatrix} 
 g_i^{[k-1]} \\
 h_i^{[k-1]} 
\end{bmatrix}
-
\theta
(F_i^{[k]} - F_{i+1}^{[k]})_{\big|_{\Gamma_i}},
\end{equation} 
where
\[
 A = 
\begin{bmatrix} 
1 & -\delta_t\frac{d^2}{dx^2} \\
\epsilon^2\frac{d^2}{dx^2} - c^2 & 1 
\end{bmatrix},\quad
 E_i^{[k]} = \begin{bmatrix} 
 u_i^{[k]}(x) \\
 v_i^{[k]}(x)
\end{bmatrix},\quad
 F_i^{[k]} = \begin{bmatrix} 
 \phi_i^{[k]}(x) \\
 \psi_i^{[k]}(x)
\end{bmatrix}.
\]
Similar to the case of one dimensional DN, we get the subdomain solution at $k-$th iteration for the Dirichlet step 
\[
E_i^{[k]} = 
\begin{bmatrix} 
\frac{\delta_t\lambda_1(\eta_{1,i}\sigma_{1,i,x} - \eta_{1,i-1}\sigma_{1,i+1,x})}{\sigma_{1,i}} + \frac{\delta_t\lambda_2(\eta_{2,i}\sigma_{3,i,x} - \eta_{2,i-1}\sigma_{3,i+1,x})}{\sigma_{3,i}} \\
\frac{\eta_{1,i}\sigma_{1,i,x} - \eta_{1,i-1}\sigma_{1,i+1,x}}{\sigma_{1,i}} + \frac{\eta_{2,i}\sigma_{3,i,x} - \eta_{2,i-1}\sigma_{3,i+1,x}}{\sigma_{3,i}}
\end{bmatrix}
\]
for $i=2,\ldots,N-1$ and for the the first and last subdomain we have
\[
E_1^{[k]} = 
\begin{bmatrix} 
\frac{\delta_t\lambda_1\eta_{1,1}\gamma_{1,1,x}}{\gamma_{1,1}} + \frac{\delta_t\lambda_2\eta_{2,1}\gamma_{3,1,x}}{\gamma_{3,1}} \\
\frac{\eta_{1,1}\gamma_{1,1,x}}{\gamma_{1,1}} + \frac{\eta_{2,1}\gamma_{3,1,x}}{\gamma_{3,1}} 
\end{bmatrix},\; 
E_{N}^{[k]} = 
\begin{bmatrix} 
\frac{\delta_t\lambda_1\eta_{1,N-1}\gamma_{1,N+1,x}}{\gamma_{1,N}} + \frac{\delta_t\lambda_2\eta_{2,N-1}\gamma_{3,N+1,x}}{\gamma_{3,N}} \\
\frac{\eta_{1,N-1}\gamma_{1,N+1,x}}{\gamma_{1,N}} + \frac{\eta_{2,N-1}\gamma_{3,N+1,x}}{\gamma_{3,N}} 
\end{bmatrix}
\]
and similarly for the Neumann step, we get
\[
F_{i}^{[k]} = 
\begin{bmatrix} 
\delta_t\lambda_1(C_{i,1}\gamma_{1,i,x} - C_{i,2}\gamma_{1,i+1,x})  + \delta_t\lambda_2(D_{i,1}\gamma_{3,i,x} - D_{i,2}\gamma_{3,i+1,x}) \\
C_{i,1}\gamma_{1,i,x} - C_{i,2}\gamma_{1,i+1,x}  + D_{i,1}\gamma_{3,i,x} - D_{i,2}\gamma_{3,i+1,x}  
\end{bmatrix}
\]
for $i=2,\ldots,N-1$ and for the the first and last subdomain we have
\[
F_{1}^{[k]} =  
\begin{bmatrix} 
\delta_t\lambda_1 C_{1,1}\gamma_{1,1,x}  + \delta_t\lambda_2 D_{1,1}\gamma_{3,1,x} \\
C_{1,1}\gamma_{1,1,x}  + D_{1,1}\gamma_{3,1,x}
\end{bmatrix},\;
F_{N}^{[k]} = 
\begin{bmatrix} 
\delta_t\lambda_1 C_{N,1}\gamma_{1,N+1,x}  + \delta_t\lambda_2 D_{N,1}\gamma_{3,N+1,x} \\
C_{N,1}\gamma_{1,N+1,x}  + D_{N,1}\gamma_{3,N+1,x}
\end{bmatrix}
\]     
where    
\[
\begin{array}{cccccccc}
C_{1,1}=\left(\frac{\eta_{1,1}}{\gamma_{1,1}} + \frac{\eta_{1,1}\gamma_{1,2}}{\sigma_{1,1}\sigma_{1,2}} - \frac{\eta_{1,2}}{\sigma_{1,1}^2}\right), 
D_{1,1} = \left(\frac{\eta_{2,1}}{\gamma_{3,1}} + \frac{\eta_{2,1}\gamma_{3,2}}{\sigma_{3,1}\sigma_{3,2}} - \frac{\eta_{2,2}}{\sigma_{3,1}^2}\right),\\
C_{2,2}=\left(\frac{\eta_{1,1}\sigma_{1,1}}{\sigma_{1,2}\gamma_{1,1}} + \frac{\eta_{1,1}\gamma_{1,2}}{\sigma_{1,2}^2} - \frac{\eta_{1,2}}{\sigma_{1,2}^2}\right),
D_{2,2}=\left(\frac{\eta_{2,1}\sigma_{3,1}}{\sigma_{3,2}\gamma_{3,1}} + \frac{\eta_{2,1}\gamma_{3,2}}{\sigma_{3,2}^2} - \frac{\eta_{2,2}}{\sigma_{3,2}^2}\right),\\
C_{i,1}=\left(-\frac{\eta_{1,i-1}}{\sigma_{1,i}^2} + \frac{\eta_{1,i}\gamma_{1,i}}{\sigma_{1,i}^2} + \frac{\eta_{1,i}\gamma_{1,i+1}}{\sigma_{1,i}\sigma_{1,i+1}} - \frac{\eta_{1,i+1}}{\sigma_{1,i}\sigma_{1,i+1}}\right), \text{for} \; i=2,\ldots,N-1,\\ 
D_{i,1}=\left(-\frac{\eta_{2,i-1}}{\sigma_{3,i}^2} + \frac{\eta_{2,i}\gamma_{3,i}}{\sigma_{3,i}^2} + \frac{\eta_{2,i}\gamma_{3,i+1}}{\sigma_{3,i}\sigma_{3,i+1}} - \frac{\eta_{2,i+1}}{\sigma_{3,i}\sigma_{3,i+1}}\right), \text{for} \; i=2,\ldots,N-1,\\
C_{i,2}=\left(-\frac{\eta_{1,i-2}}{\sigma_{1,i}\sigma_{1,i-1}} + \frac{\eta_{1,i-1}\gamma_{1,i-1}}{\sigma_{1,i}\sigma_{1,i-1}} + \frac{\eta_{1,i-1}\gamma_{1,i}}{\sigma_{1,i}^2} - \frac{\eta_{1,i}}{\sigma_{1,i}^2}\right), \text{for} \;i=3,\ldots,N,\\
D_{i,2}=\left(-\frac{\eta_{2,i-2}}{\sigma_{3,i}\sigma_{3,i-1}} + \frac{\eta_{2,i-1}\gamma_{3,i-1}}{\sigma_{3,i}\sigma_{3,i-1}} + \frac{\eta_{2,i-1}\gamma_{3,i}}{\sigma_{3,i}^2} - \frac{\eta_{2,i}}{\sigma_{3,i}^2}\right), \text{for} \;i=3,\ldots,N,\\
C_{N-1,1}=\left(-\frac{\eta_{1,N-2}}{\sigma_{1,N-1}^2} + \frac{\eta_{1,N-1}\gamma_{1,N-1}}{\sigma_{1,N-1}^2} + \frac{\eta_{1,N-1}\sigma_{1,N}}{\sigma_{1,N-1}\gamma_{1,N}}\right),\\
D_{N-1,1}=\left(-\frac{\eta_{2,N-2}}{\sigma_{3,N-1}^2} + \frac{\eta_{2,N-1}\gamma_{3,N-1}}{\sigma_{3,N-1}^2} + \frac{\eta_{2,N-1}\sigma_{3,N}}{\sigma_{3,N-1}\gamma_{3,N}}\right),\\ 
C_{N,1}=\left( \frac{\eta_{1,N-2}}{\sigma_{1,N}\sigma_{1,N-1}} -\frac{\eta_{1,N-1}}{\gamma_{1,N}} - \frac{\eta_{1,N-1}\gamma_{1,N-1}}{\sigma_{1,N-1}\sigma_{1,N}}\right),
D_{N,1} = \left( \frac{\eta_{2,N-2}}{\sigma_{3,N}\sigma_{3,N-1}} -\frac{\eta_{2,N-1}}{\gamma_{3,N}} - \frac{\eta_{2,N-1}\gamma_{3,N-1}}{\sigma_{3,N-1}\sigma_{3,N}}\right), 
\end{array}
\]
with
$d_{i,x}:=x - x_{i-1}, \sigma_{1,i,x}:=\sinh(\xi_1d_{i,x}),  \sigma_{3,i,x}:=\sinh(\xi_3d_{i,x}),  \gamma_{1,i,x}:=\cosh(\xi_1d_{i,x}),\\ \gamma_{3,i,x}:=\cosh(\xi_3d_{i,x}),   \text{and}\; \sigma_{1,i}:=\sinh(\xi_1d_{i}), \sigma_{3,i}:=\sinh(\xi_3d_{i}), \gamma_{1,i}:=\cosh(\xi_1d_{i}), \gamma_{3,i}:=\cosh(\xi_3d_{i}), \eta_{1,i}: = \frac{g_i^{[k-1]}-\delta_t\lambda_2 h_i^{[k-1]}}{\delta_t\lambda}, \eta_{2,i}: = \frac{g_i^{[k-1]}-\delta_t\lambda_1 h_i^{[k-1]}}{-\delta_t\lambda},  \text{for}\; i=1,\ldots,N.
$
Using the above subdomain solutions in \eqref{trace1d}, the update condition is reduced to the following form 
\begin{equation}\label{iter1dmulti}
 \begin{bmatrix} 
\textbf{g}^{[k]} \\
\textbf{h}^{[k]} 
\end{bmatrix} 
= \mathbb{T}
\begin{bmatrix} 
\textbf{g}^{[k-1]} \\
\textbf{h}^{[k-1]} 
\end{bmatrix}
\end{equation} 
where $
 \begin{bmatrix} 
\textbf{g}^{[k]} & \textbf{h}^{[k]}
\end{bmatrix}^\mathsf{T} := [g_1^{[k]}, h_1^{[k]}, g_2^{[k]}, h_2^{[k]}, ..., g_{N-1}^{[k]}, h_{N-1}^{[k]}]^\mathsf{T} $
and the iteration matrix $\mathbb{T} \in\mathbb{R}^{2N-2\times 2N-2}$ is given by 
{\footnotesize
\begin{equation}\label{Iterationmatrix}
\mathbb{T} = \begin{bmatrix} 
\alpha_1^{1} & \dots & \dots& \alpha_1^{6} \\ 
\beta_1^{1} & \dots & \dots& \beta_1^{6} \\ 
\alpha_2^{1}&\alpha_2^{2} & \dots & \dots& \alpha_2^{7} & \alpha_2^{8}\\ 
\beta_2^{1}&\beta_2^{2} & \dots& \dots & \beta_2^{7}& \beta_2^{8} \\
\alpha_3^{1}&\alpha_3^{2}&\alpha_3^{3}&\dots&\dots&\dots &\alpha_3^{9}&\alpha_3^{10}\\
\beta_3^{1}&\beta_3^{2}&\beta_3^{3} &\dots&\dots&\dots& \beta_3^{9} & \beta_3^{10}\\
 &  & \alpha_4^{1} & \alpha_4^{2} & \dots& \dots& \dots& \dots &\alpha_4^{9}& \alpha_4^{10} \\
 &  & \beta_4^{1} & \beta_4^{2} & \dots& \dots& \dots& \dots &\beta_4^{9}& \beta_4^{10} \\
&  & & & & \ddots & \ddots &\ddots & \ddots\\
&  & & & \alpha_{N-4}^{1} & \alpha_{N-4}^{2}&\alpha_{N-4}^{3} & \dots& \dots& \dots & \alpha_{N-4}^{9}& \alpha_{N-4}^{10}& &\\
 & & & & \beta_{N-4}^{1} & \beta_{N-4}^{2}&\beta_{N-4}^{3} & \dots& \dots& \dots & \beta_{N-4}^{9}& \beta_{N-4}^{10}& &\\
& & & & &  & \alpha_{N-3}^{1} & \alpha_{N-3}^{2}&\alpha_{N-3}^{3} & \dots& \dots& \dots & \alpha_{N-3}^{9}& \alpha_{N-3}^{10}\\
& & & & &  & \beta_{N-3}^{1} & \beta_{N-3}^{2}&\beta_{N-3}^{3} & \dots& \dots& \dots & \beta_{N-3}^{9}& \beta_{N-3}^{10}\\
& & & & & &  &  &\alpha_{N-2}^{1}&\alpha_{N-2}^{2}&\alpha_{N-2}^{3}  & \dots & \dots& \alpha_{N-2}^{8}\\
& & & & & & &   &\beta_{N-2}^{1}&\beta_{N-2}^{2}&\beta_{N-2}^{3}  & \dots& \dots & \beta_{N-2}^{8}\\
& & & & & & & & & &\alpha_{N-1}^{1}  & \dots & \dots & \alpha_{N-1}^{6}\\
& & & & & & & & & &\beta_{N-1}^{1}  & \dots & \dots & \beta_{N-1}^{6}\\
\end{bmatrix}.
\end{equation}
\par}
\noindent The explicit expressions of the elements of the matrix $\mathbb{T}$ are given in Appendix A.
We now analyze the convergence behaviour for the case $\delta_t > \frac{4\epsilon^2}{c^4}$, for which we need the following Lemma. 
\begin{lemma}\label{NNlemmamultiple}
For $t>0$, we have $\frac{\cosh(t)}{\sinh^2(t)} < \frac{2}{t^2}$
\end{lemma}
\begin{proof}
Consider the function $f(t) = t - \sinh(t)\left(\frac{\cosh(t)}{2}\right)^{-1/2}$. We have $f(0)=0$ and $f'(t) = 1-f_1(t)$, where $f_1(t)=\frac{2+\sinh^2(t)}{\sqrt{2}\cosh^{3/2}(t)}$. We now show that $f_1(t)>1$, for all $t>0$. We have 
\[f_1'(t) = \frac{\sqrt{2}\sinh(t)(\cosh^2(t)-3)}{4\cosh^{5/2}(t)}.\]
So the function $f_1(t)$ becomes monotonically decreasing if $\cosh^2(t)<3$, i.e in the interval $(0, \acosh(\sqrt{3}))$ or $(0, 1.1462)$. Hence the minimum value of $f_1(t)$ is $\lim\limits_{t\rightarrow \acosh(\sqrt{3})} f_1(t) = \frac{2\sqrt{2}}{3^{3/4}}$, which is strictly grater than one. And the function $f_1(t)$ becomes monotonically increasing if $\cosh^2(t)\geq 3$, i.e in the interval $[ \acosh(\sqrt{3}), \infty)$, so $f_1(t)> \frac{2\sqrt{2}}{3^{3/4}}$. Hence the function $f(t)$ is monotonically decreasing, i.e $f(t) < f(0)$ for all $t>0$. This completes the result. 
\end{proof}

\begin{theorem}[Convergence of NN for multiple subdomains]\label{NNequalthm}
For $\theta=1/4$, the NN algorithm \eqref{MNCH}-\eqref{MNNS} for multiple subdomains with equal length $d$, satisfying $d > d^*$, is convergent. Moreover, we have the following estimates
\[
\max_{1 \leq i \leq N-1} \parallel g_i^{[k]} \parallel_{L^\infty(\Gamma_i)} 
<  (\alpha^*)^k \max_{1 \leq i \leq N-1}\left\{\parallel g_i^{[0]} \parallel_{L^\infty(\Gamma_i)}, \parallel h_i^{[0]} \parallel_{L^\infty(\Gamma_i)}\right\},
\]
\[
\max_{1 \leq i \leq N-1}\parallel h_i^{[k]} \parallel_{L^\infty(\Gamma_i)} 
<  (\beta^*)^k \max_{1 \leq i \leq N-1}\left\{\parallel g_i^{[0]} \parallel_{L^\infty(\Gamma_i)}, \parallel h_i^{[0]} \parallel_{L^\infty(\Gamma_i)}\right\},
\]
where the expression of $\alpha^*, \beta^*, d^*$ are given in  \eqref{alphastar1dequal}, \eqref{betastar1dequal}, \eqref{dstar1dequal} respectively.
\end{theorem}
\begin{proof}
For equi-length subdomains we have $d_i = d$ for $i=1,\ldots,N$, so that $\sigma_{j,i} = \sigma_{j}, \gamma_{j, i}=\gamma_{j}$ for all $i$ and for $j = 1, 3$. The infinity-norm of $\mathbb{T}$ is given by 
\begin{equation}\label{normh}
\parallel \mathbb{T} \parallel_{\infty} = \max_{1 \leq i \leq N-1} \left\{ \sum_{j=1}^{S_i} \vert \alpha_i^j \vert, \sum_{j=1}^{S_i} \vert \beta_i^j \vert\right\}, 
\end{equation}
where 
\[
S_i = \begin{cases}
    & 6, \;\text{for}\; i=1, N-1 \\
    & 8, \;\text{for}\; i=2, N-2  \\
    & 10, \;\text{for}\; 3 \leq i \leq N-3.
\end{cases}
\]
We show that $\parallel \mathbb{T} \parallel_{\infty}$, is strictly smaller than one. For equal subdomains, we have for $3 \leq i \leq N-3$, 
\begin{equation*}
\begin{aligned}
\alpha_i^1 & = \frac{1}{4\lambda}\left( -\frac{\lambda_1}{\sigma_1^2} + \frac{\lambda_2}{\sigma_3^2} \right)= \alpha_i^9\\
\alpha_i^2 & = \frac{\delta_t\lambda_1\lambda_2}{4\lambda}\left( \frac{1}{\sigma_1^2} - \frac{1}{\sigma_3^2} \right) =\alpha_i^{10}\\
\alpha_i^3 &= \frac{1}{\lambda}\left( \frac{\lambda_1\gamma_1}{\sigma_1^2} - \frac{\lambda_2\gamma_3}{\sigma_3^2} \right)=\alpha_i^7\\
\alpha_i^4 & = \frac{\delta_t\lambda_1\lambda_2}{\lambda}\left(- \frac{\gamma_1}{\sigma_1^2} + \frac{\gamma_3}{\sigma_3^2} \right)
=\alpha_i^8\\
\alpha_i^5 &= 1 - \frac{\lambda_1}{4\lambda}\left( \frac{4\gamma_1^2}{\sigma_1^2} + \frac{2}{\sigma_1^2} \right) + \frac{\lambda_2}{4\lambda}\left( \frac{4\gamma_3^2}{\sigma_3^2} + \frac{2}{\sigma_3^2} \right)\\
\alpha_i^6 & =  \frac{\delta_t\lambda_1\lambda_2}{4\lambda}\left( \frac{4\gamma_1^2}{\sigma_1^2} + \frac{2}{\sigma_1^2} \right) - \frac{\delta_t\lambda_1\lambda_2}{4\lambda}\left( \frac{4\gamma_3^2}{\sigma_3^2} + \frac{2}{\sigma_3^2} \right).
\end{aligned}
\end{equation*}
Using the fact that $\xi_1 = \sqrt{\lambda_1}, \xi_3 = \sqrt{\lambda_2}$ and $\sigma_j > \xi_j d$ for $j=1, 3$, we have 
\[ 
\vert \alpha_i^1 \vert \leq \frac{1}{4\lambda} \left(\left| -\frac{\lambda_1}{\sigma_1^2} \right| + \left| \frac{\lambda_2}{\sigma_3^2} \right|\right) \leq \frac{1}{4\lambda} \left(\frac{\lambda_1}{\xi_1^2 d^2} + \frac{\lambda_2}{\xi_3^2 d^2}\right) = \frac{1}{2\lambda d^2}.
\]
Since $\lambda_2 < \lambda_1$, we have $\vert\alpha_i^2\vert \leq \frac{\delta_t\lambda_1}{2\lambda d^2}.$
Using Lemma \ref{NNlemmamultiple} we get the estimates for the terms $\vert\alpha_i^3\vert, \vert\alpha_i^4\vert$ as: 
\[
\begin{array}{cc} 
\vert \alpha_i^3 \vert \leq \frac{1}{\lambda} \left(\frac{2\lambda_1}{\xi_1^2 d^2} + \frac{2\lambda_2}{\xi_3^2 d^2}\right) \leq \frac{4}{\lambda d^2}, \\
\vert \alpha_i^4 \vert \leq \frac{\delta_t\lambda_1\lambda_2}{\lambda} \left(\frac{2}{\xi_1^2 d^2} + \frac{2}{\xi_3^2 d^2}\right) \leq \frac{4\delta_t\lambda_1}{\lambda d^2},
\end{array}
\]
Using the identity $\gamma_j^2 = 1 + \sigma_j^2, \text{and}\; \sigma_j > \xi_j d$ for $j = 1, 3$ we have
\begin{equation*}
\begin{aligned}
\vert\alpha_i^5\vert  = & \left| 1 - \frac{\lambda_1}{4\lambda}\left( \frac{4\gamma_1^2}{\sigma_1^2} + \frac{2}{\sigma_1^2} \right) + \frac{\lambda_2}{4\lambda}\left( \frac{4\gamma_3^2}{\sigma_3^2} + \frac{2}{\sigma_3^2} \right) \right| \\
 = & \left|  1 - \frac{\lambda_1 - \lambda_2}{\lambda} - \frac{\lambda_1}{4\lambda} \frac{6}{\sigma_1^2} + \frac{\lambda_2}{4\lambda} \frac{6}{\sigma_3^2}\right|\\
 \leq &  \left(\frac{\lambda_1}{4\lambda}  \frac{6}{\xi_1^2 d^2} + 
\frac{\lambda_2}{4\lambda}  \frac{6}{\xi_3^2 d^2}\right)
= \frac{3}{\lambda d^2}.
\end{aligned}
\end{equation*}
Similarly we get the estimate for $\vert\alpha_i^6\vert $ as:
\begin{equation*}
\begin{aligned}
\vert\alpha_i^6\vert  = & \left| \frac{\delta_t\lambda_1\lambda_2}{4\lambda}\left( \frac{4\gamma_1^2}{\sigma_1^2} + \frac{2}{\sigma_1^2} \right) - \frac{\delta_t\lambda_1\lambda_2}{4\lambda}\left( \frac{4\gamma_3^2}{\sigma_3^2} + \frac{2}{\sigma_3^2} \right)\right| \\
= & \left| \frac{\delta_t\lambda_1\lambda_2}{4\lambda}\left(\frac{6}{\sigma_1^2} +\frac{6}{\sigma_3^2} \right)\right|
\leq \frac{3\delta_t\lambda_1}{\lambda d^2}.
\end{aligned}
\end{equation*}
Therefore we have the estimate of $\sum_{j=1}^{10} \vert \alpha_i^j \vert$ for $3 \leq i \leq N-3$, as 
\begin{equation}\label{alphastar1dequal}
\sum_{j=1}^{10} \vert \alpha_i^j \vert < \alpha^* = \frac{12(1 + \delta_t\lambda_1)}{\lambda d^2}.
\end{equation}
In a similar fashion we obtain the estimate of $\sum_{j=1}^{10} \vert \beta_i^j \vert$ for $3 \leq i \leq N-3$, as 
\begin{equation}\label{betastar1dequal}
\sum_{j=1}^{10} \vert \beta_i^j \vert < \beta^* = \frac{12(1 + \delta_t\lambda_1)}{\delta_t\lambda\lambda_2 d^2}.
\end{equation}
Now if we take 
\begin{equation}\label{dstar1dequal}
d > d^* = \max \left\{\sqrt{\frac{12(1 + \delta_t\lambda_1)}{\lambda}}, \sqrt{\frac{12(1 + \delta_t\lambda_1)}{\delta_t\lambda\lambda_2}} \right\},
\end{equation}
 then $\parallel \mathbb{T} \parallel_{\infty}$ is strictly less than one. The same expression of $d$ in \eqref{dstar1dequal} works for the rows $i =1, 2, N-2, N-1$. So we get the convergence. It remains now to prove the estimates. For $3 \leq i \leq N-3$, from the iteration matrix \eqref{Iterationmatrix}, we have
\begin{equation}\label{gik}
\begin{aligned}
g_i^{[k]} & = \alpha_i^1 g_{i-2}^{[k-1]} + \alpha_i^2 h_{i-2}^{[k-1]} + \alpha_i^3 g_{i-1}^{[k-1]} +\alpha_i^4 h_{i-1}^{[k-1]} +\alpha_i^5 g_{i}^{[k-1]}  \\
& + \alpha_i^6 h_{i}^{[k-1]} + \alpha_i^7 g_{i+1}^{[k-1]} + \alpha_i^8 h_{i+1}^{[k-1]} + \alpha_i^9 g_{i+2}^{[k-1]} + \alpha_i^{10} h_{i+2}^{[k-1]}.
\end{aligned}
\end{equation}
Now, if we take max norm on both sides of \eqref{gik} and use triangle inequality we get 
\begin{equation*}
\begin{aligned}
\parallel g_i^{[k]} \parallel_{\infty} \leq  & \left(\sum_{j=1}^{10} \vert \alpha_i^j \vert\right) \max_{1 \leq i \leq N-1}\left\{\parallel g_i^{[k-1]} \parallel_{\infty}, \parallel h_i^{[k-1]} \parallel_{\infty}\right\}\\
< &  \alpha^* \max_{1 \leq i \leq N-1}\left\{\parallel g_i^{[k-1]} \parallel_{\infty}, \parallel h_i^{[k-1]} \parallel_{\infty}\right\}.
\end{aligned}
\end{equation*}
Similarly we obtain for $3 \leq i \leq N-3$
\[ 
\parallel h_i^{[k]} \parallel_{\infty} 
<  \beta^* \max_{1 \leq i \leq N-1}\left\{\parallel g_i^{[k-1]} \parallel_{\infty}, \parallel h_i^{[k-1]} \parallel_{\infty}\right\}. 
\]
One can show that the same bounds also hold for the remaining subdomains $i =1, 2, N-2, N-1$. This completes the theorem.
\end{proof}
\begin{theorem}[Convergence of NN for unequal subdomain]\label{NNunequalthm}
For $\theta=1/4$, the NN algorithm \eqref{MNCH}-\eqref{MNNS} for multiple subdomains with unequal length, satisfying $d_{\min} > d_*$, is convergent. Moreover, we have the following estimates
\[
\max_{1 \leq i \leq N-1} \parallel g_i^{[k]} \parallel_{L^\infty(\Gamma_i)} 
<  (\bar{\alpha})^k \max_{1 \leq i \leq N-1}\left\{\parallel g_i^{[0]} \parallel_{L^\infty(\Gamma_i)}, \parallel h_i^{[0]} \parallel_{L^\infty(\Gamma_i)}\right\},
\]
\[
\max_{1 \leq i \leq N-1}\parallel h_i^{[k]} \parallel_{L^\infty(\Gamma_i)} 
<  (\bar{\beta})^k \max_{1 \leq i \leq N-1}\left\{\parallel g_i^{[0]} \parallel_{L^\infty(\Gamma_i)}, \parallel h_i^{[0]} \parallel_{L^\infty(\Gamma_i)}\right\},
\]
where the expression of $\bar{\alpha}, \bar{\beta}, \bar{d}$ are given in \eqref{alphastar1dunequal}, \eqref{betastar1dunequal}, \eqref{dstar1dunequal} respectively.
\end{theorem}
\begin{proof}
Suppose $d_{\min} = \min_{1\leq i \leq N}d_i$. Define $\sigma_j := \sinh(\xi_j d_{\min}), \gamma_j := \cosh(\xi_j d_{\min})$ for $j = 1, 3$. We show the infinity-norm of $\mathbb{T}$ as given in \eqref{normh} is strictly less than one. Since $\sinh$ is an increasing function and $\sigma_j > \xi_j d_{\min}$ for $j=1, 3$, we obtain
\begin{equation*}
\begin{aligned} 
\vert \alpha_i^1 \vert = & \left| - \frac{\lambda_1}{4\lambda\sigma_{1,i}\sigma_{1,i-1}} + \frac{\lambda_2}{4\lambda\sigma_{3,i}\sigma_{3,i-1}} \right| \\
\leq & \frac{1}{4\lambda} \left(\frac{\lambda_1}{\sigma_1^2} + \frac{\lambda_2}{\sigma_3^2}\right) 
< \frac{1}{4\lambda} \left(\frac{\lambda_1}{\xi_1^2 d_{\min}^2} + \frac{\lambda_2}{\xi_3^2 d_{\min}^2}\right) 
= \frac{1}{2\lambda d_{\min}^2}.
\end{aligned}
\end{equation*}
Similarly we get $\vert \alpha_i^2 \vert < \frac{\delta_t\lambda_1}{2\lambda d_{\min}^2}.$ For the term $\vert \alpha_i^3 \vert$ we have
\begin{equation*}
\begin{aligned}
\vert \alpha_i^3 \vert = & \left| \frac{\lambda_1}{4\lambda}\Upsilon_{1,i}^2 -  \frac{\lambda_2}{4\lambda}\Upsilon_{3,i}^2 \right|\\
\leq & \left(\frac{\lambda_1}{\lambda} \frac{\gamma_1}{\sigma_1^2} + 
\frac{\lambda_2}{\lambda} \frac{\gamma_3}{\sigma_3^2}\right)
< \frac{4}{\lambda d_{\min}^2},
\end{aligned}
\end{equation*}
where the first inequality follows from the triangle inequality and decreasing property of $\coth$ for positive argument as $\xi_j d_{\min} \leq \xi_{j} d_i$, and the second inequality follows from Lemma \ref{NNlemmamultiple}. Similarly we can show that $\vert \alpha_i^4 \vert < \frac{4\delta_t\lambda_1}{\lambda d_{\min}^2}$. We rewrite $\alpha_i^5$ using the identity $\gamma_{j,i}^2 = 1+\sigma_{j,i}^2 $ and get
\begin{equation*}
\begin{aligned}
\alpha_i^5  = &  1 -\frac{\lambda_1}{4\lambda}\left( \frac{\gamma_{1,i}^2}{\sigma_{1,i}^2} + 
\frac{\gamma_{1,i+1}^2}{\sigma_{1,i+1}^2} + 
 2\frac{\gamma_{1,i}\gamma_{1,i+1}}{\sigma_{1,i}\sigma_{1,i+1}} + \frac{1}{\sigma_{1,i}^2} + \frac{1}{\sigma_{1,i+1}^2}\right)
 \\
 & + 
\frac{\lambda_2}{4\lambda}\left( \frac{\gamma_{3,i}^2}{\sigma_{3,i}^2} + \frac{\gamma_{3,i+1}^2}{\sigma_{3,i+1}^2} + 
 2\frac{\gamma_{3,i}\gamma_{3,i+1}}{\sigma_{3,i}\sigma_{3,i+1}} + \frac{1}{\sigma_{3,i}^2} + \frac{1}{\sigma_{3,i+1}^2}\right) \\ 
 = & \frac{1}{2} - \frac{\lambda_1}{4\lambda}\left(  
 2\frac{\gamma_{1,i}\gamma_{1,i+1}}{\sigma_{1,i}\sigma_{1,i+1}} + \frac{2}{\sigma_{1,i}^2} + \frac{2}{\sigma_{1,i+1}^2}\right)
  \\
 & + 
\frac{\lambda_2}{4\lambda}\left(
 2\frac{\gamma_{3,i}\gamma_{3,i+1}}{\sigma_{3,i}\sigma_{3,i+1}} + \frac{2}{\sigma_{3,i}^2} + \frac{2}{\sigma_{3,i+1}^2}\right), 
\end{aligned}
\end{equation*}
This leads further,
\begin{equation*}
\begin{aligned}
\vert \alpha_i^5 \vert = & \left| \frac{1}{2} - \frac{\lambda_1}{4\lambda}\left(  
 2\frac{\gamma_{1,i}\gamma_{1,i+1}}{\sigma_{1,i}\sigma_{1,i+1}} + \frac{2}{\sigma_{1,i}^2} + \frac{2}{\sigma_{1,i+1}^2}\right) +
\frac{\lambda_2}{4\lambda}\left(
 2\frac{\gamma_{3,i}\gamma_{3,i+1}}{\sigma_{3,i}\sigma_{3,i+1}} + \frac{2}{\sigma_{3,i}^2} + \frac{2}{\sigma_{3,i+1}^2}\right) 
  \right|\\
= & 
\left| -\frac{1}{2} + \frac{\lambda_1}{4\lambda}\left(  
 2\frac{\gamma_{1,i}\gamma_{1,i+1}}{\sigma_{1,i}\sigma_{1,i+1}} + \frac{2}{\sigma_{1,i}^2} + \frac{2}{\sigma_{1,i+1}^2}\right) -
\frac{\lambda_2}{4\lambda}\left(
 2\frac{\gamma_{3,i}\gamma_{3,i+1}}{\sigma_{3,i}\sigma_{3,i+1}} + \frac{2}{\sigma_{3,i}^2} + \frac{2}{\sigma_{3,i+1}^2}\right) 
  \right|\\
 < &
\left| -\frac{1}{2} + \frac{\lambda_1}{4\lambda}\left(  
 2\frac{\gamma_{3,i}\gamma_{3,i+1}}{\sigma_{3,i}\sigma_{3,i+1}} + \frac{2}{\sigma_{1,i}^2} + \frac{2}{\sigma_{1,i+1}^2}\right) -
\frac{\lambda_2}{4\lambda}\left(
 2\frac{\gamma_{3,i}\gamma_{3,i+1}}{\sigma_{3,i}\sigma_{3,i+1}} + \frac{2}{\sigma_{3,i}^2} + \frac{2}{\sigma_{3,i+1}^2}\right) 
  \right| \\
\leq &
\left|\left(  -\frac{1}{2} + \frac{1}{2} \frac{\gamma_{3,i}\gamma_{3,i+1}}{\sigma_{3,i}\sigma_{3,i+1}} \right) +  \frac{\lambda_1}{\lambda}\frac{1}{\sigma_{1}^2} + \frac{\lambda_2}{\lambda}\frac{1}{\sigma_{3}^2}\right|\\
< &
\left|\left(  -\frac{1}{2} + \frac{1}{2} \frac{\gamma_{1}^2}{\sigma_{1}^2} \right)\right| + \frac{\lambda_1}{\lambda}\frac{1}{\xi_1^2 d_{\min}^2} + \frac{\lambda_2}{\lambda}\frac{1}{\xi_{3}^2 d_{\min}^2}\\
< & 
\left(\frac{1}{2\sigma_1^2} + \frac{2}{\lambda d_{\min}^2} \right) <  \frac{1}{d_{\min}^2}\left(\frac{1}{2\lambda_1} + \frac{2}{\lambda}\right),
\end{aligned}
\end{equation*}
where in the first and third inequalities we have used decreasing property of $\coth$ for positive argument. Similarly we have $\vert\alpha_i^6\vert < \frac{\delta_t\lambda_1}{\lambda d_{\min}^2}$.  Similar to $\vert\alpha_i^3\vert, \vert\alpha_i^4\vert$ we have $\vert\alpha_i^7\vert < \frac{4}{\lambda d_{\min}^2}, \vert\alpha_i^8\vert < \frac{4\delta_t\lambda_1}{\lambda d_{\min}^2}$,  and similar to $\vert\alpha_i^1\vert, \vert\alpha_i^2\vert$ we have $\vert\alpha_i^9\vert < \frac{1}{2\lambda d_{\min}^2}, \vert\alpha_i^{10}\vert < \frac{\delta_t\lambda_1}{2\lambda d_{\min}^2}$.  
Therefore for $3 \leq i \leq N-3$ we obtain,
\begin{equation}\label{alphastar1dunequal}
\sum_{j=1}^{10} \vert \alpha_i^j \vert < \bar{\alpha} = \frac{1}{d_{\min}^2}\left(\frac{11}{\lambda} + \frac{11\delta_t\lambda_1}{\lambda} +\frac{1}{2\lambda_1} \right).
\end{equation} 
Similarly we obtain for $3 \leq i \leq N-3$, 
\begin{equation}\label{betastar1dunequal}
\sum_{j=1}^{10} \vert \beta_i^j \vert < \bar{\beta} = \frac{1}{d_{\min}^2}\left(\frac{11}{\delta_t\lambda\lambda_2} + \frac{11\lambda_1}{\lambda\lambda_2} +\frac{1}{2\lambda_1} \right).
\end{equation}
Now if we take 
\begin{equation}\label{dstar1dunequal}
d_{\min} > \bar{d} = \max \left\{\sqrt{\frac{11}{\lambda} + \frac{10\delta_t\lambda_1}{\lambda} +\frac{1}{2\lambda_1}}, \sqrt{\frac{11}{\delta_t\lambda\lambda_2} + \frac{11\lambda_1}{\lambda\lambda_2} +\frac{1}{2\lambda_1}} \right\},
\end{equation}
 then the $\parallel \mathbb{T} \parallel_{\infty}$ is strictly less than one. It is easy to show that same bounds also hold for $i =1, 2, N-2, N-1$. This proves the convergence. \\
By following the 2nd part of the proof of Theorem \ref{NNequalthm}, we get similar estimates with the constants $\bar{\alpha}, \bar{\beta}$. 
\end{proof}

\subsection{NN for multiple subdomain in 2D}
We now analyse the NN algorithm \eqref{MNCH} - \eqref{MNNS} for the two-dimensional CH equation. We decompose the domain of our interest $\Omega = (0, L)\times (0, l)$ into strips of the form $\Omega_i:=(x_{i-1}, x_i)\times(0,l)$ for $i=1,\ldots,N$, with subdomain width $d_{i}:=x_{i}-x_{i-1}$. We perform a Fourier sine transform along $y$-direction to reduce the original problem into a collection of one-dimensional problems. Expanding the solution $u_i^{[k]}, v_i^{[k]}, \phi_{i}^{[k]}, \psi_{i}^{[k]}$ for $i=1,\ldots,N$ in a Fourier sine series along the $y$-direction yields
\[
\begin{array}{cc}
u_{i}^{[k]}(x,y)=\sum_{m\geq1}\hat u_{i}^{[k]}(x,m)\sin(\frac{m\pi y}{l}),\; &
v_{i}^{[k]}(x,y)=\sum_{m\geq1}\hat v_{i}^{[k]}(x,m)\sin(\frac{m\pi y}{l}),\\
\phi_{i}^{[k]}(x,y)=\sum_{m\geq1}\hat \phi_{i}^{[k]}(x,m)\sin(\frac{m\pi y}{l}),\;&
\psi_{i}^{[k]}(x,y)=\sum_{m\geq1}\hat \psi_{i}^{[k]}(x,m)\sin(\frac{m\pi y}{l}).
\end{array}
\]
After a Fourier sine transform, the NN algorithm \eqref{MNCH} - \eqref{MNNS} for the error equation in 2D becomes 
\begin{flalign*}
  & \begin{aligned} & \begin{cases}
\hat A \hat E_i^{[k]} & = 0, \quad \quad \text{in} \,\ \Omega_i, \\
\hat E_i^{[k]} & = \begin{bmatrix} 
\hat g_{i-1}^{[k-1]} \\
\hat h_{i-1}^{[k-1]} 
\end{bmatrix}, \quad \text{on} \,\ \Gamma_{i-1}, \\
\hat E_i^{[k]} & = \begin{bmatrix} 
\hat g_i^{[k-1]} \\
\hat h_i^{[k-1]}
\end{bmatrix}, \quad \text{on} \,\ \Gamma_{i}, \\
\end{cases}\\
  \MoveEqLeft[2]\text{}
  \end{aligned}
  & &
  \begin{aligned}
      & \begin{cases}
\hat A \hat F_i^{[k]} & = 0, \quad \quad \text{in} \,\ \Omega_i \\
\frac{\partial}{\partial x}\ \hat F_i^{[k]} & = \frac{\partial}{\partial x}\big[\hat E_{i-1}^{[k]} - \hat E_i^{[k]}\big], \quad \text{on} \,\ \Gamma_{i-1}, \\
\frac{\partial}{\partial x}\ \hat F_i^{[k]} & = \frac{\partial}{\partial x}\big[\hat E_{i}^{[k]} - \hat E_{i+1}^{[k]}\big], \quad \text{on} \,\ \Gamma_i ,\\
 \end{cases} \\
    \MoveEqLeft[-5]\text{}
  \end{aligned}
\end{flalign*}
except for the first and last subdomains, which are handled differently as in 1D case. The interface values for the next step are then updated as 
\[ 
 \begin{bmatrix} 
\hat g_i^{[k]} \\
\hat h_i^{[k]} 
\end{bmatrix} 
= 
\begin{bmatrix} 
\hat g_i^{[k-1]} \\
\hat h_i^{[k-1]} 
\end{bmatrix}
-
\theta
\big[\hat F_{i}^{[k]} - \hat F_{i+1}^{[k]}\big]_{\big|_{\Gamma_i}},
\]
where
{\footnotesize
\[
\hat A = 
\begin{bmatrix} 
1 & -\delta_t(\frac{d^2}{dx^2} - p_m^2) \\
\epsilon^2(\frac{d^2}{dx^2}-p_m^2) - c^2 & 1 
\end{bmatrix},\; \text{with}\; p_m^2=\frac{\pi^2 m^2}{l^2} \;
 \text{and}\;
\hat E_i^{[k]} = \begin{bmatrix} 
\hat u_i^{[k]}(x) \\
\hat v_i^{[k]}(x)
\end{bmatrix},
\hat F_i^{[k]} = \begin{bmatrix} 
\hat \phi_i^{[k]}(x) \\
\hat \psi_i^{[k]}(x)
\end{bmatrix}.
\]
\par}
With similar argument as in 1D case, we get the recurrence relation in 2D as: 
\begin{equation}\label{iternmatrix2dmul}
 \begin{bmatrix} 
\hat{\textbf{g}}^{[k]} \\
\hat{\textbf{ h}}^{[k]}
\end{bmatrix} 
=  \hat{\mathbb{T}}
\begin{bmatrix} 
\hat{\textbf{ g}}^{[k-1]} \\
\hat{\textbf{ h}}^{[k-1]} 
\end{bmatrix}
\end{equation}
where $
 \begin{bmatrix} 
\hat{\textbf{g}}^{[k]} & \hat{\textbf{ h}}^{[k]}
\end{bmatrix}^\mathsf{T} := [\hat g_1^{[k]}, \hat h_1^{[k]}, \hat g_2^{[k]}, \hat h_2^{[k]}, ..., \hat g_{N-1}^{[k]}, \hat h_{N-1}^{[k]}]^\mathsf{T} $
and the iteration matrix $ \hat{\mathbb{T}} \in\mathbb{R}^{2N-2\times 2N-2}$ has the same form as in \eqref{Iterationmatrix}, except that the elements have Fourier symbol in it, as  $\xi_{1,2}, \xi_{3,4}$ are modified as 
\[ 
\xi_{1,2} = \pm\sqrt{\lambda_1 + p_m^2},\; 
\xi_{3,4} = \pm\sqrt{\lambda_2 + p_m^2},  
\]
where $\lambda_{1,2}$ are exactly as defined earlier. Note that $\xi_i$'s are function of Fourier variable $m$ for $i=1,\cdots,4$.   
We now prove the convergence result for NN method in 2D for multiple subdomain for $\delta_t > \frac{4\epsilon^2}{c^4}$.
\begin{theorem}[Convergence of NN in 2D]\label{NNequal2dthm}
For $\theta=1/4$, the NN \eqref{MNCH}-\eqref{MNNS} algorithm with equal subdomain width $d$, satisfying $d > d_e^*$, is convergent. Moreover, we have the following estimates
\[
\max_{1 \leq i \leq N-1} \parallel g_i^{[k]} \parallel_{L^2(\Gamma_i)} 
<  (\sqrt{\alpha_e^*})^k \max_{1 \leq i \leq N-1}\left\{\parallel g_i^{[0]} \parallel_{L^2(\Gamma_i)}, \parallel h_i^{[0]} \parallel_{L^2(\Gamma_i)}\right\},
\]
\[
\max_{1 \leq i \leq N-1}\parallel h_i^{[k]} \parallel_{L^2(\Gamma_i)} 
<  (\sqrt{\beta_e^*})^k \max_{1 \leq i \leq N-1}\left\{\parallel g_i^{[0]} \parallel_{L^2(\Gamma_i)}, \parallel h_i^{[0]} \parallel_{L^2(\Gamma_i)}\right\},
\]
where the expression of $d_e^*, \alpha_e^*, \beta_e^*$ are given in \eqref{dstar2dequal}, \eqref{alphastar2dequal}, \eqref{betastar2dequal} respectively.
\end{theorem}
\begin{proof}
For equal-spaced subdomains we have $d_i = d$ for $i=1,\ldots,N$, so that $\sigma_{j,i} = \sigma_{j}, \gamma_{j, i}=\gamma_{j}$ for all $i$ and for $j = 1, 3$. Consider the infinity-norm of $\hat{\mathbb{T}}$ given by
\begin{equation}\label{normh2d}
\parallel \hat{\mathbb{T}} \parallel_{\infty} = \max_{1 \leq i \leq N-1} \left\{ \sum_{j=1}^{S_i} \vert \alpha_i^j(m) \vert, \sum_{j=1}^{S_i} \vert \beta_i^j(m) \vert \right\}, 
\end{equation}
where $S_i$'s are defined earlier. 
We follow the proof of Theorem \ref{NNequalthm} to estimate $\sum_{j=1}^{10} \vert \alpha_i^j(m)\vert$
for $3 \leq i \leq N-3$,  and get
\[
\vert \alpha_i^1(m)\vert < \frac{\lambda_1}{2\lambda}\frac{1}{\sigma_3^2}, \; 
\vert \alpha_i^2(m)\vert < \frac{\delta_t\lambda_1\lambda_2}{2\lambda}\frac{1}{\sigma_3^2}, \; 
\vert \alpha_i^3(m)\vert < \frac{\lambda_1}{\lambda}\frac{4}{\sigma_3^2}, \; 
\] 
\[
\vert \alpha_i^4(m)\vert < \frac{\delta_t\lambda_1\lambda_2}{\lambda}\frac{4}{\sigma_3^2}, \; 
\vert \alpha_i^5(m)\vert < \frac{\lambda_1}{\lambda}\frac{3}{\sigma_3^2}, \; 
\vert \alpha_i^6(m)\vert < \frac{\delta_t\lambda_1\lambda_2}{\lambda}\frac{3}{\sigma_3^2}.
\]
As $m\mapsto {\sigma_3(p_m)}$ is monotonically increasing, we have the estimate of $\sum_{j=1}^{10} \vert \alpha_i^j(m) \vert$ for $3 \leq i \leq N-3$, as 
\[ 
\sum_{j=1}^{10} \vert \alpha_i^j(m) \vert <  \frac{12}{\sigma_3^2(p_1)}\left( \frac{\lambda_1}{\lambda} +  \frac{\delta_t\lambda_1\lambda_2}{\lambda}\right) = \frac{c_{\alpha}}{\sigma_3^2(p_1)}.
\]
Similarly we obtain the estimate of $\sum_{j=1}^{10} \vert \beta_i^j(m) \vert$ for $3 \leq i \leq N-3$, as 
\[ 
\sum_{j=1}^{10} \vert \beta_i^j(m) \vert < \frac{12}{\sigma_3^2(p_1)}\left( \frac{\lambda_1}{\lambda} +  \frac{1}{\delta_t\lambda}\right)= \frac{c_{\beta}}{\sigma_3^2(p_1)}.
\]
Now if we take 
\begin{equation}\label{dstar2dequal}
d > d_e^* = \max \Big\{\sinh^{-1}(\sqrt{c_{\alpha}}),  \sinh^{-1}(\sqrt{c_{\beta}}) \Big\} \frac{1}{\xi_3(p_1)},
\end{equation}
 then $\parallel \hat{\mathbb{T}} \parallel_{\infty}$ becomes strictly less than one. One can show that the same $d$ also works for estimating the remaining rows $i =1, 2, N-2, N-1$. This proves  the convergence. It remains now to get the estimates. 
For $3 \leq i \leq N-3$ from the iteration matrix \eqref{iternmatrix2dmul}, we have
\begin{equation}\label{gikm}
\begin{aligned}
\hat g_i^{[k]}(m) = &\alpha_i^1(m) \hat g_{i-2}^{[k-1]}(m) + \alpha_i^2(m) \hat h_{i-2}^{[k-1]}(m) + \alpha_i^3(m) \hat g_{i-1}^{[k-1]}(m) \\
 & +\alpha_i^4(m) \hat h_{i-1}^{[k-1]}(m) +\alpha_i^5(m) \hat g_{i}^{[k-1]}(m) +\alpha_i^6(m) \hat h_{i}^{[k-1]}(m) \\
& + \alpha_i^7(m) \hat g_{i+1}^{[k-1]}(m) + \alpha_i^8(m) \hat h_{i+1}^{[k-1]}(m) + \alpha_i^9(m) \hat g_{i+2}^{[k-1]}(m) \\ 
& + \alpha_i^{10}(m) \hat h_{i+2}^{[k-1]}(m),
\end{aligned}
\end{equation}
for each $m \geq 1.$ We define $\Lambda_i^{[k]} = \{ \vert \hat g_i^{[k]}(m) \vert\}_{m\geq 1} \in l_2$ and $\Xi_i^{[k]} = \{ \vert \hat h_i^{[k]}(m) \vert\}_{m\geq 1} \in l_2$. Using Parseval-Plancherel identity we get
$\parallel g_i^{[k]} \parallel_{L^2}^2 = \frac{l}{2} \sum_{m=1}^{\infty} \vert \hat g_i^{[k]}(m) \vert ^2 = \frac{l}{2} \parallel\Lambda_i^{[k]} \parallel_2^2$ and $\parallel h_i^{[k]} \parallel_{L^2}^2 = \frac{l}{2} \sum_{m=1}^{\infty} \vert \hat h_i^{[k]}(m) \vert ^2 = \frac{l}{2} \parallel\Xi_i^{[k]} \parallel_2^2$. Using the above identities and the triangle inequality in $l_2$, we have for $i = 3,\cdots,N-3$
\begin{equation*}
\begin{aligned}
\parallel g_i^{[k]} \parallel_{L^2}^2 =  &  \frac{l}{2} \sum_{m=1}^{\infty} \Big( \alpha_i^1(m) \hat g_{i-2}^{[k-1]}(m) + \alpha_i^2(m) \hat h_{i-2}^{[k-1]}(m) + \alpha_i^3(m) \hat g_{i-1}^{[k-1]}(m) +\alpha_i^4(m) \hat h_{i-1}^{[k-1]}(m) \\
 &  +\alpha_i^5(m) \hat g_{i}^{[k-1]}(m) +\alpha_i^6(m) \hat h_{i}^{[k-1]}(m)\alpha_i^7(m) \hat g_{i+1}^{[k-1]}(m) + \alpha_i^8(m) \hat h_{i+1}^{[k-1]}(m)  \\
& +  \alpha_i^9(m) \hat g_{i+2}^{[k-1]}(m)+ \alpha_i^{10}(m) \hat h_{i+2}^{[k-1]}(m) \Big)^2 ,\\
\leq & C_1 \frac{l}{2}\sum_{m=1}^{\infty} \Big( \frac{1}{2}\vert\hat g_{i-2}^{[k-1]}(m)\vert + \frac{1}{2}\vert\hat h_{i-2}^{[k-1]}(m)\vert + 4\vert \hat g_{i-1}^{[k-1]}(m)\vert+ 4\vert\hat h_{i-1}^{[k-1]}(m)\vert+3\vert\hat g_{i}^{[k-1]}(m) \vert \\
 & + 3\vert\hat h_{i}^{[k-1]}(m)\vert+ 4\vert\hat g_{i+1}^{[k-1]}(m) \vert+ 4\vert \hat h_{i+1}^{[k-1]}(m)\vert +  \frac{1}{2}\vert\hat g_{i+2}^{[k-1]}(m)\vert + \frac{1}{2}\vert\hat h_{i+2}^{[k-1]}(m)\vert  \Big)^2 ,\\
< & C_1 \frac{l}{2} \Big(\parallel \frac{1}{2}\Lambda_{i-2}^{[k-1]} + \frac{1}{2}\Xi_{i-2}^{[k-1]} +  4\Lambda_{i-1}^{[k-1]} + 4\Xi_{i-1}^{[k-1]} + 3\Lambda_i^{[k-1]} + 3\Xi_i^{[k-1]} +4\Lambda_{i+1}^{[k-1]} \\
& + 4\Xi_{i+1}^{[k-1]} + \frac{1}{2}\Lambda_{i+2}^{[k-1]} + \frac{1}{2}\Xi_{i+2}^{[k-1]} \parallel_2 ^2\Big),\\
\leq & C_1 \frac{l}{2} \Big(\frac{1}{2}\parallel \Lambda_{i-2}^{[k-1]} \parallel_2 + \frac{1}{2}\parallel\Xi_{i-2}^{[k-1]} \parallel_2+  4\parallel\Lambda_{i-1}^{[k-1]}\parallel_2 + 4\parallel\Xi_{i-1}^{[k-1]}\parallel_2 + 3\parallel\Lambda_i^{[k-1]} \parallel_2  \\
 & + 3\parallel\Xi_i^{[k-1]} \parallel_2 +4\parallel\Lambda_{i+1}^{[k-1]} \parallel_2+ 4\parallel\Xi_{i+1}^{[k-1]}\parallel_2 + \frac{1}{2}\parallel\Lambda_{i+2}^{[k-1]} \parallel_2+ \frac{1}{2}\parallel\Xi_{i+2}^{[k-1]} \parallel_2 \Big)^2 ,\\
 \leq & \alpha_e^* \max_{1 \leq i \leq N-1} \left\{\parallel g_i^{[k-1]} \parallel_{L^2}^2, \parallel h_i^{[k-1]} \parallel_{L^2}^2 \right\}
\end{aligned}
\end{equation*}
where 
\begin{equation}\label{alphastar2dequal}
\alpha_e^*=576C_1
\end{equation}
 with $C_1 =\frac{5}{\sigma_3^4(p_1)}\left[\left( \frac{\lambda_1}{\lambda}\right)^2 +  \left(\frac{\delta_t\lambda_1\lambda_2}{\lambda}\right)^2\right]$. Similarly we obtain  for $3 \leq i \leq N-3$
\[ 
\parallel h_i^{[k]} \parallel_{L^2}^2 
<  \beta_e^* \max_{1 \leq i \leq N-1}\left\{\parallel g_i^{[k-1]} \parallel_{L^2}^2, \parallel h_i^{[k-1]} \parallel_{L^2}^2\right\},
\]
where 
\begin{equation}\label{betastar2dequal}
\beta_e^* = 576C_2
\end{equation}
 with $C_2 = \frac{5}{\sigma_3^4(p_1)}\left[\left( \frac{\lambda_1}{\lambda}\right)^2 +  \left(\frac{1}{\delta_t\lambda}\right)^2\right]$. The same estimate also holds for the remaining subdomains $i =1, 2, N-2, N-1$, and hence we get the result. 
\end{proof}
\begin{theorem}[Convergence of NN for unequal subdomain in 2D]\label{NNunequal2dthm}
For $\theta=1/4$, the NN algorithm \eqref{MNCH}-\eqref{MNNS} for multiple subdomains with unequal subdomain width, satisfying $d_{\min} > d_u^*$, is convergent. Moreover, we have the following estimates
\[
\max_{1 \leq i \leq N-1} \parallel g_i^{[k]} \parallel_{L^2(\Gamma_i)} 
<  (\sqrt{\alpha_u^*})^k \max_{1 \leq i \leq N-1}\left\{\parallel g_i^{[0]} \parallel_{L^2(\Gamma_i)}, \parallel h_i^{[0]} \parallel_{L^2(\Gamma_i)}\right\},
\]
\[
\max_{1 \leq i \leq N-1}\parallel h_i^{[k]} \parallel_{L^2(\Gamma_i)} 
<  (\sqrt{\beta_u^*})^k \max_{1 \leq i \leq N-1}\left\{\parallel g_i^{[0]} \parallel_{L^2(\Gamma_i)}, \parallel h_i^{[0]} \parallel_{L^2(\Gamma_i)}\right\},
\]
where the expression of $d_u^*, \alpha_u^*, \beta_u^*$ are given in \eqref{dstar2dunequal}, \eqref{alphabetastar2dunequal} respectively.
\end{theorem}
\begin{proof}
For unequal subdomains we define $\sigma_j := \sinh(\xi_j d_{\min}), \gamma_j := \cosh(\xi_j d_{\min})$ for $j = 1, 3$, where $d_{\min} = \min_{1\leq i \leq N}d_i$. We are going to estimate $\sum_{j=1}^{10} \vert \alpha_i^j(m) \vert$ for $3 \leq i \leq N-3$. Similar to the proof of Theorem \ref{NNequalthm} we have the following estimates 
\[
\vert \alpha_i^1(m)\vert < \frac{\lambda_1}{2\lambda}\frac{1}{\sigma_3^2}, \; 
\vert \alpha_i^2(m)\vert < \frac{\delta_t\lambda_1\lambda_2}{2\lambda}\frac{1}{\sigma_3^2}, \; 
\vert \alpha_i^3(m)\vert < \frac{\lambda_1}{\lambda}\frac{4}{\sigma_3^2}, \; 
\] 
\[
\vert \alpha_i^4(m)\vert < \frac{\delta_t\lambda_1\lambda_2}{\lambda}\frac{4}{\sigma_3^2}, \; 
\vert \alpha_i^5(m)\vert < \left(\frac{2\lambda_1}{\lambda} + \frac{1}{2}\right)\frac{1}{\sigma_3^2}, \; 
\vert \alpha_i^6(m)\vert < \frac{\delta_t\lambda_1\lambda_2}{\lambda}\frac{1}{\sigma_3^2}.
\]
Likewise $\vert\alpha_i^3\vert, \vert\alpha_i^4\vert$ we have the estimates for $\vert\alpha_i^7\vert, \vert\alpha_i^8\vert$ as $\vert\alpha_i^7\vert < \frac{\lambda_1}{\lambda}\frac{4}{\sigma_3^2}, \vert\alpha_i^8\vert < \frac{\delta_t\lambda_1\lambda_2}{\lambda}\frac{4}{\sigma_3^2}$, and similar to $\vert\alpha_i^1\vert, \vert\alpha_i^2\vert$ we have $\vert\alpha_i^9\vert < \frac{\lambda_1}{2\lambda}\frac{1}{\sigma_3^2}, \vert\alpha_i^{10}\vert < \frac{\delta_t\lambda_1\lambda_2}{2\lambda}\frac{1}{\sigma_3^2}.$ 
We now have the estimate of $\sum_{j=1}^{10} \vert \alpha_i^j(m) \vert$ for $3 \leq i \leq N-3$, as 
\[ 
\sum_{j=1}^{10} \vert \alpha_i^j(m) \vert <  \frac{1}{\sigma_3^2(p_1)}\left( \frac{11\lambda_1}{\lambda} +  \frac{10\delta_t\lambda_1\lambda_2}{\lambda} + \frac{1}{2}\right) = \frac{c_{\alpha}^*}{\sigma_3^2(p_1)},
\]
using $\sigma_3(p_m)$ being an increasing function in $m$.
Similarly we obtain the estimate of $\sum_{j=1}^{10} \vert \beta_i^j(m) \vert$ for $3 \leq i \leq N-3$, as 
\[ 
\sum_{j=1}^{10} \vert \beta_i^j(m) \vert < \frac{1}{\sigma_3^2(p_1)}\left( \frac{11\lambda_1}{\lambda} +  \frac{10}{\delta_t\lambda} + \frac{1}{2}\right) = \frac{c_{\beta}^*}{\sigma_3^2(p_1)}.
\]
When
\begin{equation}\label{dstar2dunequal}
d_{\min} > d_u^* = \max \Big\{\sinh^{-1}(\sqrt{c_{\alpha}^*}),  \sinh^{-1}(\sqrt{c_{\beta}^*}) \Big\} \frac{1}{\xi_3(p_1)}, 
\end{equation}
$\parallel \hat{\mathbb{T}} \parallel_{\infty}$ becomes strictly less than one. The same estimates works for $i =1, 2, N-2, N-1$.
Similar to the second part of Theorem \ref{NNequal2dthm}, we get the estimate with 
\begin{equation}\label{alphabetastar2dunequal}
\alpha_u^* = 462.25C_1,\quad \beta_u^* = 462.25C_2, 
\end{equation}
 where $C_1, C_2$ are as defined earlier. This completes the proof.
\end{proof}
\begin{remark}
The convergence behaviour of NN method in multisubdomain setting for $\delta_t < \frac{4\epsilon^2}{c^4}$ can be proven in a similar way as in the case of $\delta_t > \frac{4\epsilon^2}{c^4}$ in 1D and 2D.
\end{remark}

\section{Numerical Illustration}\label{Section4}
In this section we present the numerical experiments for the DN \& NN algorithm for the CH equation \eqref{mixedCH}. We discretize the CH equation using the centered finite difference in space and the backward Euler in time with the linearization described in \eqref{linearCH}. The parameter $\epsilon $ is taken as $0.01$, except otherwise stated. The iterations start from a random initial guess and stop as the error $\parallel u-u^{[k]}\parallel_{L^{\infty}}$ in 1D and $\parallel u-u^{[k]}\parallel_{L^{2}}$ in 2D reaches a tolerance of $10^{-6}$, where $u$ is the discrete monodomain solution and $u^{[k]}$ is the discrete DN or NN solution at $k-$th iteration. The phase separation is rapid in time, and consequently small time steps should be taken. We then choose in this case $\delta_t=10^{-6}$, which results in the case $\delta_t< \frac{4\epsilon^2}{c^4}$. For the CH equation the phase coarsening stage is slow in time, and so one chooses relatively large time steps to reduce the total amount of computation. We choose $\delta_t = 10^{-3}$, which results in the case $\delta_t > \frac{4\epsilon^2}{c^4}$. In the following section we give numerical results by taking the above consideration.
\subsection{Numerics of DN \& NN method in 1D}
First we have given convergence results in terms of iteration count for DN method in equal subdomain case in Table \ref{dnequaltable_1d} by considering the domain $\Omega=(0, 1)$ and partitioned into $\Omega_1=(0, 1/2)$ and $\Omega_2 = (1/2, 1)$. 
For unequal subdomain: first we consider Neumann subdomain is larger than Dirichlet subdomain by choosing $a=1, b=2$, i.e., the domain $\Omega=(1, 2)$
and split it into $\Omega_1= (1, 1.4)$ and $\Omega_2=(1.4, 2)$, corresponding to Theorem \ref{NLD1d}, see Table \ref{dnunequal_1d_nld}.
And secondly we consider Dirichlet subdomain is larger than Neumann subdomain by choosing $a=1.5, b=1$, i.e., the domain $\Omega=(-1.5, 1)$
and split it into $\Omega_1= (-1.5, 0)$ and $\Omega_2=(0, 1)$, corresponding to Theorem \ref{DLN}, see Table \ref{dnunequal_1d_dln}. In Figure \ref{DNerrorNLD}, we compare theoretical error estimates given in Theorem \ref{NLD1d} with numerical error in the case of $b>a$. And in Figure \ref{DNerrorDLN}, we plot the theoretical estimates of error bound presented in Theorem \ref{DLN} to compare with numerical error in the case of $a>b$.
For experiment of NN method we have taken the spatial domain $\Omega=(0, 20)$. In Table \ref{NNmany_equal_1d} and \ref{NNmany_unequal_1d}, we compare the iteration number required for NN method to converge for the parameter $\theta=1/4$, by varying  number of subdomains ('sd'), mesh size $h$ and time steps $\delta_t$.
In Figure \ref{NNlongtimemultiple} we compare the numerical error behaviour of NN method with our theoretical error estimates presented in Theorem \ref{NNequalthm} for multiple subdomain of equal length. 

\begin{table}
\centering
\begin{tabular}{|p{.7cm}|p{1.2cm}|c|c|c|c|c|c|c|c|c|} \hline
$\delta_t$& \diagbox{$h$}{$\theta$} & 0.1 & 0.2 & 0.3 & 0.4 & 0.5 & 0.6 &0.7 &0.8&0.9\\ \hline\hline
$10^{-6}$ &
\begin{tabular}{c}  1/64 \\ 1/128\\ 1/256 \\1/512 \\  
\end{tabular} &
\begin{tabular}{c}  81 \\82\\84\\85\\
\end{tabular} &
\begin{tabular}{c}  36 \\36\\37\\37
\end{tabular} &
\begin{tabular}{c} 20\\20\\21\\21\\
\end{tabular}&
\begin{tabular}{c}  11\\12\\12\\12\\
\end{tabular} &
\begin{tabular}{c}  2\\2\\2\\2\\
\end{tabular} &
\begin{tabular}{c}  11\\12\\12\\12\\
\end{tabular} &
\begin{tabular}{c}  20\\20\\21\\21\\
\end{tabular} &
\begin{tabular}{c}  36\\36\\36\\36\\
\end{tabular} &
\begin{tabular}{c}  81\\82\\84\\85\\
\end{tabular} \\ \hline
$10^{-3}$ &
\begin{tabular}{c} 1/64 \\ 1/128\\ 1/256 \\1/512 \\
\end{tabular} &
\begin{tabular}{c}  84\\86\\88\\89\\
\end{tabular} &
\begin{tabular}{c}  37\\38\\38\\38\\
\end{tabular} &
\begin{tabular}{c}  22\\21\\22\\22\\
\end{tabular}&
\begin{tabular}{c}  13\\12\\13\\13\\
\end{tabular} &
\begin{tabular}{c}  2\\2\\2\\2\\
\end{tabular} &
\begin{tabular}{c}  13\\12\\13\\13\\
\end{tabular} &
\begin{tabular}{c}  21\\21\\22\\22\\
\end{tabular} &
\begin{tabular}{c}  37\\38\\38\\39\\
\end{tabular} &
\begin{tabular}{c}  84\\86\\88\\89\\
\end{tabular} \\ \hline
\end{tabular}
\caption{Number of iteration compared for DN method for two equal subdomain with short and large time step.}
\label{dnequaltable_1d}
\end{table}

\begin{table}
\centering
\begin{tabular}{|p{.7cm}|p{1.2cm}|c|c|c|c|c|c|c|c|c|} \hline
$\delta_t$& \backslashbox{$h$}{$\theta$} & 0.1 & 0.2 & 0.3 & 0.4 & 0.5 & 0.6 &0.7 &0.8&0.9\\ \hline\hline
$10^{-6}$ &
\begin{tabular}{c}  1/64 \\ 1/128\\ 1/256 \\1/512 \\  
\end{tabular} &
\begin{tabular}{c}  75\\76\\77\\79\\
\end{tabular} &
\begin{tabular}{c}  35 \\35\\36\\37
\end{tabular} &
\begin{tabular}{c} 20\\20\\20\\22\\
\end{tabular}&
\begin{tabular}{c}  12\\12\\12\\12\\
\end{tabular} &
\begin{tabular}{c}  2\\2\\2\\2\\
\end{tabular} &
\begin{tabular}{c}  12\\12\\12\\12\\
\end{tabular} &
\begin{tabular}{c}  20\\20\\20\\22\\
\end{tabular} &
\begin{tabular}{c}  37\\38\\39\\39\\
\end{tabular} &
\begin{tabular}{c}  84\\86\\87\\89\\
\end{tabular} \\ \hline
$10^{-3}$ &
\begin{tabular}{c} 1/64 \\ 1/128\\ 1/256 \\1/512 \\
\end{tabular} &
\begin{tabular}{c}  74\\82\\83\\84\\
\end{tabular} &
\begin{tabular}{c}  34\\39\\39\\39\\
\end{tabular} &
\begin{tabular}{c}  22\\22\\23\\23\\
\end{tabular}&
\begin{tabular}{c}  12\\13\\12\\13\\
\end{tabular} &
\begin{tabular}{c}  3\\3\\3\\5\\
\end{tabular} &
\begin{tabular}{c}  12\\13\\12\\13\\
\end{tabular} &
\begin{tabular}{c}  22\\22\\23\\23\\
\end{tabular} &
\begin{tabular}{c}  39\\45\\45\\44\\
\end{tabular} &
\begin{tabular}{c}  84\\87\\88\\91\\
\end{tabular} \\ \hline
\end{tabular}
\caption{Number of iteration compared of DN with Neumann subdomain larger than Dirichlet subdomain for short and large time step.}
\label{dnunequal_1d_nld}
\end{table}

\begin{table}
\centering
\begin{tabular}{|p{.7cm}|p{1.2cm}|c|c|c|c|c|c|c|c|c|} \hline
$\delta_t$& \backslashbox{$h$}{$\theta$} & 0.1 & 0.2 & 0.3 & 0.4 & 0.5 & 0.6 &0.7 &0.8&0.9\\ \hline\hline
$10^{-6}$ &
\begin{tabular}{c}  1/64 \\ 1/128\\ 1/256 \\1/512 \\  
\end{tabular} &
\begin{tabular}{c}  65\\66\\67\\67\\
\end{tabular} &
\begin{tabular}{c}  35 \\35\\36\\37
\end{tabular} &
\begin{tabular}{c} 20\\20\\20\\22\\
\end{tabular}&
\begin{tabular}{c}  12\\12\\12\\12\\
\end{tabular} &
\begin{tabular}{c}  2\\2\\2\\2\\
\end{tabular} &
\begin{tabular}{c}  12\\12\\12\\12\\
\end{tabular} &
\begin{tabular}{c}  20\\20\\20\\22\\
\end{tabular} &
\begin{tabular}{c}  37\\38\\39\\39\\
\end{tabular} &
\begin{tabular}{c}  66\\66\\67\\69\\
\end{tabular} \\ \hline
$10^{-3}$ &
\begin{tabular}{c} 1/64 \\ 1/128\\ 1/256 \\1/512 \\
\end{tabular} &
\begin{tabular}{c}  70\\70\\72\\73\\
\end{tabular} &
\begin{tabular}{c}  37\\37\\38\\39\\
\end{tabular} &
\begin{tabular}{c}  22\\22\\22\\23\\
\end{tabular}&
\begin{tabular}{c}  12\\13\\12\\13\\
\end{tabular} &
\begin{tabular}{c}  2\\2\\2\\2\\
\end{tabular} &
\begin{tabular}{c}  12\\13\\12\\13\\
\end{tabular} &
\begin{tabular}{c}  22\\22\\22\\23\\
\end{tabular} &
\begin{tabular}{c}  37\\59\\54\\47\\
\end{tabular} &
\begin{tabular}{c}  70\\70\\72\\73\\
\end{tabular} \\ \hline
\end{tabular}
\caption{Number of iteration compared for DN with Dirichlet subdomain larger than Neumann subdomain for short and large time step.}
\label{dnunequal_1d_dln}
\end{table}

\begin{figure}
    \centering
    \subfloat{{\includegraphics[width=5cm,height=3cm]{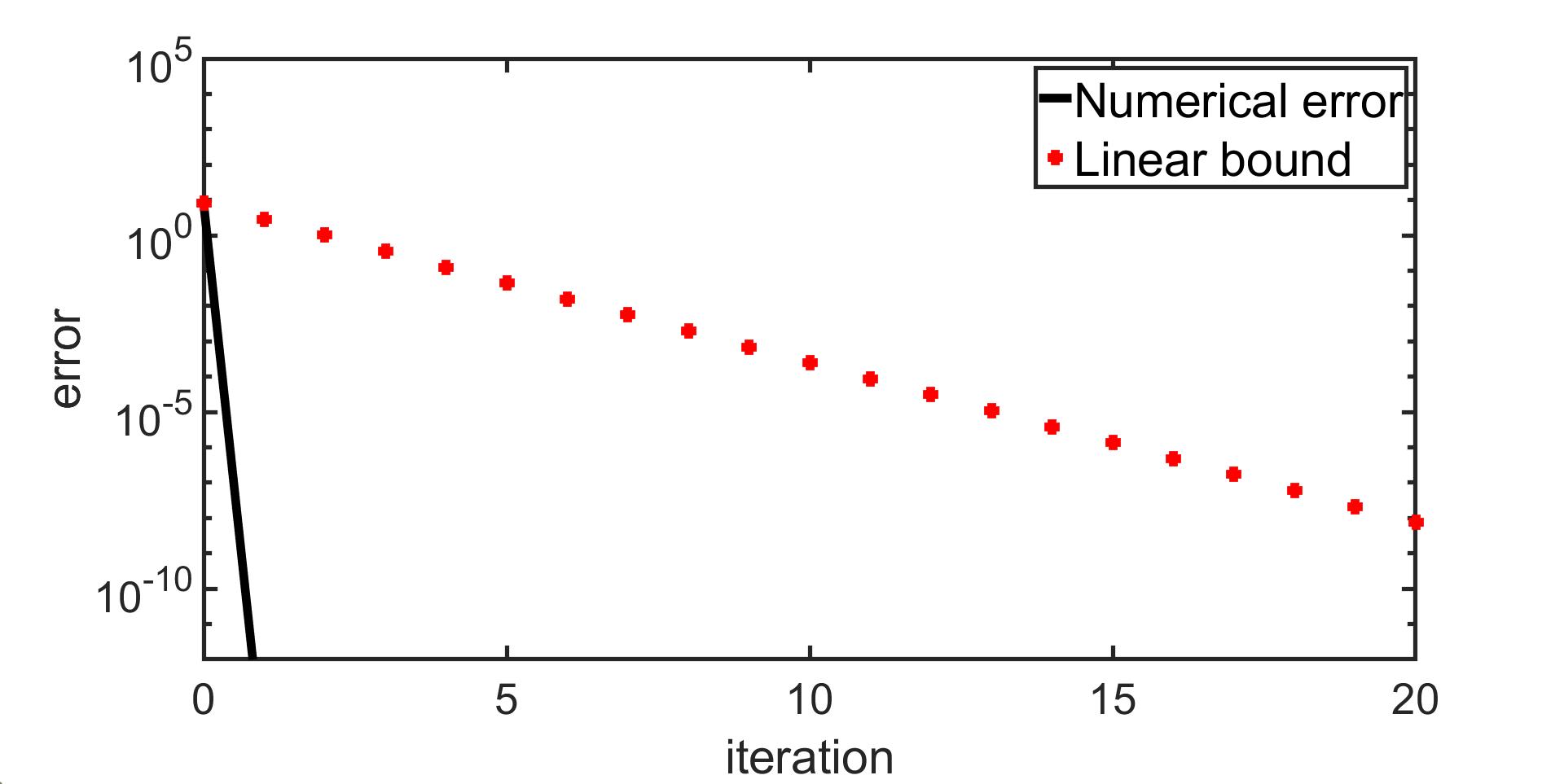} }}
    \quad
    \subfloat{{\includegraphics[width=5cm,height=3cm]{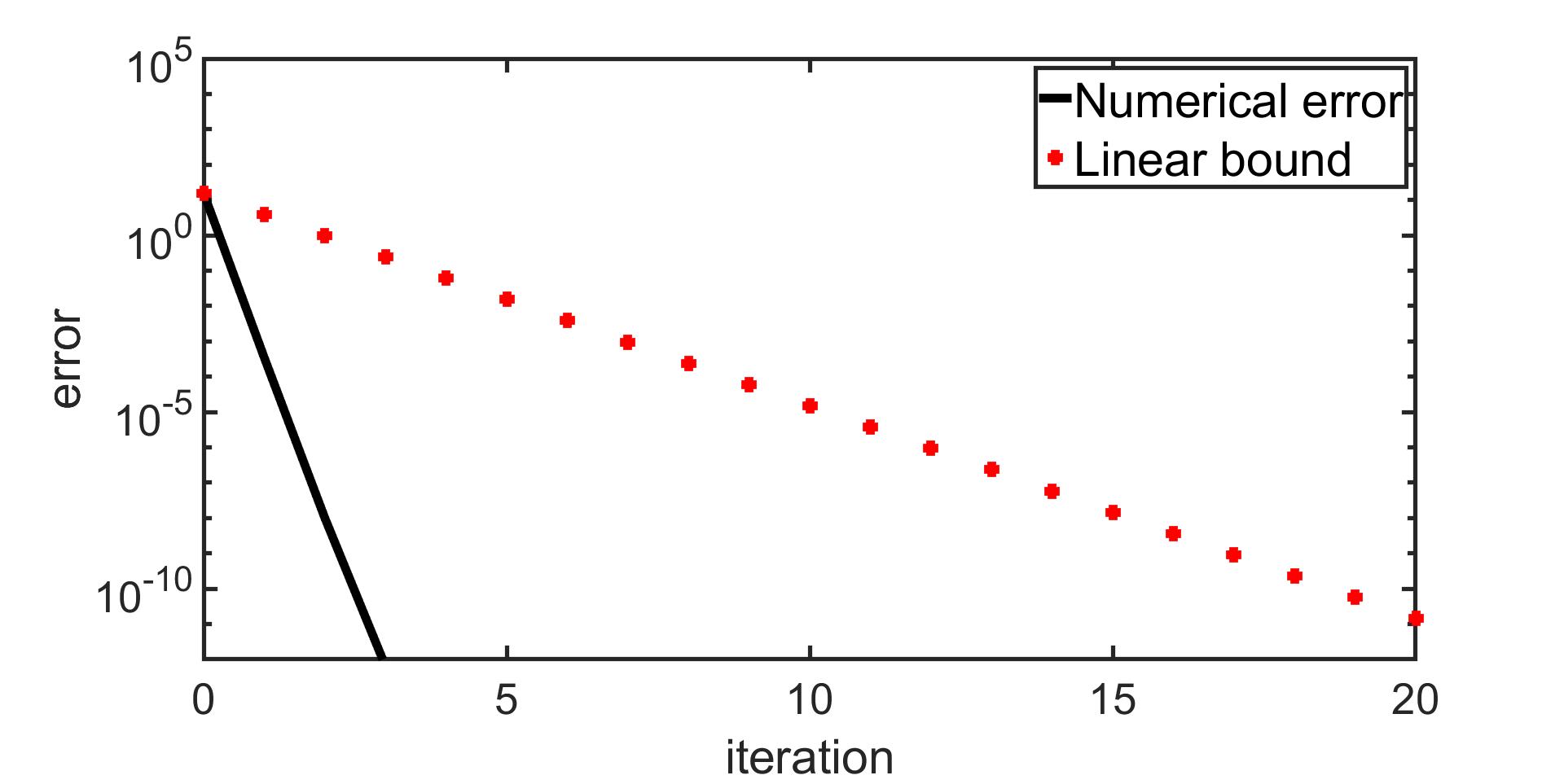} }}
    \caption{Comparison of the numerically measured error and the theoretical error estimates for DN for the mesh size $h=1/64$ and $\theta=1/2$ with Neumann subdomain larger than Dirichlet subdomain  for $\delta_t = 10^{-6}$(left), and $\delta_t = 10^{-3}$ (right).}
    \label{DNerrorNLD}
\end{figure}

\begin{figure}
    \centering
    \subfloat{{\includegraphics[width=5cm,height=3cm]{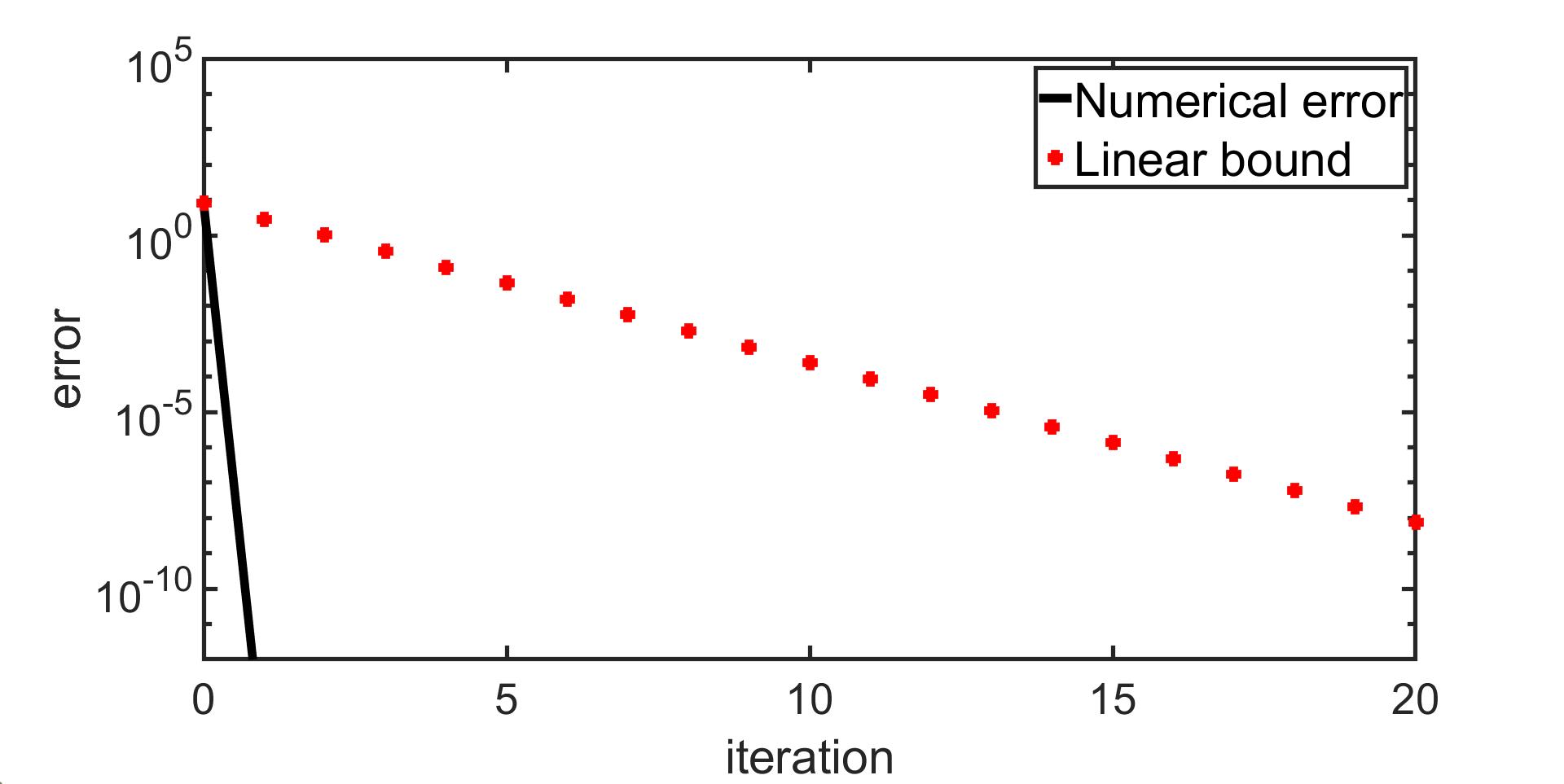} }}
    \quad
    \subfloat{{\includegraphics[width=5cm,height=3cm]{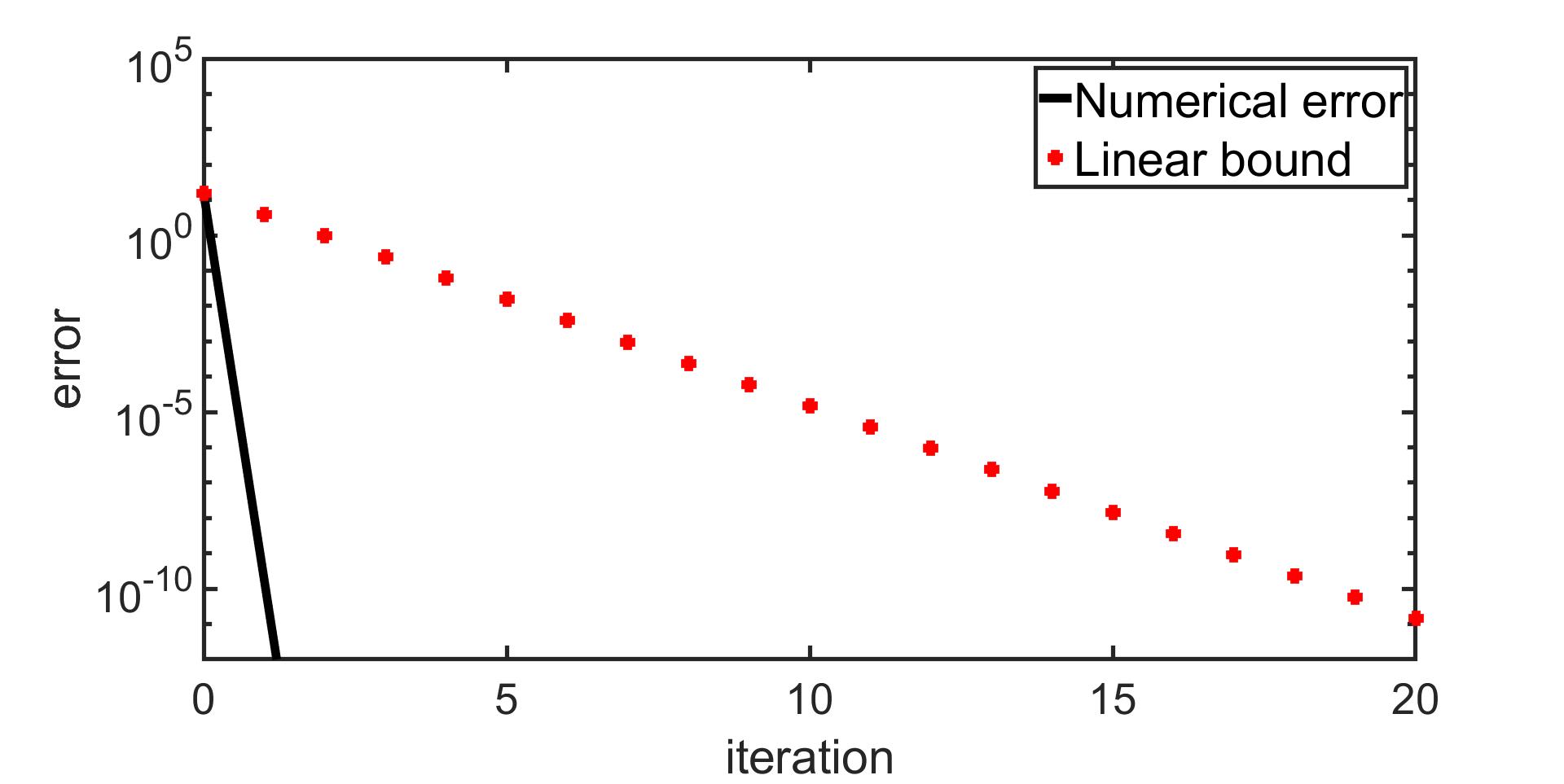} }}
    \caption{Comparison of the numerical error and the theoretical error estimates for DN for the mesh size $h=1/64$ and $\theta=1/2$ with Dirichlet subdomain larger than Neumann subdomain for $\delta_t = 10^{-6}$(left), and $\delta_t = 10^{-3}$ (right).}
    \label{DNerrorDLN}
\end{figure}

\begin{table}
    \begin{minipage}{.5\textwidth}
      \centering
      \begin{tabular}{|p{1.2cm}| p{.3cm}| p{.3cm}| p{.3cm}| p{.3cm}| p{.3cm}| p{.3cm}| }\hline
\diagbox{$h$}{ sd}&{2}&{4}&{8}&{16}&{32}&{64}\\
		\hline1/64&  2& 2&  2& 2&  2& 2\\
		\hline1/128& 2& 2&  2& 2&  2& 2\\
		\hline1/256& 2& 2&  2& 2&  2& 3\\
		\hline1/512& 2& 2&  2& 2&  3& 3\\
		\hline
	\end{tabular}
    \end{minipage}
    \begin{minipage}{.4\textwidth}
      \centering
      \begin{tabular}{|p{1.2cm}| p{.3cm}| p{.3cm}| p{.3cm}| p{.3cm}| p{.3cm}| p{.3cm}| }\hline
\diagbox{$h$}{sd}&{2}&{4}&{8}&{16}&{32}&{64}\\
		\hline1/64&  2& 2&  2& 2&  2& 2\\
		\hline1/128& 2& 2&  2& 2&  2& 3\\
		\hline1/256& 2& 2&  2& 2&  3& 5\\
		\hline1/512& 2& 2&  2& 3&  5& 10\\
		\hline
	\end{tabular}
    \end{minipage}
    \caption{Number of iteration compared of NN for many subdomains of equal length with $\delta_t=10^{-6} $ on the left Table and $\delta_t=10^{-3}$ on the right Table.}
\label{NNmany_equal_1d}
  \end{table}

\begin{table}
    \begin{minipage}{.5\textwidth}
      \centering
      \begin{tabular}{|p{1.2cm}| p{.3cm}| p{.3cm}| p{.3cm}| p{.3cm}| p{.3cm}| p{.3cm}| }\hline
\diagbox{$h$}{ sd}&{2}&{4}&{8}&{16}&{32}&{64}\\
		\hline1/64&  2& 2&  3& 3&  4& 6\\
		\hline1/128& 2& 2&  3& 4&  4& 6\\
		\hline1/256& 2& 2&  4& 4&  6& 8\\
		\hline1/512& 2& 2&  4& 4&  8& 10\\
		\hline
	\end{tabular}
    \end{minipage}
    \begin{minipage}{.4\textwidth}
      \centering
      \begin{tabular}{|p{1.2cm}| p{.3cm}| p{.3cm}| p{.3cm}| p{.3cm}| p{.3cm}| p{.3cm}| }\hline
\diagbox{$h$}{sd}&{2}&{4}&{8}&{16}&{32}&{64}\\
		\hline1/64&  2& 2&  3& 3&  4& 6\\
		\hline1/128& 2& 2&  3& 4&  4& 6\\
		\hline1/256& 2& 2&  4& 4&  6& 9\\
		\hline1/512& 2& 2&  4& 7&  9& 13\\
		\hline
	\end{tabular}
    \end{minipage}
    \caption{Number of iteration compared of NN for many subdomains of unequal length with $\delta_t=10^{-6} $ on the left Table and $\delta_t=10^{-3}$ on the right Table.}
\label{NNmany_unequal_1d}
  \end{table}

\begin{figure}
    \centering
    \subfloat{{\includegraphics[width=5cm,height=3cm]{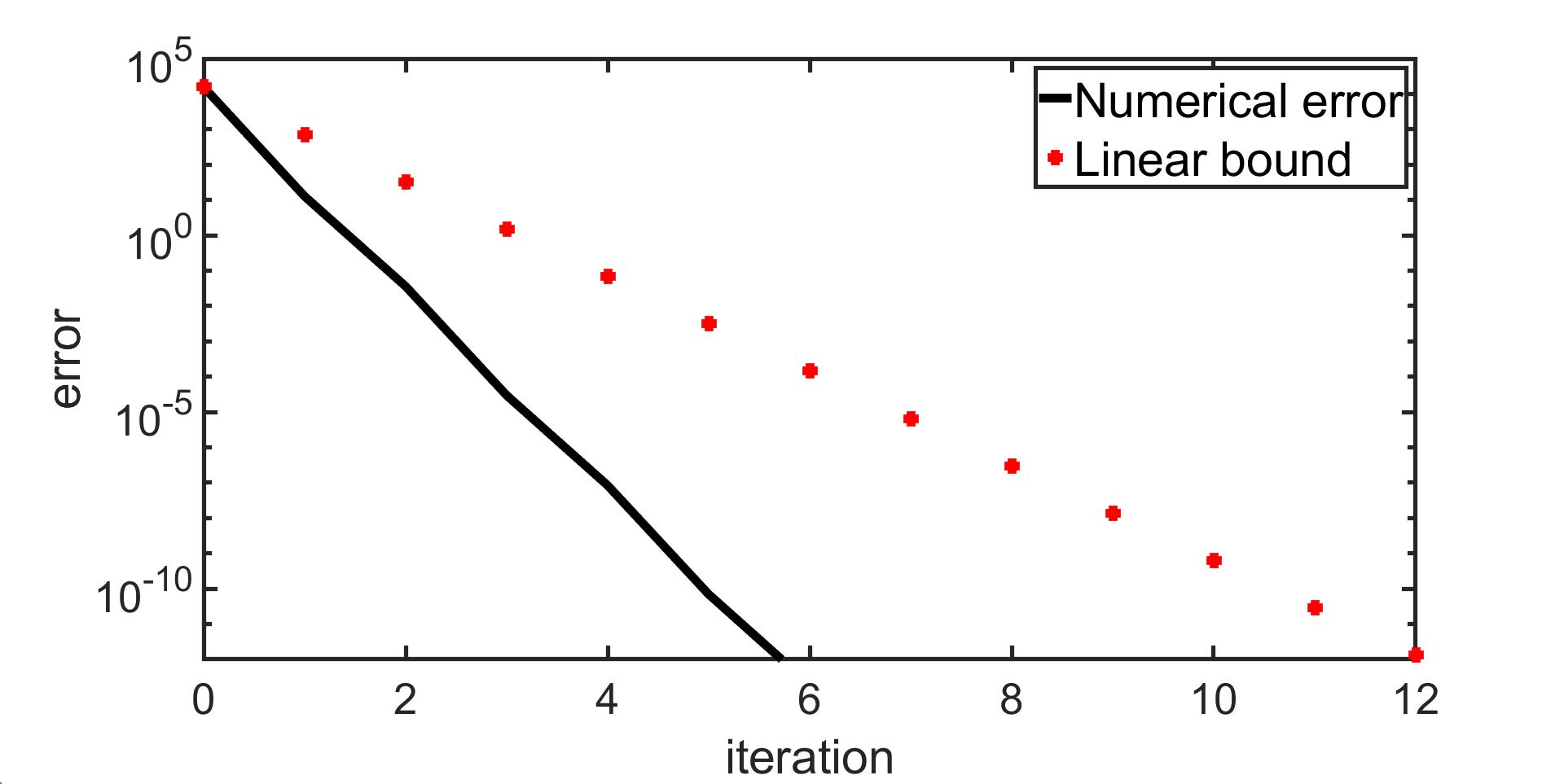} }}
    \quad
    \subfloat{{\includegraphics[width=5cm,height=3cm]{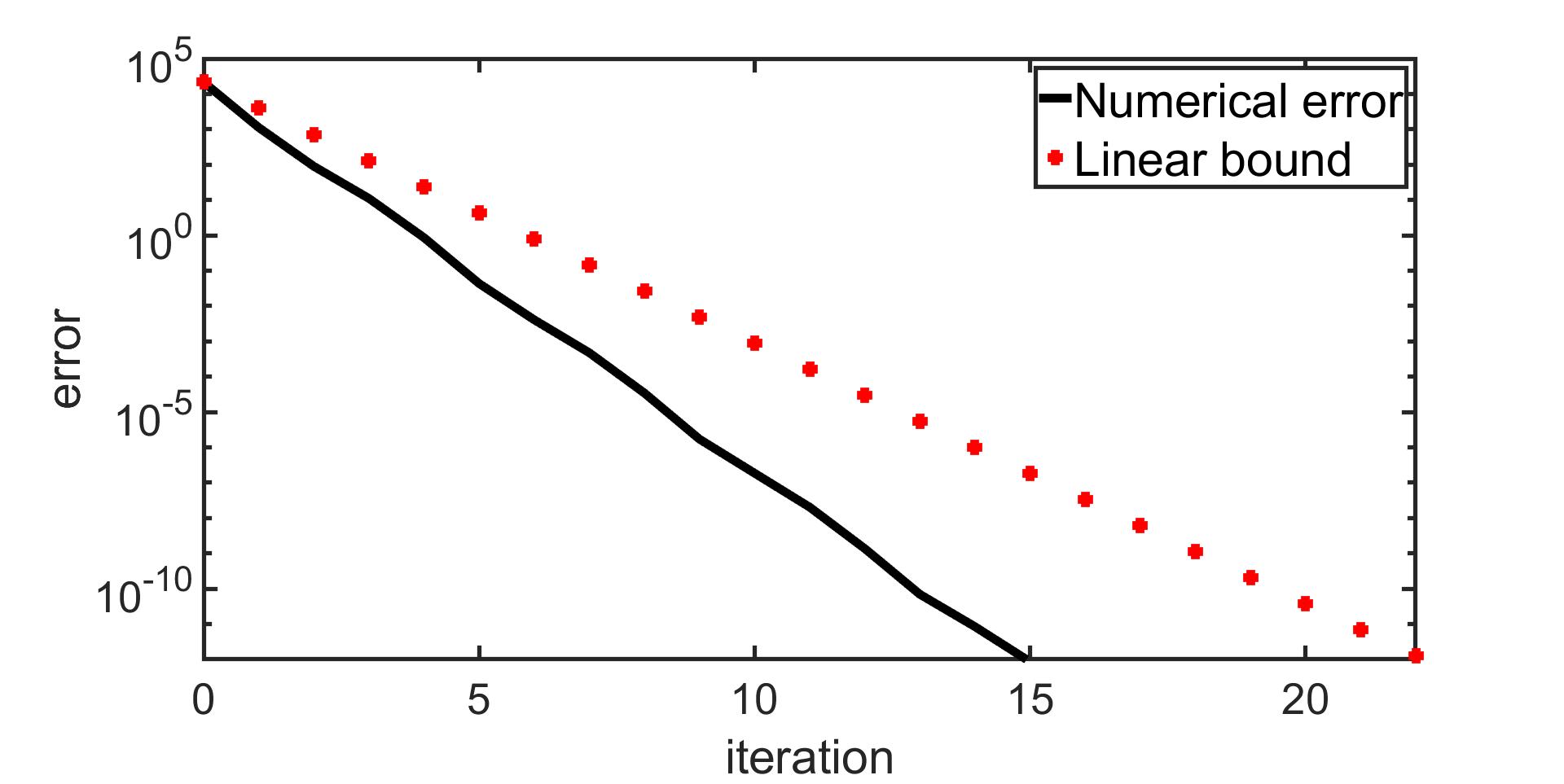} }}
    \caption{Comparison of the numerical error and the theoretical error estimates for  thirty-two (left) and sixty-four (right) subdomains with $h=1/512, \theta=1/4$.}
    \label{NNlongtimemultiple}
\end{figure}
\subsection{Numerics of DN \& NN method in 2D}
To perform numerical experiments for DN method in 2D  for equal subdomain we have taken $\Omega=(0, 1)\times(0, 1)$ and split it into $\Omega_1=(0, 1/2)\times(0, 1)$ and $\Omega_2=(1/2, 1)\times(0, 1)$. For unequal subdomain: first we consider Neumann subdomain is larger than Dirichlet Subdomain by choosing the domain $\Omega=(1, 2)\times(0, 1)$ and partitioned into $\Omega_1=(1, 1.4)\times(0, 1)$ and $\Omega_2=(1.4, 2)\times(0, 1)$. And secondly for Dirichlet subdomain is larger than Neumann subdomain,  we choose the domain  $\Omega=(-1.5, 1)\times(0, 1)$ and partitioned into $\Omega_1=(-1.5, 0)\times(0, 1)$ and $\Omega_2=(0, 1)\times(0, 1)$. 
We plot the error curves of DN method with above described decomposition and for various parameter of $\theta$, mesh size $h_x$ (discretization parameter in $x-$ direction), $h_y$ (discretization parameter in $y-$ direction) and time step $\delta_t$  in Figure \ref{DNshorttime2d}, \ref{DNshorttime2d_128u}, \ref{DNlongtime2d} and \ref{DNlongtime2d128u}, where we find that the number of iterations conform well with the results established in Theorems in Section \ref{Section2}.
For experiments of NN method in multisubdomain setting in 2D we take the domain $\Omega= (0, 16)\times(0, 1)$. For short time step $\delta_t$ we have given comparison result in terms of iteration count for NN method in multiple subdomain case with equal and unequal width in Table \ref{NNmany_equalandunequal_2d} by fixing $h_y=1/32$ and parameter $\theta=1/4$ and varying $h_x$. In Figure \ref{NNlongtimeerror2d_4816} we plot the error curve of NN method for $4, 8, 16$ subdomain decomposition of equal width for $h_x=1/64, h_y=1/32$ and by varying the parameter $\theta$.
In Figure \ref{NNlongtimeerror2d_comp} we have given  comparison of the numerical error and the theoretical error estimate coming from Theorem \ref{NNequal2dthm} for NN method with fixed $h_y=1/32, \theta=1/4$ and by varying the number of subdomain and mesh size $h_x$ for long time step $\delta_t$.
    
\begin{figure}
    \centering
    \subfloat{{\includegraphics[width=5cm,height=3cm]{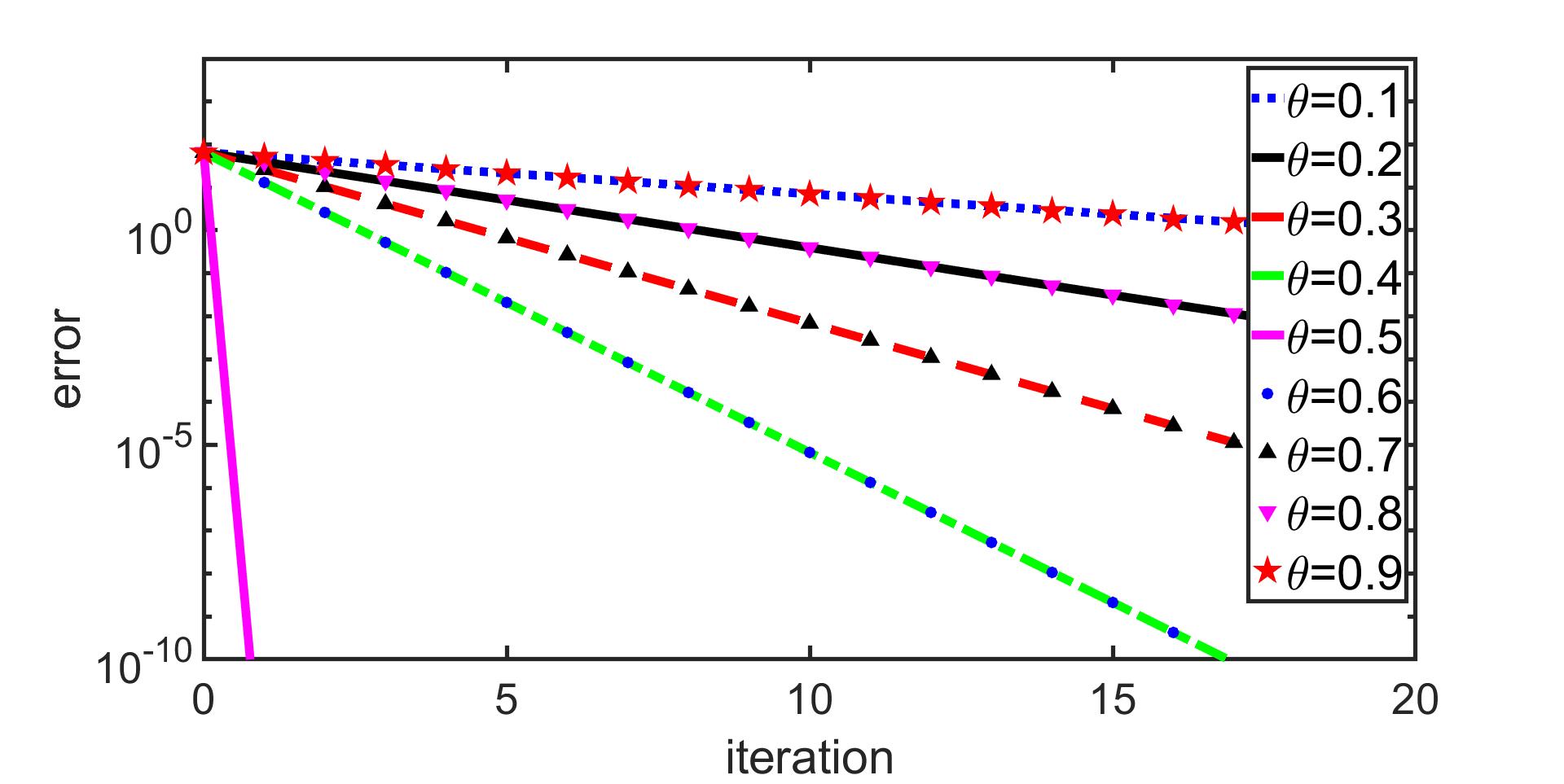} }} 
    \subfloat{{\includegraphics[width=5cm,height=3cm]{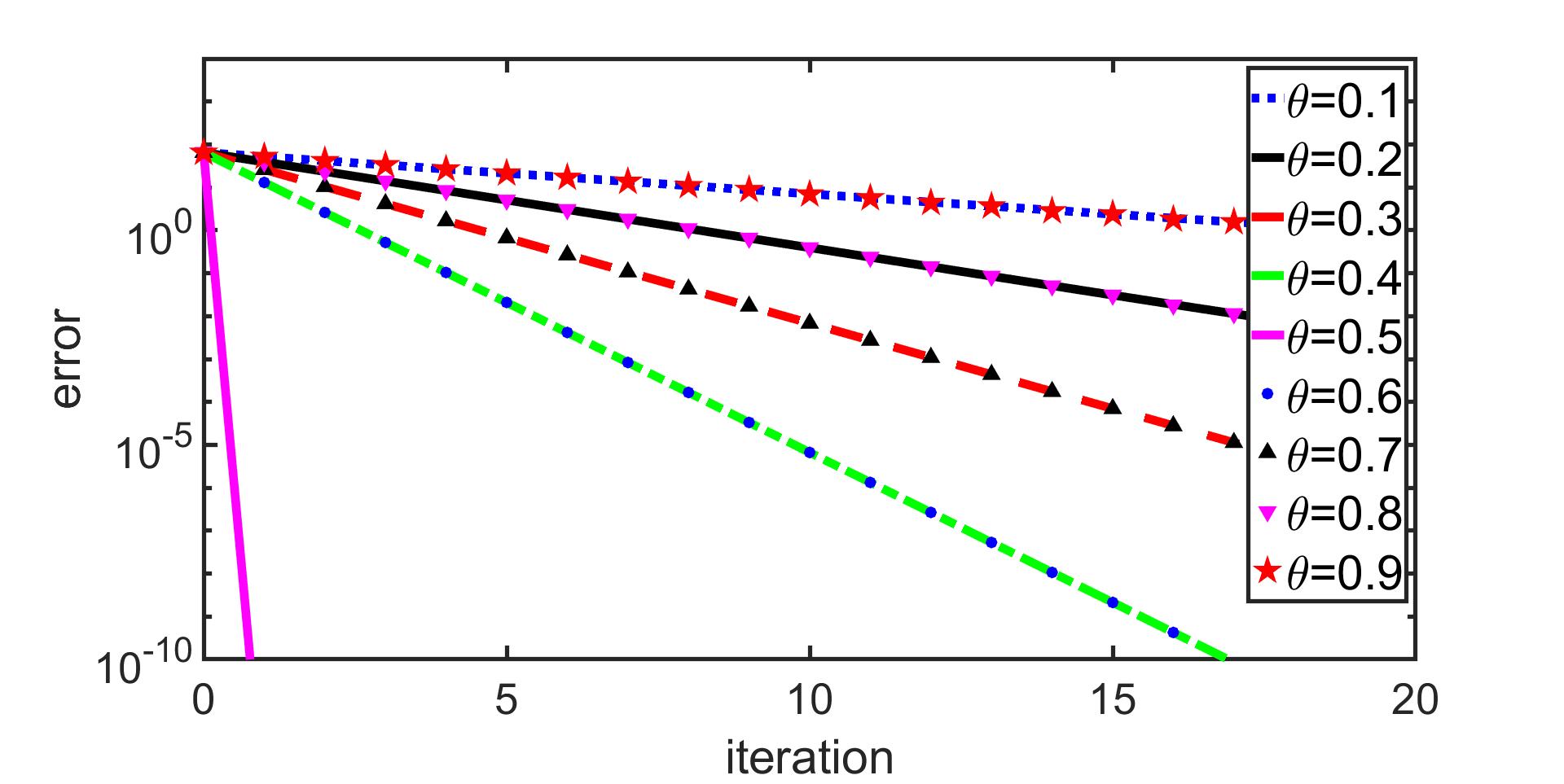} }}
    \subfloat{{\includegraphics[width=5cm,height=3cm]{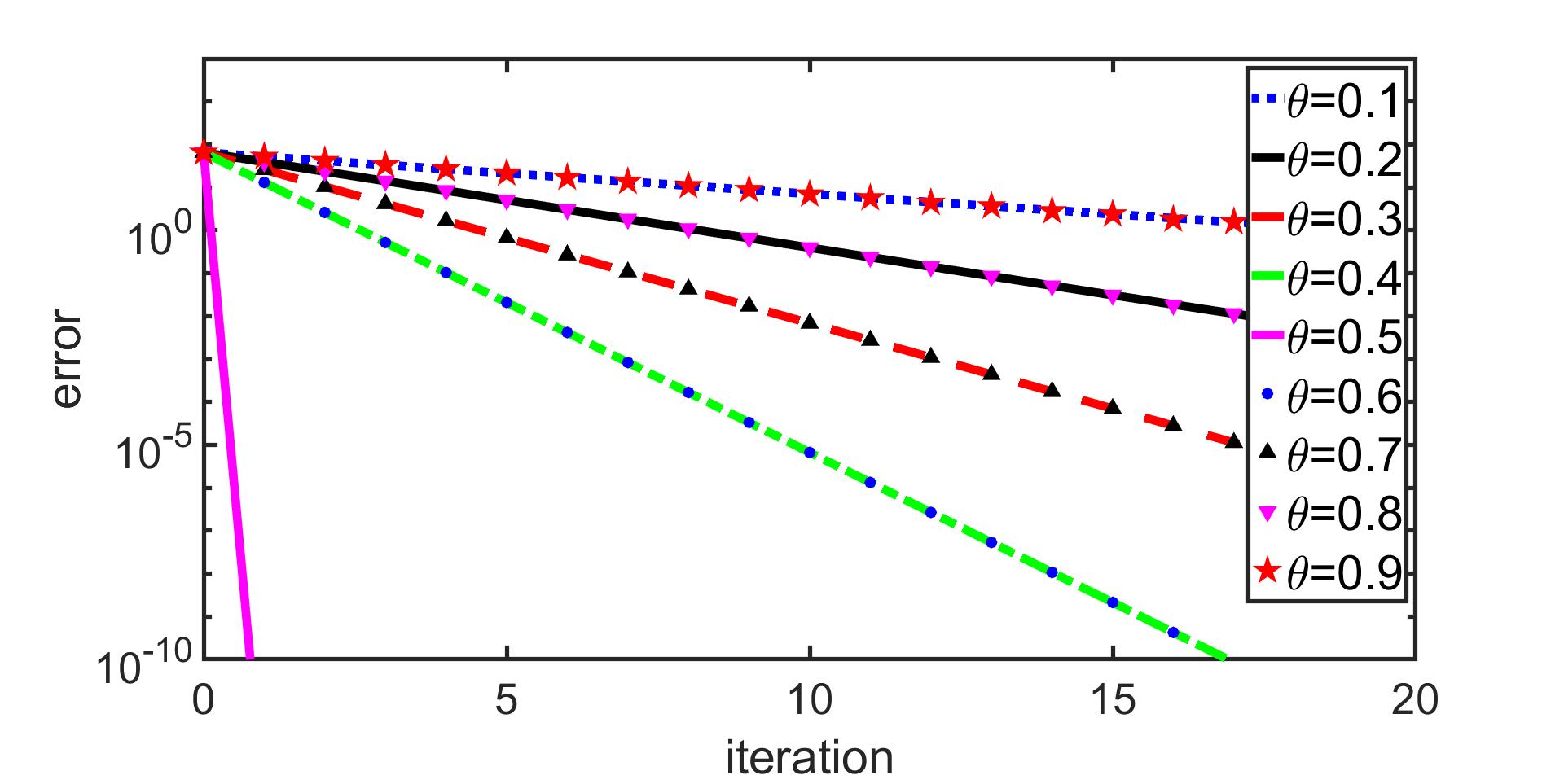} }}
    \caption{From left to right, iteration compared for DN with equal subdomaon (1st), Dirichlet subdomain larger than Neumann subdomain (2nd), and  Neumann subdomain larger than Dirichlet subdomain (3rd), for mesh size $h_x=h_y=1/64$ and time step $\delta_t=10^{-6}$.}
    \label{DNshorttime2d}
\end{figure}
\begin{figure}
    \centering
    \subfloat{{\includegraphics[width=5cm,height=3cm]{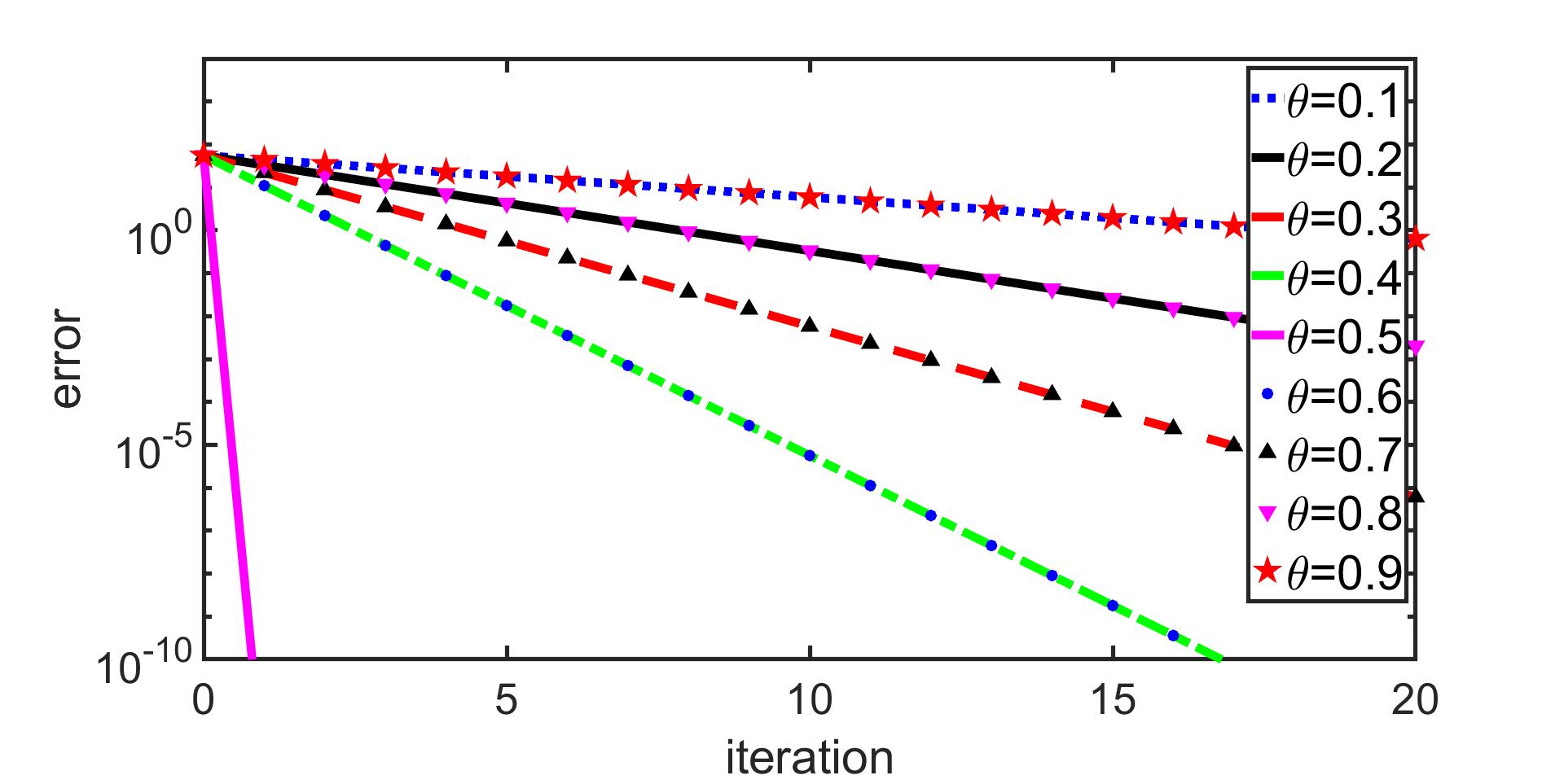} }} 
    \subfloat{{\includegraphics[width=5cm,height=3cm]{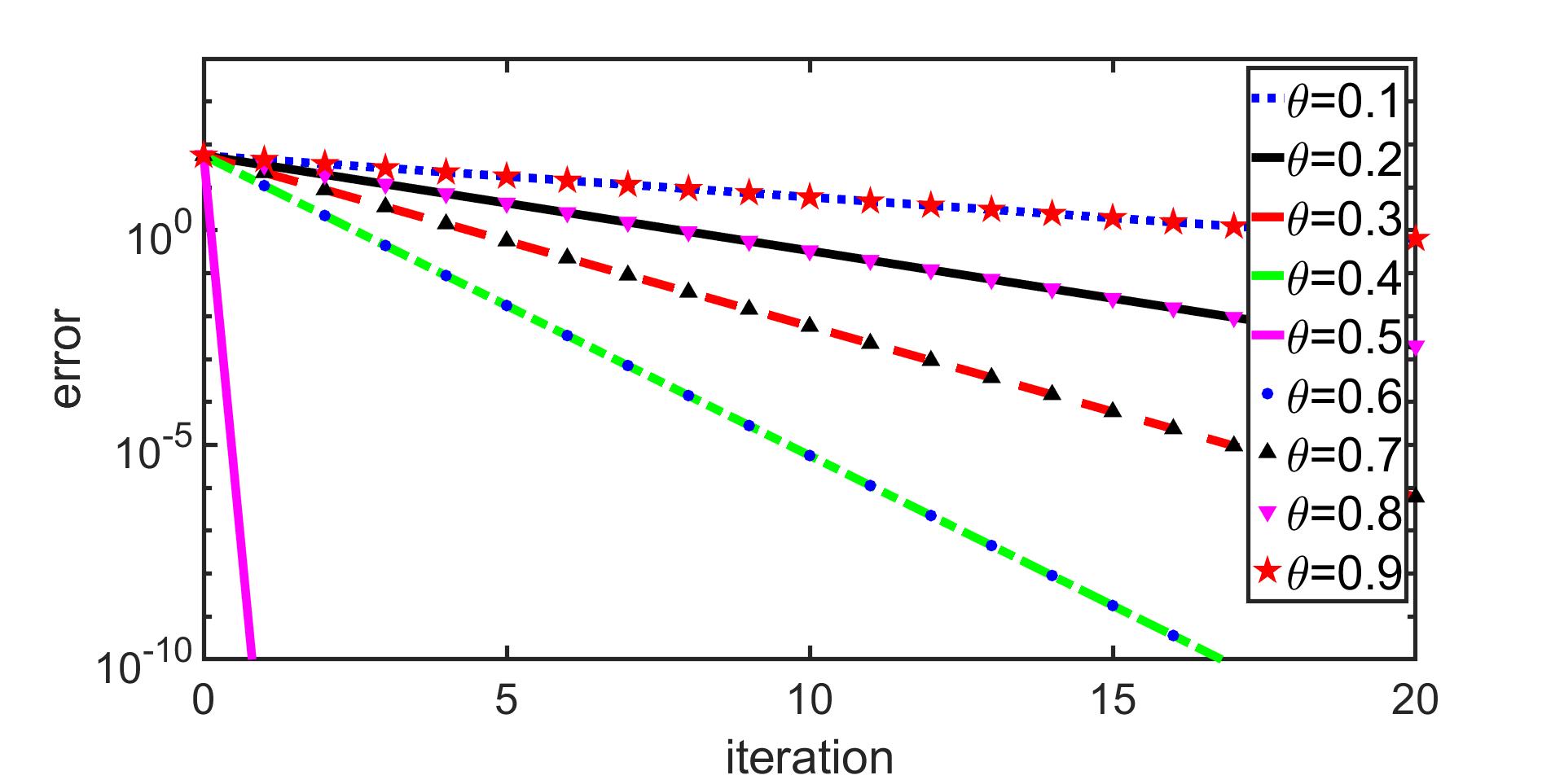} }}
    \subfloat{{\includegraphics[width=5cm,height=3cm]{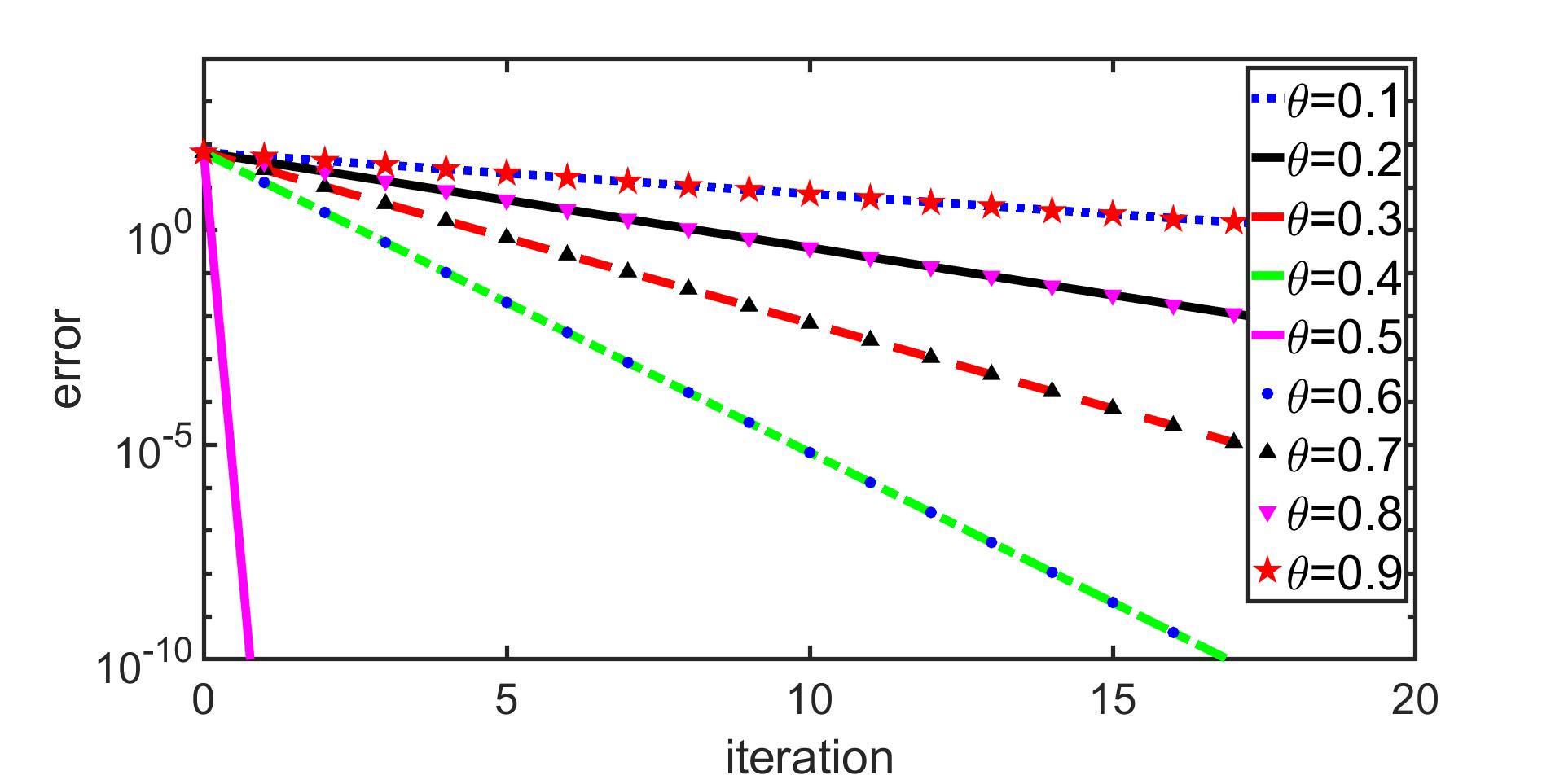} }}
    \caption{From left to right, iteration compared for DN with equal subdomaon (1st), Dirichlet subdomain larger than Neumann subdomain (2nd), and  Neumann subdomain larger than Dirichlet subdomain (3rd), for mesh size $h_x=h_y=1/128$ and time step $ \delta_t=10^{-6}$.}
    \label{DNshorttime2d_128u}
\end{figure}
\begin{figure}
    \centering
    \subfloat{{\includegraphics[width=5cm,height=3cm]{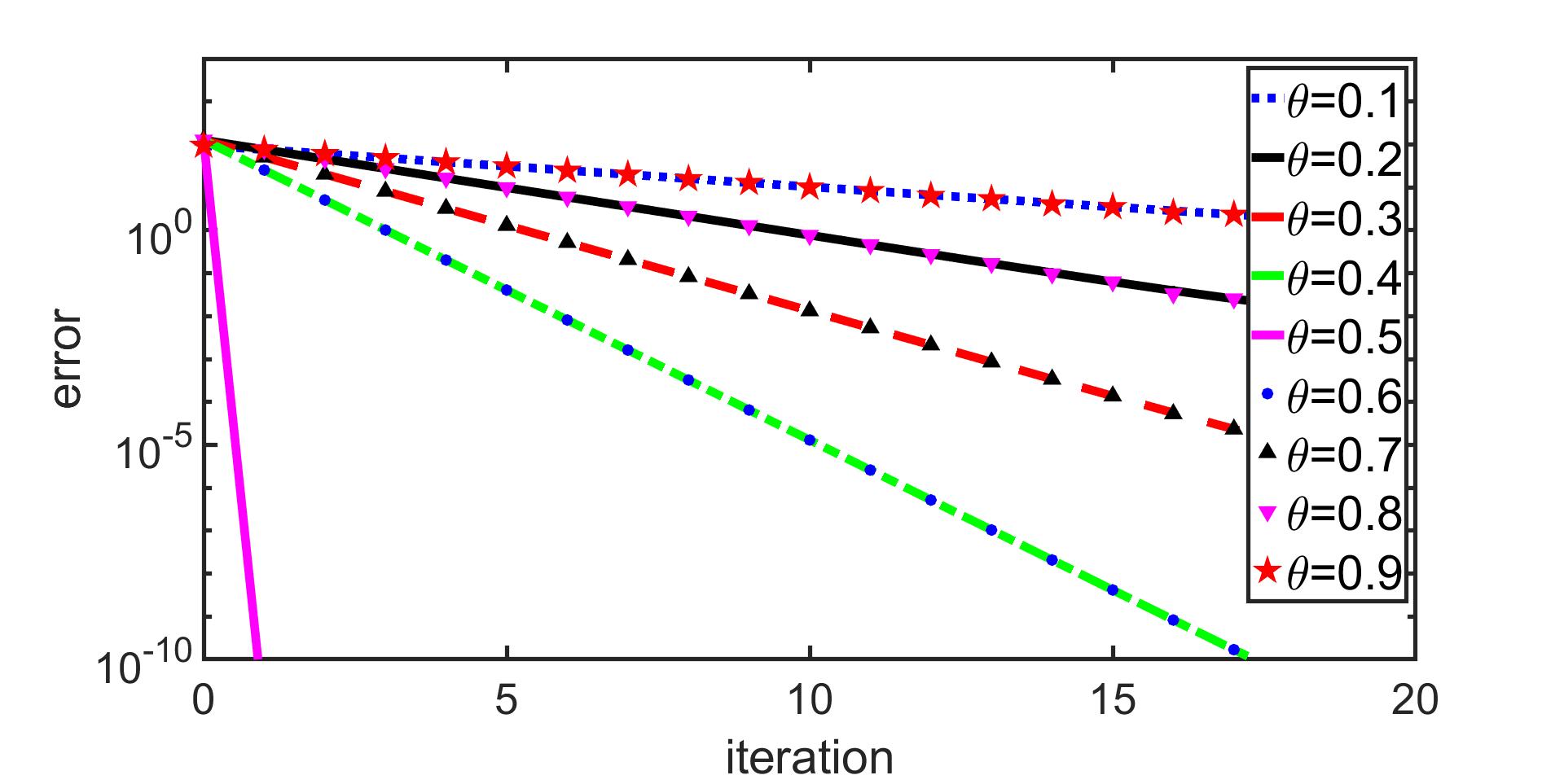} }}
    \subfloat{{\includegraphics[width=5cm,height=3cm]{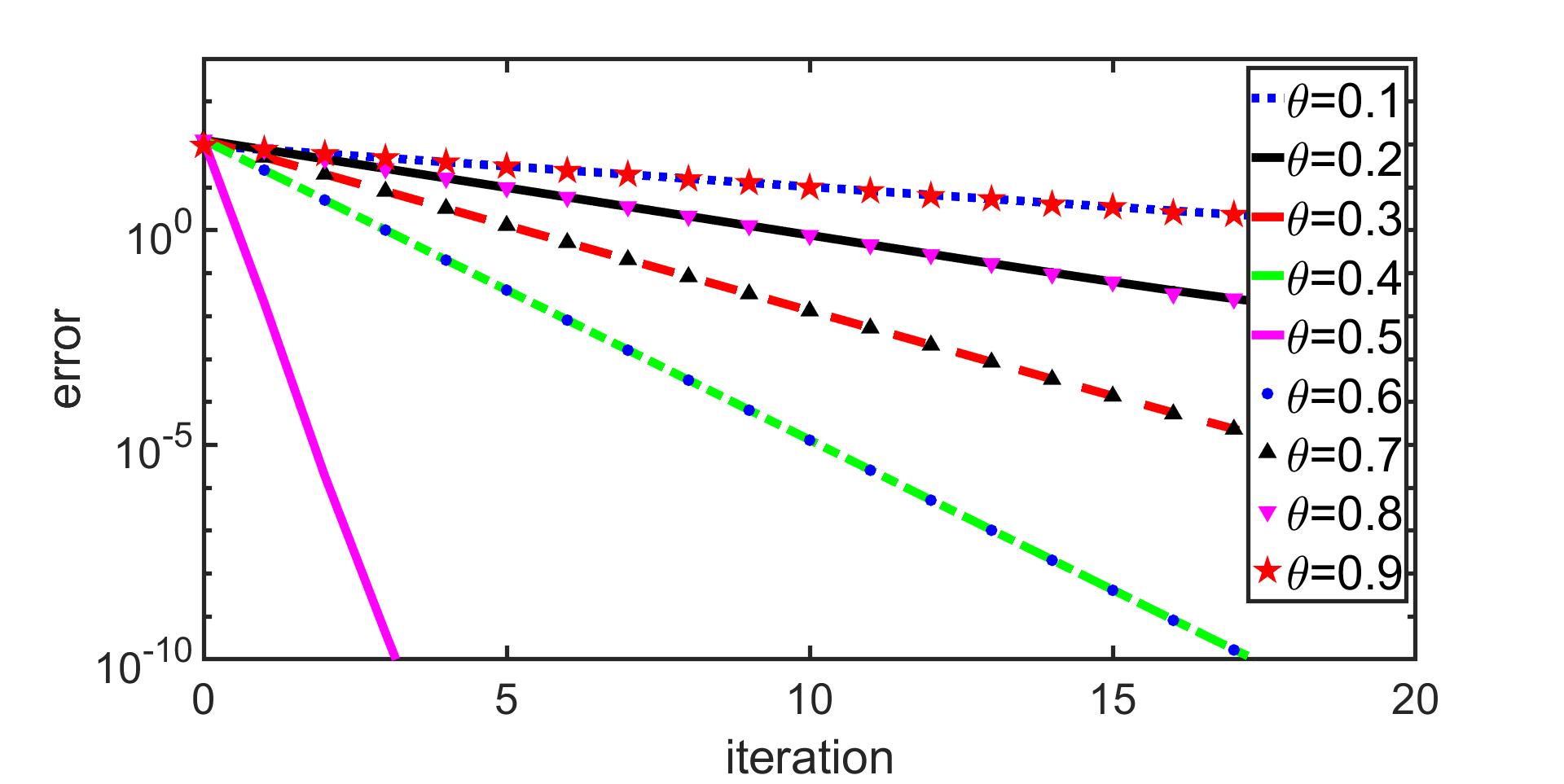} }}
    \subfloat{{\includegraphics[width=5cm,height=3cm]{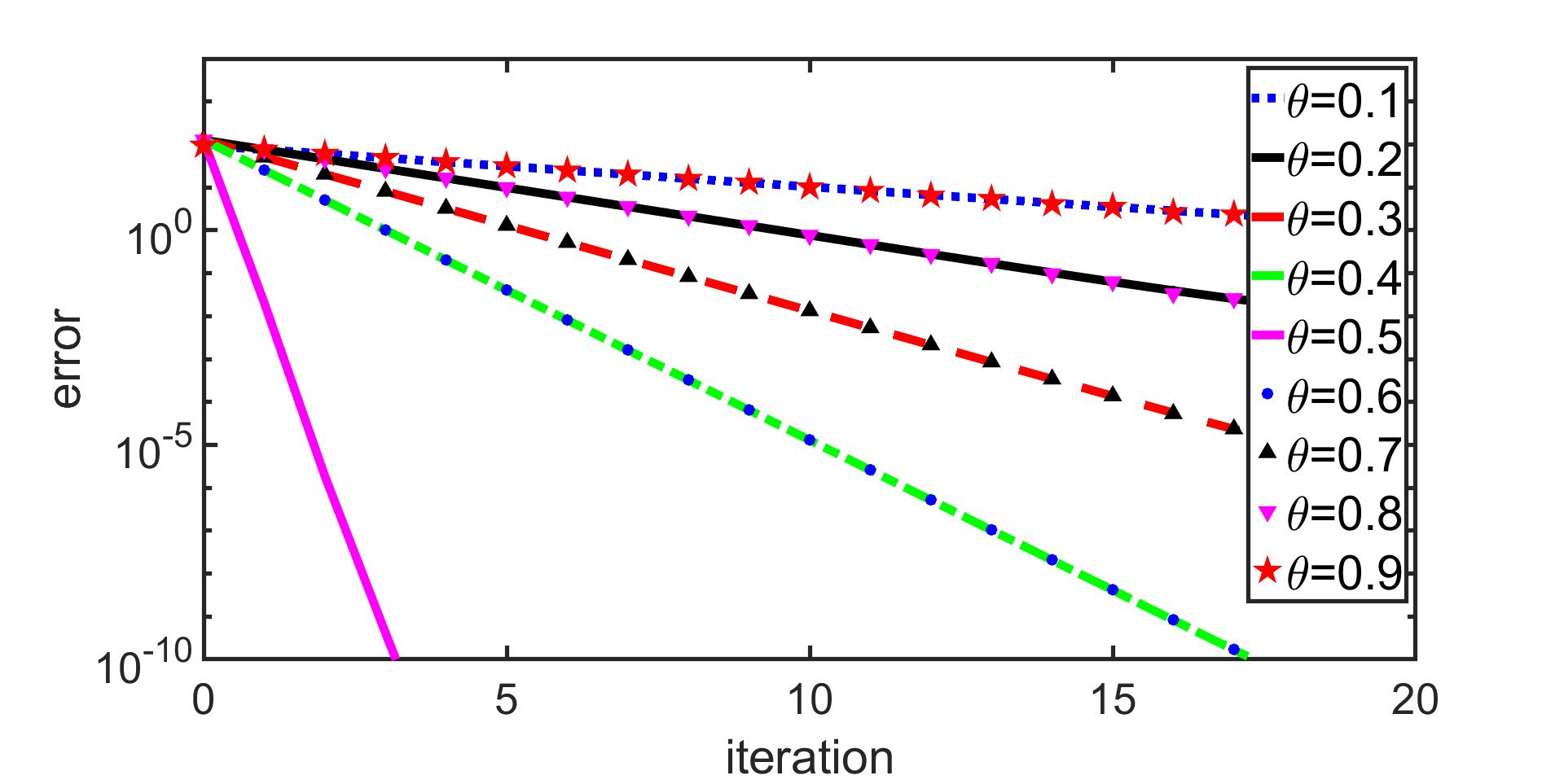} }}
    \caption{From left to right, iteration compared for DN with equal subdomaon (1st), Dirichlet subdomain larger than Neumann subdomain (2nd), and  Neumann subdomain larger than Dirichlet subdomain (3rd), for mesh size $h_x=h_y=1/64$ and time step $ \delta_t=10^{-3}$.}
    \label{DNlongtime2d}
\end{figure}
\begin{figure}
    \centering
    \subfloat{{\includegraphics[width=5cm,height=3cm]{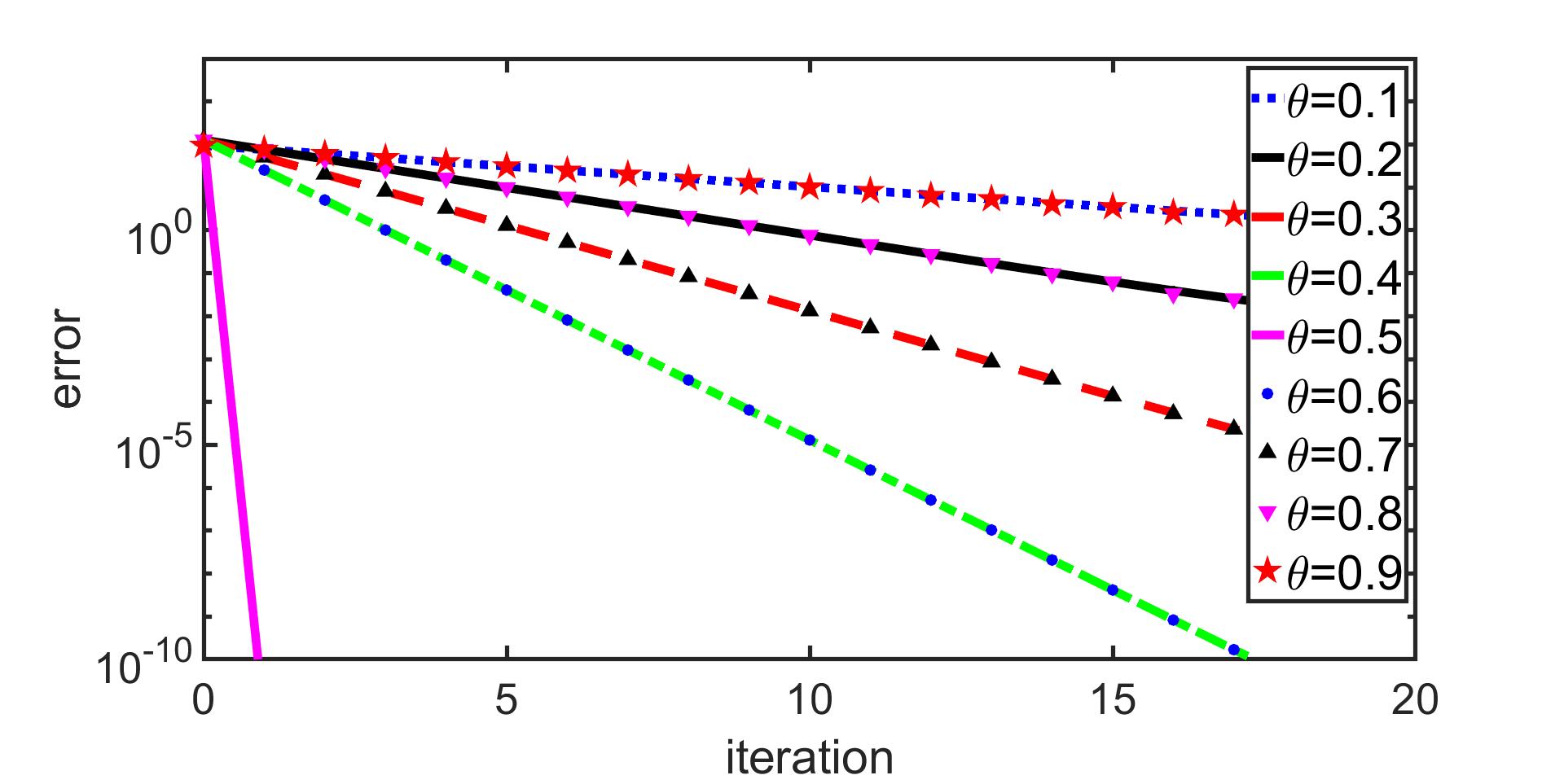} }}
    \subfloat{{\includegraphics[width=5cm,height=3cm]{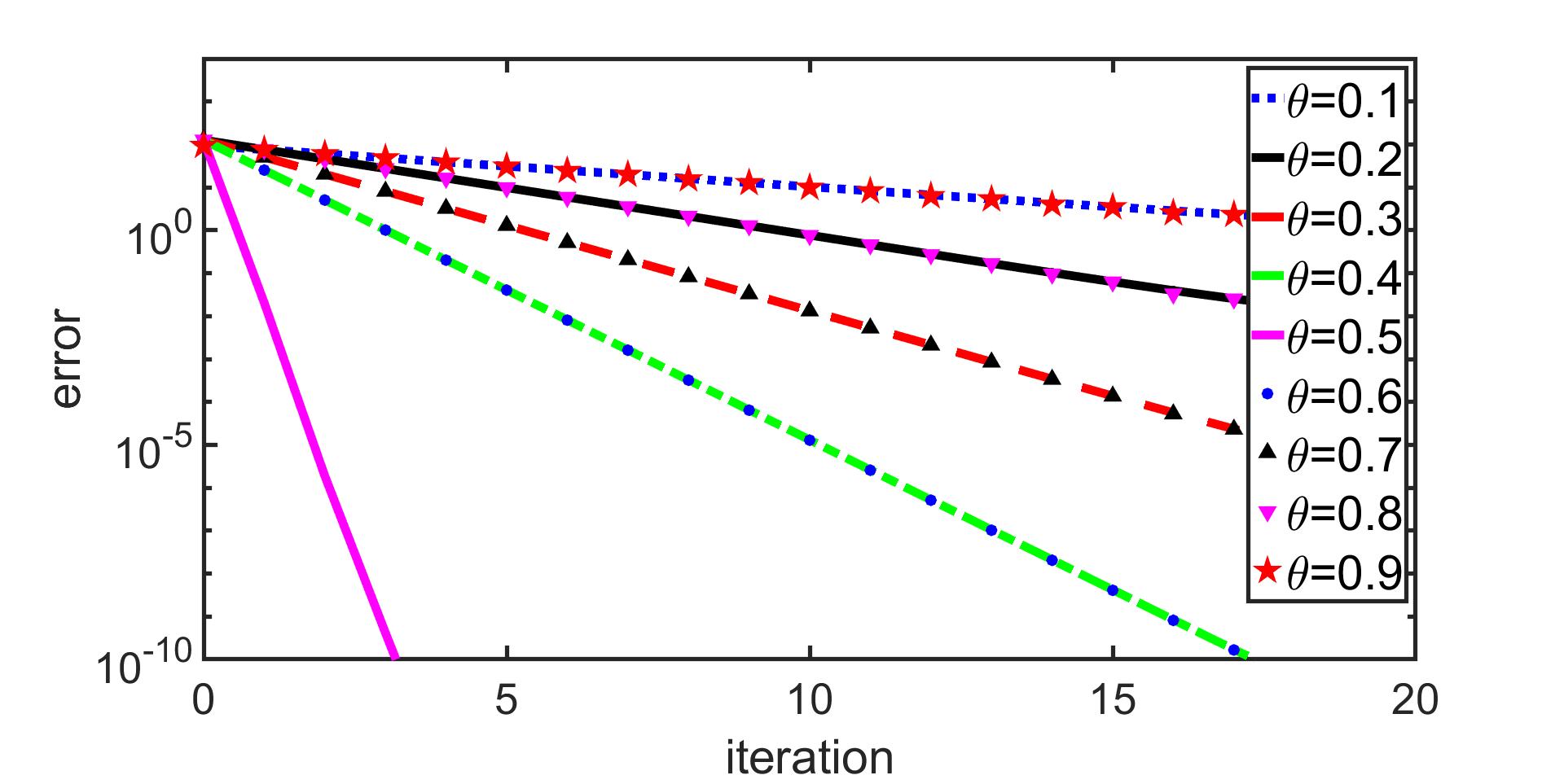} }}
    \subfloat{{\includegraphics[width=5cm,height=3cm]{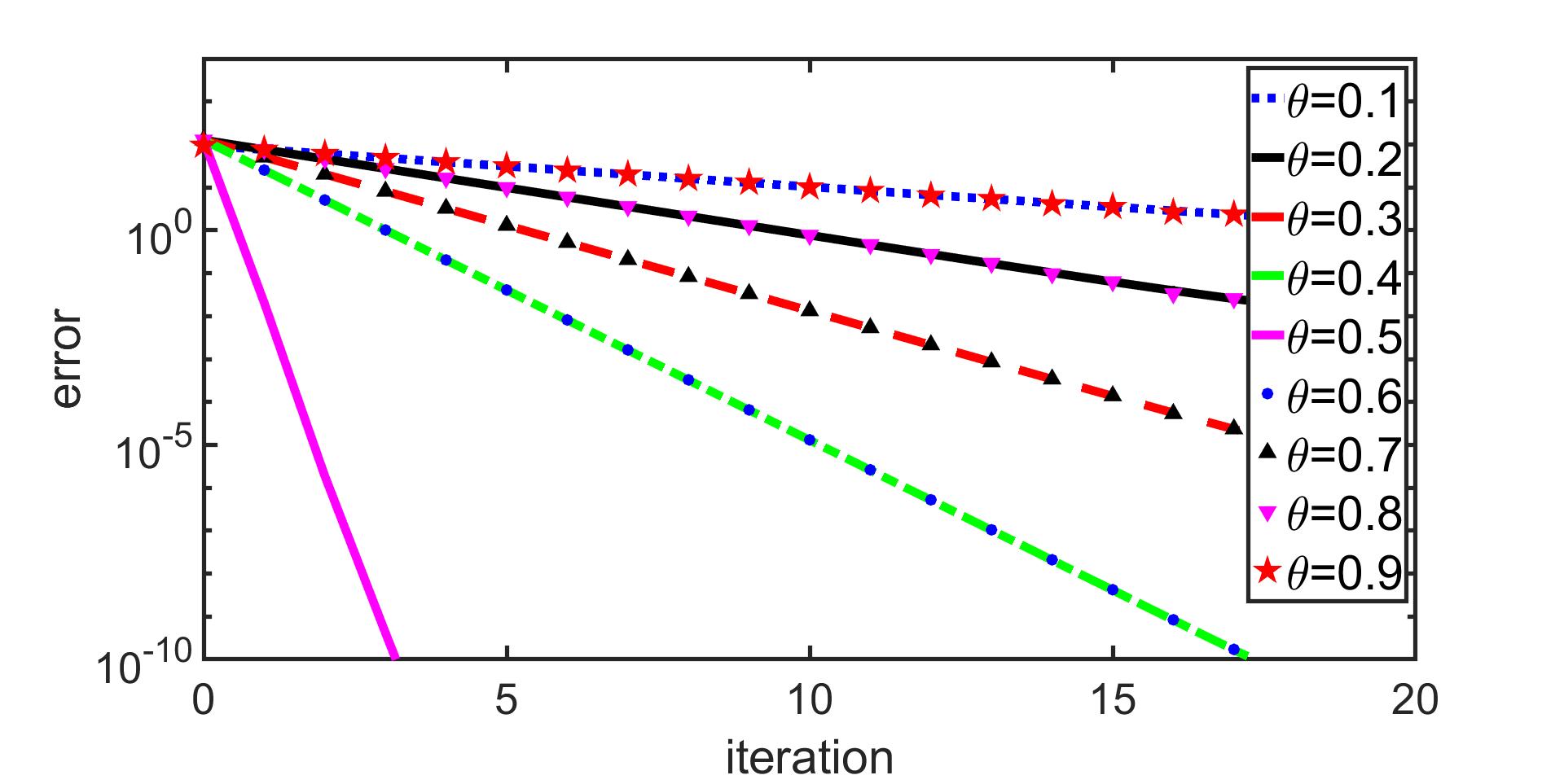} }}
    \caption{From left to right, iteration compared for DN with equal subdomaon (1st), Dirichlet subdomain larger than Neumann subdomain (2nd), and  Neumann subdomain larger than Dirichlet subdomain (3rd), for mesh size $h_x=h_y=1/128$ and time step $ \delta_t=10^{-3}$.}
    \label{DNlongtime2d128u}
\end{figure}


\begin{table}
    \begin{minipage}{.5\textwidth}
      \centering
      \begin{tabular}{|p{1.2cm}| p{.3cm}| p{.3cm}| p{.3cm}| p{.3cm}| p{.3cm}| p{.3cm}| }\hline
\diagbox{$h_x$}{ sd}&{2}&{4}&{8}&{16}&{32}&{64}\\
		\hline1/64&  2& 2&  2& 2&  2& 2\\
		\hline1/128& 2& 2&  2& 2&  2& 2\\
		\hline1/256& 2& 2&  2& 2&  2& 3\\
		\hline1/512& 2& 2&  2& 2&  3& 3\\
		\hline
	\end{tabular}
    \end{minipage}
    \begin{minipage}{.4\textwidth}
      \centering
      \begin{tabular}{|p{1.2cm}| p{.3cm}| p{.3cm}| p{.3cm}| p{.3cm}| p{.3cm}| p{.3cm}| }\hline
\diagbox{$h_x$}{sd}&{2}&{4}&{8}&{16}&{32}&{64}\\
		\hline1/64&  2& 2&  3& 3&  4& 6\\
		\hline1/128& 2& 2&  3& 4&  4& 6\\
		\hline1/256& 2& 2&  4& 4&  6& 8\\
		\hline1/512& 2& 2&  4& 4&  8& 10\\
		\hline
	\end{tabular}
    \end{minipage}
    \caption{Number of iteration compared of NN for many subdomains of equal width (left) and unequal width (right) with $\delta_t=10^{-6}$ and $h_y=1/32$.}
\label{NNmany_equalandunequal_2d}
  \end{table}

\begin{figure}
    \centering
    \subfloat{{\includegraphics[width=5cm,height=3cm]{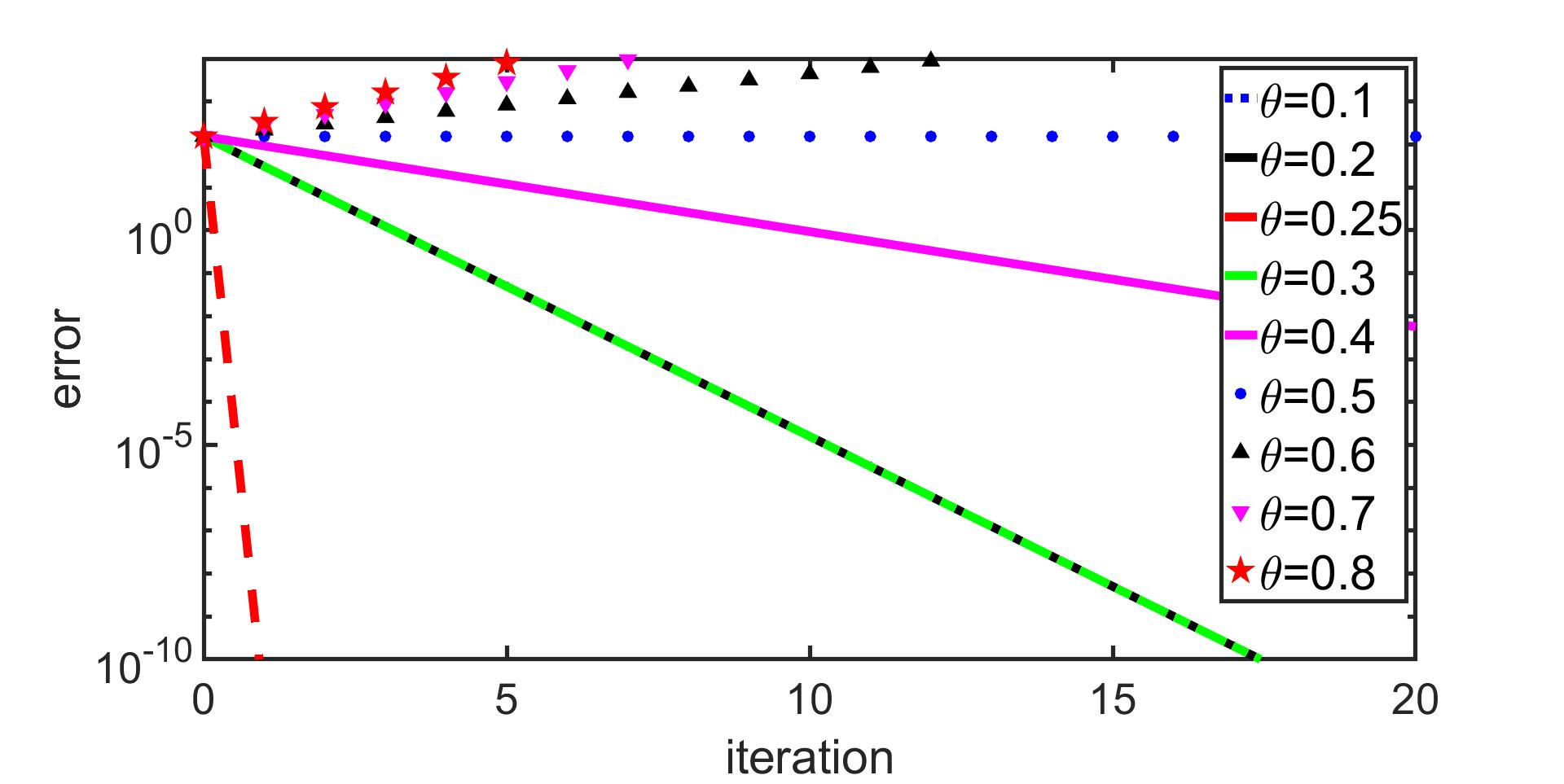} }}
    \subfloat{{\includegraphics[width=5cm,height=3cm]{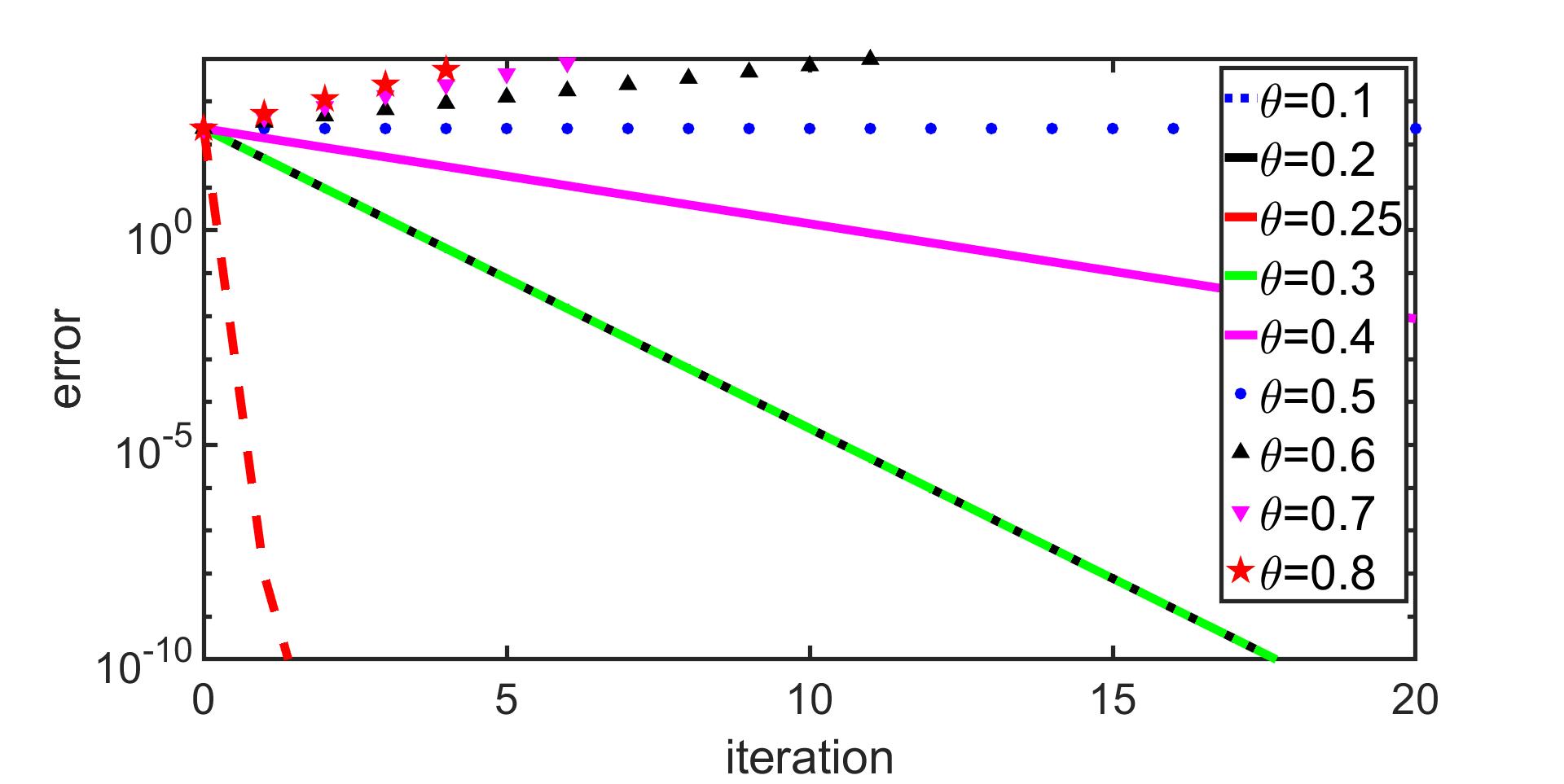} }}
    \subfloat{{\includegraphics[width=5cm,height=3cm]{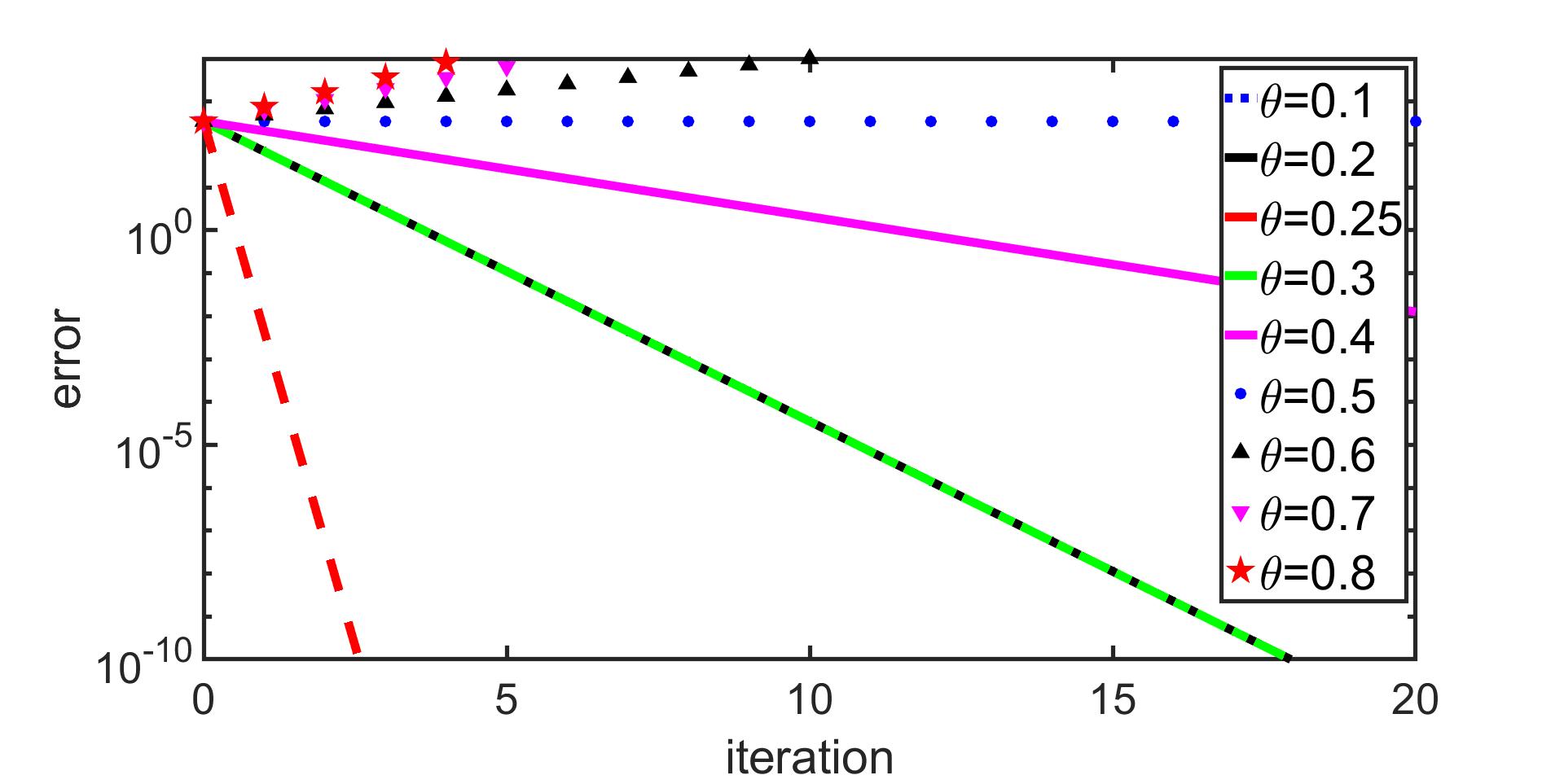} }}
    \caption{From left to right, iteration compared for NN method with 4, 8, 16 subdomain for $\delta_t=10^{-3}, h_x=1/64, h_y=1/32$.}
    \label{NNlongtimeerror2d_4816}
\end{figure}

\begin{figure}
    \centering
    \subfloat{{\includegraphics[width=5cm,height=3cm]{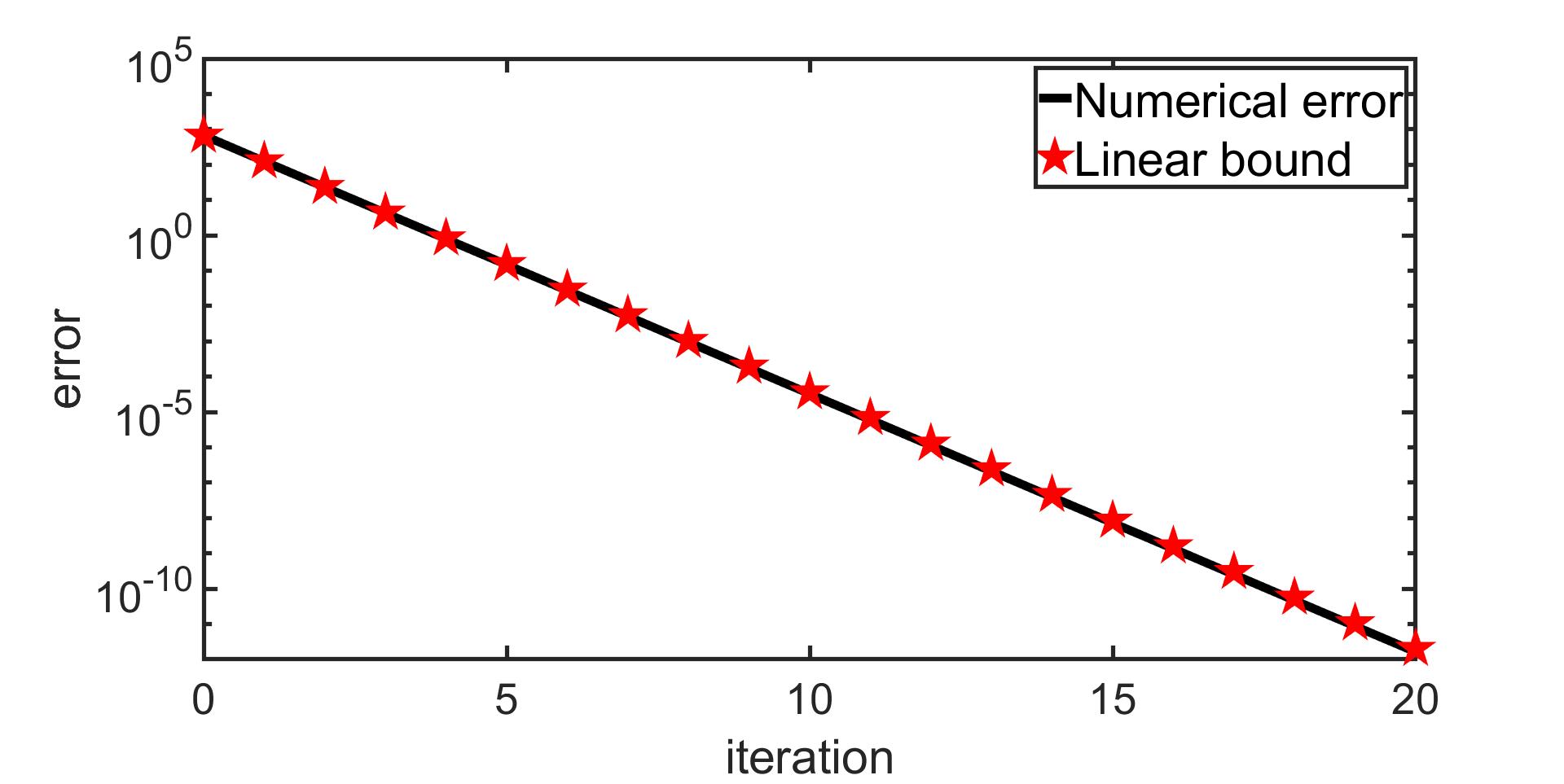} }}
    \subfloat{{\includegraphics[width=5cm,height=3cm]{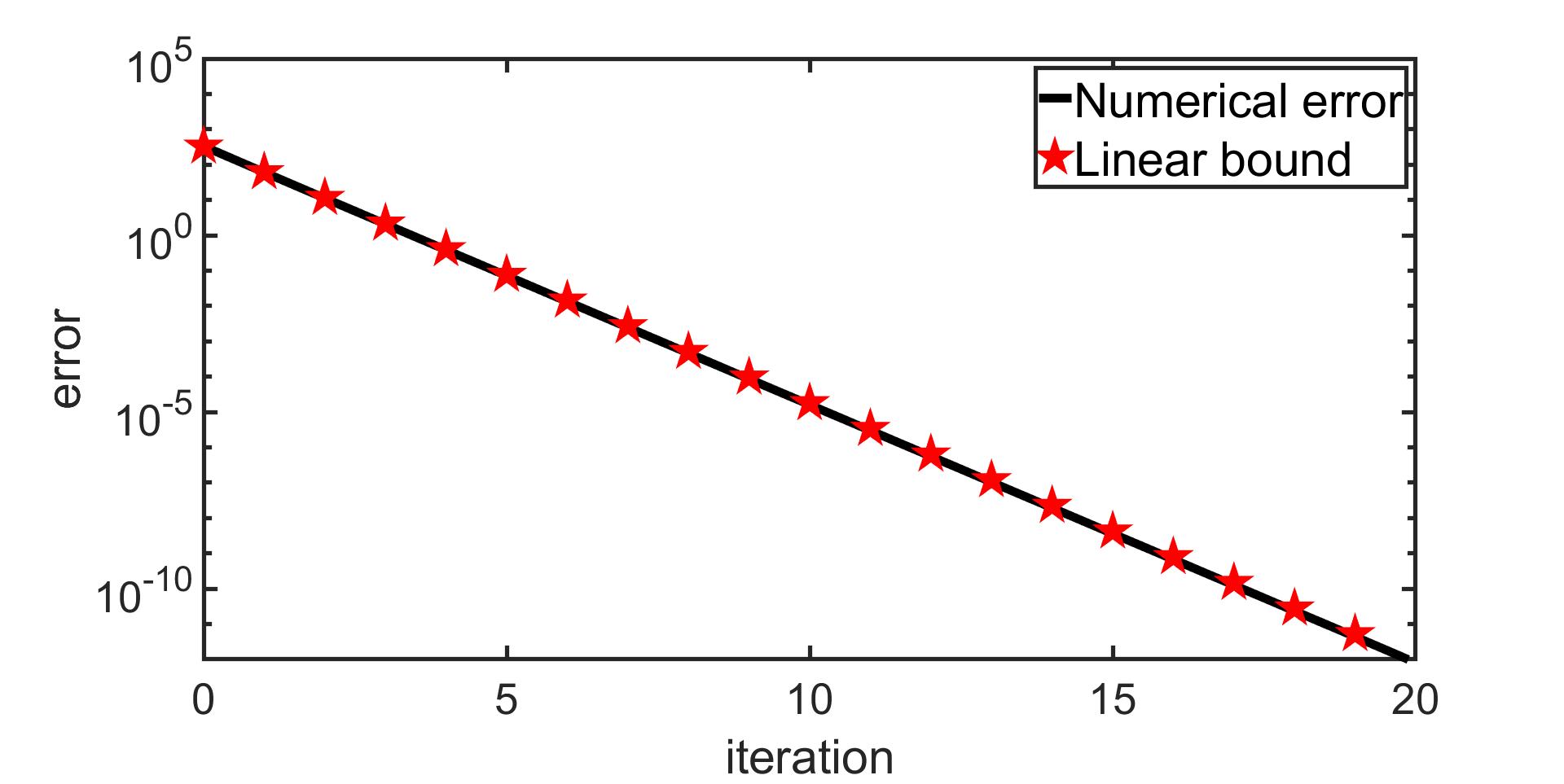} }}
    \subfloat{{\includegraphics[width=5cm,height=3cm]{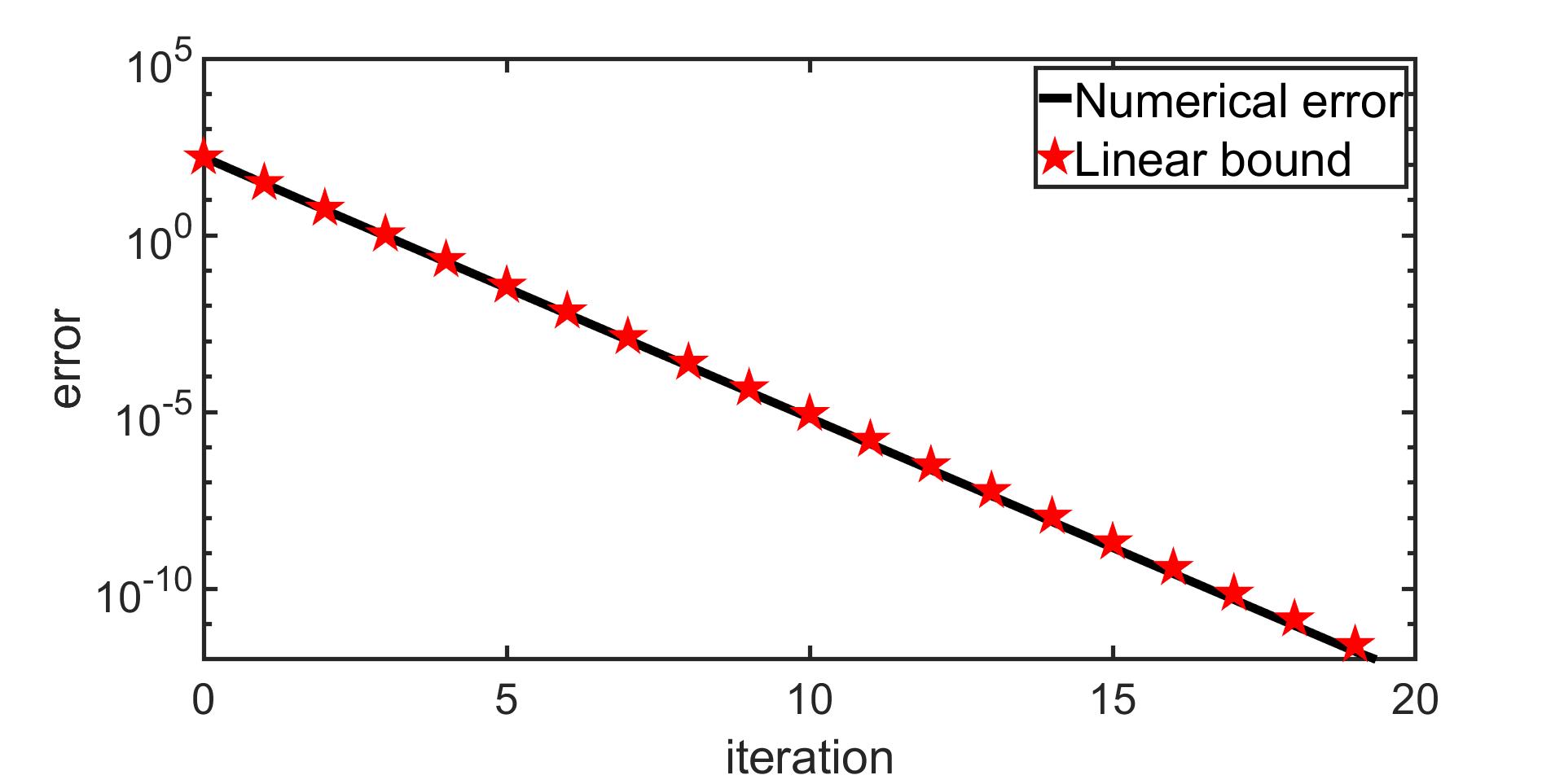} }}
    \caption{Comparison of the numerical error and the theoretical error estimates of NN for the mesh size $h_x=1/64, 1/128, 1/256$ with $64, 32, 16$ subdomain of equal width respectively from left to right for $\delta_t=10^{-3}, h_y=1/32$ and $\theta=1/4$.}
    \label{NNlongtimeerror2d_comp}
\end{figure}

\section{Conclusions}

We studied the Dirichlet-Neumann  and Neumann-Neumann method for the CH equation for two as well as multiple subdomain decomposition. We proved convergence estimates for the case of one dimensional DN and NN. We also extended our analysis to the two dimensional CH equation case using Fourier techniques, and obtained convergence estimates for DN and NN. Using numerical experiments we showed that a proper choice of relaxation parameter gives finite step convergence of the proposed algorithms. We have also given numerical study of DN and NN for the CH equation with various parameters.
\section*{Acknowledgement} The first author would like to thank the CSIR India for the research grant and IIT Bhubaneswar for providing nice research environment.
\appendix
\section*{Appendix A}
In this extended section we provide the detail expressions of the elements of the matrix $\mathbb{T}$, defined by \eqref{Iterationmatrix} in Section \ref{Section3}. To find the error estimates in Theorem \ref{NNequalthm}, \ref{NNunequalthm}, \ref{NNequal2dthm} and \ref{NNunequal2dthm}, we have used these following expressions to get individual bounds. The elements are given as below.\\
$\alpha_1^1=1-\frac{\lambda_1}{\lambda}\Upsilon_{1,1}^1 +  \frac{\lambda_2}{\lambda}\Upsilon_{3,1}^1,
\alpha_1^2=\frac{\delta_t\lambda_1\lambda_2}{\lambda}\Upsilon_{1,1}^1 -  \frac{\delta_t\lambda_1\lambda_2}{\lambda}\Upsilon_{3,1}^1,
\alpha_1^3=\frac{\lambda_1}{\lambda}\Upsilon_{1,1}^2 - \frac{\lambda_2}{\lambda}\Upsilon_{3,1}^2,
\alpha_1^4=-\frac{\delta_t\lambda_1\lambda_2}{\lambda}\Upsilon_{1,1}^2 +  \frac{\delta_t\lambda_1\lambda_2}{\lambda}\Upsilon_{3,1}^2,
\alpha_1^5=-\frac{\lambda_1}{\lambda}\Upsilon_{1,1}^3 +  \frac{\lambda_2}{\lambda}\Upsilon_{3,1}^3,
\alpha_1^6=\frac{\delta_t\lambda_1\lambda_2}{\lambda}\Upsilon_{1,1}^3 -  \frac{\delta_t\lambda_1\lambda_2}{\lambda}\Upsilon_{3,1}^3,
$\\
$\beta_1^1=-\frac{1}{\delta_t\lambda}\Upsilon_{1,1}^1 +  \frac{1}{\delta_t\lambda}\Upsilon_{3,1}^1,
\beta_1^2=1 + \frac{\lambda_2}{\lambda}\Upsilon_{1,1}^1 -  \frac{\lambda_1}{\lambda}\Upsilon_{3,1}^1,
\beta_1^3=\frac{1}{\delta_t\lambda}\Upsilon_{1,1}^2 -  \frac{1}{\delta_t\lambda}\Upsilon_{3,1}^2,
\beta_1^4=-\frac{\lambda_2}{\lambda}\Upsilon_{1,1}^2 +  \frac{\lambda_1}{\lambda}\Upsilon_{3,1}^2,
\beta_1^5=-\frac{1}{\delta_t\lambda}\Upsilon_{1,1}^3 +  \frac{1}{\delta_t\lambda}\Upsilon_{3,1}^3,
\beta_1^6=\frac{\lambda_2}{\lambda}\Upsilon_{1,1}^3 -  \frac{\lambda_1}{\lambda}\Upsilon_{3,1}^3,
$\\
$\alpha_2^1=\frac{\lambda_1}{\lambda}\Upsilon_{1,2}^1 -  \frac{\lambda_2}{\lambda}\Upsilon_{3,2}^1,
\alpha_2^2=-\frac{\delta_t\lambda_1\lambda_2}{\lambda}\Upsilon_{1,2}^1 +  \frac{\delta_t\lambda_1\lambda_2}{\lambda}\Upsilon_{3,2}^1,
\alpha_2^3=1 -\frac{\lambda_1}{\lambda}\Upsilon_{1,2}^2 + \frac{\lambda_2}{\lambda}\Upsilon_{3,2}^2,
\alpha_2^4=-\frac{\delta_t\lambda_1\lambda_2}{\lambda}\Upsilon_{1,2}^2 +  \frac{\delta_t\lambda_1\lambda_2}{\lambda}\Upsilon_{3,2}^2,
\alpha_2^5=\frac{\lambda_1}{\lambda}\Upsilon_{1,2}^3 - \frac{\lambda_2}{\lambda}\Upsilon_{3,2}^3,
\alpha_2^6=-\frac{\delta_t\lambda_1\lambda_2}{\lambda}\Upsilon_{1,2}^3 +  \frac{\delta_t\lambda_1\lambda_2}{\lambda}\Upsilon_{3,2}^3,
\alpha_2^7=-\frac{\lambda_1}{\lambda}\Upsilon_{1,2}^4 +  \frac{\lambda_2}{\lambda}\Upsilon_{3,2}^4,
\alpha_2^8=\frac{\delta_t\lambda_1\lambda_2}{\lambda}\Upsilon_{1,2}^4 -  \frac{\delta_t\lambda_1\lambda_2}{\lambda}\Upsilon_{3,2}^4,
$\\
$\beta_2^1=\frac{1}{\delta_t\lambda}\Upsilon_{1,2}^1 -  \frac{1}{\delta_t\lambda}\Upsilon_{3,2}^1,
\beta_2^2=-\frac{\lambda_2}{\lambda}\Upsilon_{1,2}^1 +  \frac{\lambda_1}{\lambda}\Upsilon_{3,2}^1,
\beta_2^3= -\frac{1}{\delta_t\lambda}\Upsilon_{1,2}^2 + \frac{1}{\delta_t\lambda}\Upsilon_{3,2}^2,
\beta_2^4=1-\frac{\lambda_2}{\lambda}\Upsilon_{1,2}^2 +  \frac{\lambda_1}{\lambda}\Upsilon_{3,2}^2,
\beta_2^5=\frac{1}{\delta_t\lambda}\Upsilon_{1,2}^3 - \frac{1}{\delta_t\lambda}\Upsilon_{3,2}^3,
\beta_2^6=-\frac{\lambda_2}{\lambda}\Upsilon_{1,2}^3 +  \frac{\lambda_1}{\lambda}\Upsilon_{3,2}^3,
\beta_2^7=-\frac{1}{\delta_t\lambda}\Upsilon_{1,2}^4 +  \frac{1}{\delta_t\lambda}\Upsilon_{3,2}^4,
\beta_2^8=\frac{\lambda_2}{\lambda}\Upsilon_{1,2}^4 -  \frac{\lambda_1}{\lambda}\Upsilon_{3,2}^4,
$ and for $i = 3,\dots, N-3$ we have\\
$\alpha_i^1 = -\frac{\lambda_1}{\lambda}\Upsilon_{1,i}^1 +  \frac{\lambda_2}{\lambda}\Upsilon_{3,i}^1,
\alpha_i^2 = \frac{\delta_t\lambda_1\lambda_2}{\lambda}\Upsilon_{1,i}^1  - \frac{\delta_t\lambda_1\lambda_2}{\lambda}\Upsilon_{3,i}^1,
\alpha_i^3 = \frac{\lambda_1}{\lambda}\Upsilon_{1,i}^2 -  \frac{\lambda_2}{\lambda}\Upsilon_{3,i}^2,
\alpha_i^4 = -\frac{\delta_t\lambda_1\lambda_2}{\lambda}\Upsilon_{1,i}^2 +  \frac{\delta_t\lambda_1\lambda_2}{\lambda}\Upsilon_{3,i}^2,
\alpha_i^5 = 1 - \frac{\lambda_1}{\lambda}\Upsilon_{1,i}^3 + \frac{\lambda_2}{\lambda}\Upsilon_{3,i}^3,
\alpha_i^6 = \frac{\delta_t\lambda_1\lambda_2}{\lambda}\Upsilon_{1,i}^3 -  \frac{\delta_t\lambda_1\lambda_2}{\lambda}\Upsilon_{3,i}^3,
\alpha_i^7 = \frac{\lambda_1}{\lambda}\Upsilon_{1,i}^4 -  \frac{\lambda_2}{\lambda}\Upsilon_{3,i}^4,
\alpha_i^8 = -\frac{\delta_t\lambda_1\lambda_2}{\lambda}\Upsilon_{1,i}^4 +  \frac{\delta_t\lambda_1\lambda_2}{\lambda}\Upsilon_{3,i}^4,
\alpha_i^9 = -\frac{\lambda_1}{\lambda}\Upsilon_{1,i}^5 + \frac{\lambda_2}{\lambda}\Upsilon_{3,i}^5,
\alpha_i^{10} = \frac{\delta_t\lambda_1\lambda_2}{\lambda}\Upsilon_{1,i}^5 -  \frac{\delta_t\lambda_1\lambda_2}{\lambda}\Upsilon_{3,i}^5,
$\\
and 
$\beta_i^1 = -\frac{1}{\delta_t\lambda}\Upsilon_{1,i}^1 +  \frac{1}{\delta_t\lambda}\Upsilon_{3,i}^1,
\beta_i^2 = \frac{\lambda_2}{\lambda}\Upsilon_{1,i}^1  - \frac{\lambda_1}{\lambda}\Upsilon_{3,i}^1,
\beta_i^3 = \frac{1}{\delta_t\lambda}\Upsilon_{1,i}^2 -  \frac{1}{\delta_t\lambda}\Upsilon_{3,i}^2,
\beta_i^4 = -\frac{\lambda_2}{\lambda}\Upsilon_{1,i}^2 +  \frac{\lambda_2}{\lambda}\Upsilon_{3,i}^2,
\beta_i^5 = - \frac{1}{\delta_t\lambda}\Upsilon_{1,i}^3 + \frac{1}{\delta_t\lambda}\Upsilon_{3,i}^3,
\beta_i^6 = 1 + \frac{\lambda_2}{\lambda}\Upsilon_{1,i}^3 -  \frac{\lambda_1}{\lambda}\Upsilon_{3,i}^3,
\beta_i^7 = \frac{1}{\delta_t\lambda}\Upsilon_{1,i}^4 -  \frac{1}{\delta_t\lambda}\Upsilon_{3,i}^4,
\beta_i^8 = -\frac{\lambda_2}{\lambda}\Upsilon_{1,i}^4 +  \frac{\lambda_1}{\lambda}\Upsilon_{3,i}^4,
\beta_i^9 = -\frac{1}{\delta_t\lambda}\Upsilon_{1,i}^5 + \frac{1}{\delta_t\lambda}\Upsilon_{3,i}^5,
\beta_i^{10} = \frac{\lambda_2}{\lambda}\Upsilon_{1,i}^5 -  \frac{\lambda_1}{\lambda}\Upsilon_{3,i}^5,
$\\
and $\alpha_{N-2}^1=-\frac{\lambda_1}{\lambda}\Upsilon_{1,N-2}^1 + \frac{\lambda_2}{\lambda}\Upsilon_{3,N-2}^1,
\alpha_{N-2}^2=\frac{\delta_t\lambda_1\lambda_2}{\lambda}\Upsilon_{1,N-2}^1 -  \frac{\delta_t\lambda_1\lambda_2}{\lambda}\Upsilon_{3,N-2}^1,
\alpha_{N-2}^3=\frac{\lambda_1}{\lambda}\Upsilon_{1,N-2}^2 - \frac{\lambda_2}{\lambda}\Upsilon_{3,N-2}^2,
\alpha_{N-2}^4=-\frac{\delta_t\lambda_1\lambda_2}{\lambda}\Upsilon_{1,N-2}^2 +  \frac{\delta_t\lambda_1\lambda_2}{\lambda}\Upsilon_{3,N-2}^2,
\alpha_{N-2}^5=1-\frac{\lambda_1}{\lambda}\Upsilon_{1,N-2}^3 + \frac{\lambda_2}{\lambda}\Upsilon_{3,N-2}^3,
\alpha_{N-2}^6=\frac{\delta_t\lambda_1\lambda_2}{\lambda}\Upsilon_{1,N-2}^3 -  \frac{\delta_t\lambda_1\lambda_2}{\lambda}\Upsilon_{3,N-2}^3,
\alpha_{N-2}^7=\frac{\lambda_1}{\lambda}\Upsilon_{1,N-2}^4 - \frac{\lambda_2}{\lambda}\Upsilon_{3,N-2}^4,
\alpha_{N-2}^8=-\frac{\delta_t\lambda_1\lambda_2}{\lambda}\Upsilon_{1,N-2}^4 +  \frac{\delta_t\lambda_1\lambda_2}{\lambda}\Upsilon_{3,N-2}^4,
$\\
$\beta_{N-2}^1=-\frac{1}{\delta_t\lambda}\Upsilon_{1,N-2}^1 + \frac{1}{\delta_t\lambda}\Upsilon_{3,N-2}^1,
\beta_{N-2}^2=\frac{\lambda_2}{\lambda}\Upsilon_{1,N-2}^1 -  \frac{\lambda_1}{\lambda}\Upsilon_{3,N-2}^1,
\beta_{N-2}^3=\frac{1}{\delta_t\lambda}\Upsilon_{1,N-2}^2 - \frac{1}{\delta_t\lambda}\Upsilon_{3,N-2}^2,
\beta_{N-2}^4=-\frac{\lambda_2}{\lambda}\Upsilon_{1,N-2}^2 +  \frac{\lambda_1}{\lambda}\Upsilon_{3,N-2}^2,
\beta_{N-2}^5=-\frac{1}{\delta_t\lambda}\Upsilon_{1,N-2}^3 + \frac{1}{\delta_t\lambda}\Upsilon_{3,N-2}^3,
\beta_{N-2}^6=1 + \frac{\lambda_2}{\lambda}\Upsilon_{1,N-2}^3 -  \frac{\lambda_1}{\lambda}\Upsilon_{3,N-2}^3,
\beta_{N-2}^7=\frac{1}{\delta_t\lambda}\Upsilon_{1,N-2}^4 - \frac{1}{\delta_t\lambda}\Upsilon_{3,N-2}^4,
\beta_{N-2}^8=-\frac{\lambda_2}{\lambda}\Upsilon_{1,N-2}^4 +  \frac{\lambda_1}{\lambda}\Upsilon_{3,N-2}^4,
$\\
$\alpha_{N-1}^1=-\frac{\lambda_1}{\lambda}\Upsilon_{1,N-1}^1 + \frac{\lambda_2}{\lambda}\Upsilon_{3,N-1}^1,
\alpha_{N-1}^2=\frac{\delta_t\lambda_1\lambda_2}{\lambda}\Upsilon_{1,N-1}^1 -  \frac{\delta_t\lambda_1\lambda_2}{\lambda}\Upsilon_{3,N-1}^1,
\alpha_{N-1}^3=\frac{\lambda_1}{\lambda}\Upsilon_{1,N-1}^2 - \frac{\lambda_2}{\lambda}\Upsilon_{3,N-1}^2,
\alpha_{N-1}^4=-\frac{\delta_t\lambda_1\lambda_2}{\lambda}\Upsilon_{1,N-1}^2 +  \frac{\delta_t\lambda_1\lambda_2}{\lambda}\Upsilon_{3,N-1}^2,
\alpha_{N-1}^5=1-\frac{\lambda_1}{\lambda}\Upsilon_{1,N-1}^3 + \frac{\lambda_2}{\lambda}\Upsilon_{3,N-1}^3,
\alpha_{N-1}^6=\frac{\delta_t\lambda_1\lambda_2}{\lambda}\Upsilon_{1,N-1}^3 -  \frac{\delta_t\lambda_1\lambda_2}{\lambda}\Upsilon_{3,N-1}^3,
$\\
$\beta_{N-1}^1=-\frac{1}{\delta_t\lambda}\Upsilon_{1,N-1}^1 + \frac{1}{\delta_t\lambda}\Upsilon_{3,N-1}^1,
\beta_{N-1}^2=\frac{\lambda_2}{\lambda}\Upsilon_{1,N-1}^1 -  \frac{\lambda_1}{\lambda}\Upsilon_{3,N-1}^1,
\beta_{N-1}^3=\frac{1}{\delta_t\lambda}\Upsilon_{1,N-1}^2 - \frac{1}{\delta_t\lambda}\Upsilon_{3,N-1}^2,
\beta_{N-1}^4=-\frac{\lambda_2}{\lambda}\Upsilon_{1,N-1}^2 +  \frac{\lambda_1}{\lambda}\Upsilon_{3,N-1}^2,
\beta_{N-1}^5=-\frac{1}{\delta_t\lambda}\Upsilon_{1,N-1}^3 + \frac{1}{\delta_t\lambda}\Upsilon_{3,N-1}^3,
\beta_{N-1}^6=1 +\frac{\lambda_2}{\lambda}\Upsilon_{1,N-1}^3 -  \frac{\lambda_1}{\lambda}\Upsilon_{3,N-1}^3,
$\\
where for $j=1,3,$ and  for $i = 3,\dots, N-3$ we have 
$\Upsilon_{j,i}^1 = \frac{1}{\sigma_{j,i}\sigma_{j,i-1}},
\Upsilon_{j,i}^2 = 2\frac{\gamma_{j,i}}{\sigma_{j,i}^2} + 
\frac{\gamma_{j,i-1}}{\sigma_{j,i-1}\sigma_{j,i}} + 
 \frac{\gamma_{j,i+1}}{\sigma_{j,i}\sigma_{j,i+1}}, 
\Upsilon_{j,i}^3 = \frac{\gamma_{j,i}^2}{\sigma_{j,i}^2} + 
\frac{\gamma_{j,i+1}^2}{\sigma_{j,i+1}^2} + 
 2\frac{\gamma_{j,i}\gamma_{j,i+1}}{\sigma_{j,i}\sigma_{j,i+1}} + \frac{1}{\sigma_{j,i}^2} + \frac{1}{\sigma_{j,i+1}^2}, 
\Upsilon_{j,i}^4 = \frac{\gamma_{j,i}}{\sigma_{j,i}\sigma_{j,i+1}} + 
\frac{\gamma_{j,i+2}}{\sigma_{j,i+1}\sigma_{j,i+2}} + 
 2\frac{\gamma_{j,i+1}}{\sigma_{j,i+1}^2},
\Upsilon_{j,i}^5 = \frac{1}{\sigma_{j,i+1}\sigma_{j,i+2}}$,
\\
and we have $\Upsilon_{j,1}^1 = 1+ \frac{\gamma_{j,1}\gamma_{j,2}}{\sigma_{j,1}\sigma_{j,2}} + \frac{\sigma_{j,1}\gamma_{j,2}}{\gamma_{j,1}\sigma_{j,2}} + \frac{\gamma_{j,2}^2}{\sigma_{j,2}^2} + \frac{1}{\sigma_{j,2}^2},
\Upsilon_{j,1}^2=\frac{\gamma_{j,1}}{\sigma_{j,1}^2} + 2\frac{\gamma_{j,2}}{\sigma_{j,2}^2} + \frac{\gamma_{j,3}}{\sigma_{j,2}\sigma_{j,3}},
\Upsilon_{j,1}^3 = \frac{1}{\sigma_{j,2}\sigma_{j,3}},\\
\Upsilon_{j,2}^1 = \frac{\sigma_{j,1}}{\sigma_{j,2}\gamma_{j,1}} + 2\frac{\gamma_{j,2}}{\sigma_{j,2}^2} + \frac{\gamma_{j,3}}{\sigma_{j,2}\sigma_{j,3}},
\Upsilon_{j,2}^2 =  \frac{\gamma_{j,2}^2}{\sigma_{j,2}^2} + 2\frac{\gamma_{j,2}\gamma_{j,3}}{\sigma_{j,2}\sigma_{j,3}} + \frac{1}{\sigma_{j,2}^2} + \frac{1}{\sigma_{j,3}^2} + \frac{\gamma_{j,3}^2}{\sigma_{j,3}^2},\\
\Upsilon_{j,2}^3 = \frac{\gamma_{j,2}}{\sigma_{j,2}\sigma_{j,3}} + 2\frac{\gamma_{j,3}}{\sigma_{j,3}^2} + \frac{\gamma_{j,4}}{\sigma_{j,3}\sigma_{j,4}},
\Upsilon_{j,2}^4= \frac{1}{\sigma_{j,3}\sigma_{j,4}}, \text{and}\\
\Upsilon_{j,N-2}^1 = \frac{1}{\sigma_{j,N-3}\sigma_{j,N-2}},
\Upsilon_{j,N-2}^2 = 2\frac{\gamma_{j,N-2}}{\sigma_{j,N-2}^2} + \frac{\gamma_{j,N-3}}{\sigma_{j,N-3}\sigma_{j,N-2}} + \frac{\gamma_{j,N-1}}{\sigma_{j,N-2}\sigma_{j,N-1}},
\Upsilon_{j,N-2}^3 = \frac{\gamma_{j,N-2}^2}{\sigma_{j,N-2}^2} + 2\frac{\gamma_{j,N-2}\gamma_{j,N-1}}{\sigma_{j,N-2}\sigma_{j,N-1}} + \frac{1}{\sigma_{j,N-2}^2} +\frac{1}{\sigma_{j,N-1}^2} + \frac{\gamma_{j,N-1}^2}{\sigma_{j,N-1}^2},
\Upsilon_{j,N-2}^4 = 2\frac{\gamma_{j,N-1}}{\sigma_{j,N-1}^2} + \frac{\gamma_{j,N-2}}{\sigma_{j,N-2}\sigma_{j,N-1}} + \frac{\sigma_{j,N}}{\gamma_{j,N}\sigma_{j,N-1}},\\
\Upsilon_{j,N-1}^1 = \frac{1}{\sigma_{j,N-2}\sigma_{j,N-1}},
\Upsilon_{j,N-1}^2 = 2\frac{\gamma_{j,N-1}}{\sigma_{j,N-1}^2} + \frac{\gamma_{j,N-2}}{\sigma_{j,N-2}\sigma_{j,N-1}} + \frac{\gamma_{j,N}}{\sigma_{j,N-1}\sigma_{j,N}},
\Upsilon_{j,N-1}^3= 1+ \frac{\gamma_{j,N-1}^2}{\sigma_{j,N-1}^2} + \frac{\sigma_{j,N-1}\sigma_{j,N}}{\gamma_{j,N-1}\gamma_{j,N}} +\frac{1}{\sigma_{j,N-1}^2}+ \frac{\gamma_{j,N-1}\gamma_{j,N}}{\sigma_{j,N-1}\sigma_{j,N}},
$
for $j=1, 3$, and $\sigma_{1,i}:=\sinh(\xi_1d_{i}), \sigma_{3,i}:=\sinh(\xi_3d_{i}), \gamma_{1,i}:=\cosh(\xi_1d_{i}), \gamma_{3,i}:=\cosh(\xi_3d_{i})$ for $i=1\cdots N$.

\bibliographystyle{siam}
\bibliography{dn2nnm}


\end{document}